\documentclass[10 pt]{amsart}

[1995/06/01]

\usepackage{amssymb,amsmath}
\usepackage{graphicx}
\usepackage{mathdots}
\usepackage{amsbsy}
\usepackage{amscd}
\usepackage{amsthm}
\usepackage{mathrsfs}
\usepackage{verbatim}
\usepackage[colorlinks]{hyperref}
\usepackage{authblk}
\usepackage{fullpage}
\usepackage[english]{babel}
\usepackage[T1]{fontenc}
\usepackage[utf8]{inputenc}
\usepackage{enumitem}

\input xy
\usepackage[all]{xy}

\usepackage{url}
\usepackage[english]{babel}
\usepackage[T1]{fontenc}
\usepackage[utf8]{inputenc}
\usepackage{booktabs}
\usepackage{blkarray}
\usepackage[fleqn,tbtags]{mathtools}
\usepackage{graphics,graphicx}
\usepackage{pstricks,pst-node,pst-tree}

\usepackage{tikz}
\usetikzlibrary{backgrounds,fit,decorations.pathreplacing}
\usetikzlibrary{arrows,shapes,positioning}
\usetikzlibrary{calc,through,backgrounds,chains}
\usetikzlibrary{snakes,automata,backgrounds,petri}
\usetikzlibrary{arrows.meta}
\usetikzlibrary{shapes,decorations,arrows,calc,arrows.meta,fit,positioning}
\usepackage{calc}
\usetikzlibrary{decorations.markings}
\tikzstyle{vertex}=[circle, draw, inner sep=0pt, minimum size=6pt]

\theoremstyle{definition}
\newtheorem{Theorem}{Theorem}[section]
\newtheorem{Lemma}[Theorem]{Lemma}
\newtheorem{Example}[Theorem]{Example}
\newtheorem{Remark}[Theorem]{Remark}
\newtheorem{Definition}[Theorem]{Definition}
\newtheorem{Corollary}[Theorem]{Corollary}
\newtheorem{Proposition}[Theorem]{Proposition}
\newtheorem*{Theorem*}{Theorem}

\newcommand{\C}{\mathbb{C}}
\newcommand{\N}{\mathbb{N}}

\newcommand{\F}{\mathbb{F}}
\newcommand{\Z}{\mathbb{Z}}

\newcommand{\A}{\mathbb{A}}

\newcommand{\mat}{\text{Mat}}

\newcommand{\U}{\mathcal{U}}
\newcommand{\codim}{\text{codim }}

\newcommand{\mc}[1]{\mathcal{#1}} 

\definecolor {processblue}{cmyk}{0.96,0,0,0}

\begin{document}

\mathtoolsset{showonlyrefs}

\title{Stability of the center of the symplectic group rings over finite fields}

\author{\c{S}afak \"Ozden\footnote{Address: Weigandufer 7, 12045, Berlin. email: sozden@tulane.edu.}}
\maketitle

{\centering \c{S}afak \"Ozden\footnote{Address: Weigandufer 7, 12045, Berlin. email: sozden@tulane.edu.}\par }

\begin{abstract}We investigate the structure constants of the center $\mc{H}_n$ of the group algebra
	$\Z[Sp_{n}(q)]$ over the finite field with $q$ elements. The reflection length on the group $GL_{2n}(q)$ induces a
	filtration on the algebras $\mc{H}_n$. We prove that the structure constants of the
	associated filtered algebra  $\mc{S}_n$ are independent of $n$. As a technical tool in
	the proof, we determine the growth of the centralizers under the embedding
	$Sp_m(q)\subset Sp_{m+l}(q)$ and we show that the index of $C_{Sp_m}(g)\cap C_{Sp_m}(h)$ in $C_{Sp_{m+l}}(g)\cap C_{Sp_{m+l}}(h)$ is equal to
	$q^{2ld}|Sp_{r+l}(q)||Sp_{r}(q)|^{-1}$ for some $d$ and $r$ which are uniquely determined by the conjugacy classes of $g, h$ and $gh$ in $GL_{2m}(q).$
	
\end{abstract}

\tableofcontents

\section{Introduction}\label{C:1}

Let $G_1\subset \cdots \subset G_n \subset \cdots$ be a family of finite groups and let $\mc{H}_n$
denote the center of the group algebra $\Z[G_n]$ for $n\in \N$. The set of conjugacy
classes of $G_n$ is denoted by $\widehat{G_n}$. For $\lambda\in \widehat{G_n}$, the
class sum $\sum_{g\in \lambda}g\in \Z[G_n]$ is denoted by
$K_{\lambda} $. The class sums $K_{\lambda}$, $\lambda\in \widehat{G_n}$, form a basis
for $\mc{H}_n$. We introduce the term \textbf{saturated family} to refer to the families $(G_n)_{n\in \N}$ for which non-conjugate elements of $G_n$ remain non-conjugate in
$G_{n+1}$. Assume that the family $(G_n)_{n\in \N}$ is saturated. The embedding $G_n\hookrightarrow G_{n+1}$ induces an injection
$\widehat{G_n}\hookrightarrow\widehat{G_{n+1}}$. 
Let $G$ be the union of $G_n$. For each $\lambda\in \widehat{G}$,
the intersection $\lambda(n):=\lambda\cap G_n$ is either empty or an element of
$\widehat{G_n}$, and every element of $\widehat{G_n}$ can be represented as such an intersection. Given three elements $\lambda$, $\mu$, $\eta$ in $\widehat{G}$
there is an $m_{\lambda,\mu,\eta}=m\in \N$ such that for all $n\geq m$, each of $\lambda(n)$, $\mu(n)$, $\eta(n)$
are nonempty. So, for fixed $\lambda,\mu,\eta\in \widehat{G}$ and $n\geq m_{\lambda,\mu,\eta}$, the product $K_{\lambda(n)}\cdot K_{\mu(n)}$ can be written
as
\begin{eqnarray}
K_{\lambda(n)}\cdot K_{\mu(n)} & = & \sum_{\eta\in \widehat{G}} c_{\lambda,\mu}^{\eta}(n)K_{\eta(n)} = \sum_{\eta\in \widehat{G}\atop \eta(n)\neq \emptyset} c_{\lambda,\mu}^{\eta}(n)K_{\zeta(n)}\nonumber
\end{eqnarray} where $c_{\lambda,\mu}^{\eta}(n)\in \N$, in which $c_{\lambda,\mu}^{\eta}(n)$ is
uniquely determined as $K_{\eta(n)}\neq 0$. For a fixed $n$, the collection of
$c_{\lambda,\mu}^{\eta}(n)$, where $\lambda(n),\mu(n),\eta(n)$ runs over
$\widehat{G_n}$, are called the \textit{structure constants} of the algebra $\mc{H}_n$.
We will call the functions $n\longmapsto c_{\lambda,\mu}^{\eta}(n)$ the
\textit{structure functions} of the family. If $||\cdot||_n$ is an $\N$ valued function
on $G_n$ which is constant on conjugacy classes then $||\cdot ||$ induces a function on
$\widehat{G_n}$ as well. In this case, if the function is also sub-additive, in the
sense that $||gh||_n\leq ||g||_n+||h||_n$, and if $||\cdot||_n$ is invariant under the
embedding $G_n\subset G_{n+1}$ then the algebra $\mc{H}_n$ induces a filtered algebra
$\mc{S}_n$ with the same basis elements, where the multiplication is defined as
\begin{equation}\label{structure functions of the filtered algebra}
K_{\lambda(n)}\cdot K_{\mu(n)}=\sum_{\eta\in \widehat{G}\atop ||\eta||=||\lambda||+||\mu||} c_{\lambda,\mu}^{\eta}(n)K_{\eta(n)}.
\end{equation}
When the structure functions defined via Eq.\eqref{structure functions of the filtered algebra} of the filtered algebra of a family
$(G_n)_{n\in \N}$ are independent of $n$, following Wan and Wang \cite{WW18}, we will
say that the family satisfies the \textit{stability property}.

For $n\in \N$, let $S_n$ denote the symmetric group of the set $\{1,2,\cdots,n\}$.
Farahat and Higman considered the family $(S_n)_{n\in \N}$ in \cite{FH} and proved that with respect to the filtration induced by reflection length, the structure constants
$c_{\lambda,\mu}^{\eta}(n)$ of the induced filtered algebra structure on $Z(\Z[S_n])$
are independent of $n$. They used this result to answer the question of determining whether two representations of $S_n$ belong to the same $p$-block. In \cite{W04}, as a
generalization of the case considered by Farahat and Higman, Wang proved that
the families given by the wreath product $(H\wr S_n)_{n\in \N}$, where $H$ is a finite group, satisfy the stability property. In the case studied by Wang, when the group $H$ is a finite subgroup of $SL_2(\C)$, the associated
graded algebra of $\mc{H}_n$ is isomorphic to the cohomology ring of Hilbert scheme of
$n$-points on the minimal resolution of $\C^2/H$. Recently, in \cite{WW18}, Wan and Wang considered the
family $(GL_n(q))_{n\in \N}$ and proved that this family also satisfies the stability property
with respect to the filtration induced by reflection length. The result of Wan and Wang was also obtained by P.-L. M\'eliot in \cite{PM14}.

In this paper we study the family $(Sp_n(q))_{n\in \N}$ of symplectic groups over the
finite field with $q$ elements. We introduce the set of modified symplectic partition
valued functions  and prove that these functions parameterize the conjugacy classes of
$\cup_{n\in\N}Sp_n(q)$ and that the family $(Sp_n(q))_{n\in \N}$ is saturated. We consider
the filtration induced from the reflection length in $GL_{2n}(q)$. The set of reflections generate $GL_{2n}(q)$
and for $U\in GL_{2n}(q)$, the minimum value of $l$ where $U$ can be written as a product of
$l$ many reflections is called the reflection length of $U$ and denoted by $rl(U)$. It
is constant on conjugacy classes, sub-additive function and stable under the embedding
$Sp_n(q)\subset Sp_{n+1}(q)$. Therefore, for a stabilized symplectic partition valued function $\pmb{\lambda}$, one can talk about $||\pmb{\lambda}||$. With this setting, the main result is following.

\begin{Theorem*}[Stability property][Theorem \ref{stability theorem for symplectic group}]
	Let $\pmb{\lambda},$ $\pmb{\mu}$, $\pmb{\eta}$ be three stabilized symplectic partition valued functions and assume that $||\pmb{\eta}||=||\pmb{\lambda}||+||\pmb{\mu}||$. Then $c_{\pmb{\lambda},\pmb{\mu}}^{\pmb{\eta}}(n)$ is a non-negative integer independent of $n$.
\end{Theorem*}

We observe that all the \textit{stability properties} proved so far rely on two fundamental facts:
A certain action admits finitely many orbits and certain splitting of the
centralizers. More precisely, in each case one first proves that a pair $(g,h)\in
G\times G$ can be mapped to $G_m\times G_m$ by simultaneous conjugation, where $m$ is a
fixed integer completely
determined by the conjugacy classes of $g,h$ and $gh$. To prove such a result, one needs to find a so-called \textit{normal form}, a formulation introduced in \cite{WW18}. We will refer to the existence of normal forms as \textit{normal form theorems}. Secondly, one shows that the
centralizer of $g\in G_m$ \textit{"splits"} in the centralizer of $g$ in $G_n$ for $n\geq m$, which we will call the \textit{growth of centralizers}.

In the case of symplectic groups, finding a normal form can be
derived from the case of general linear groups. However, the investigation of the growth of centralizers in the case of symplectic
group is more complicated than the case of general linear groups, as it consists of
non-linear equations. To overcome this obstacle, we first introduce a concept called primitive symplectic centralizer, and using suitable rational forms we investigate the elements in the centralizers of a unipotent element and then
invoke the concept of primitive symplectic centralizer to reduce the question of centralizer growth
to a linear question. Once the degree 2 problem is reduced to a linear problem the problem
becomes much more manageable. The simplified versions of these results (Proposition \ref{centralizer growth in sp} and Proposition \ref{centralizer intersection growth in sp}) are packed into the following:

\begin{Theorem*}[Growth of centralizers]\label{growth of centralizer in symplectic group- main statement}
	Let $U=U_1U_2\in Sp_m$ and $d_{\overline{\pmb{\eta}}}$ be the dimension of the fixed space $V^U:=\ker (U-I)$ of $U$, c.f. Eq, \eqref{residual and fixed space}. Assume that there is no identity block in the Jordan form of $U$. Then for $m\leq n$ the following equalities hold:
	\begin{equation}\label{growth of centralizer in symplectic group- main equation}
	|C_{Sp_n(q)}(U)|=|C_{Sp_m(q)}(U)|\cdot |Sp_{n-m}(q)|\cdot q^{2(n-m)d_{\pmb{\eta}}}
	. \end{equation}
	If $rl(U_1)+rl(U_2)=rl(U)$, where $rl$ denotes the reflection length, then
	\begin{equation}\label{growth of intersection of centralizers in symplectic group- main equation}
	|C_{Sp_{n}}(U_1)\cap C_{Sp_{n}}(U_2)|=|C_{Sp_{m}}(U_1)\cap C_{Sp_{m}}(U_2)|\cdot |Sp_{n-m}(q)|\cdot q^{2(n-m)d_{\overline{\pmb{\eta}}}}.
	\end{equation}
\end{Theorem*}

It is worth to mention a generalized approach to the center of the integral group rings. Namely, in terms of Gel'fand pairs. Recall that a pair of finite groups $H\subseteq G$ is called a Gel'fand pair, if the convolution algebra  \begin{equation}
\mathcal{H}(G,H)=\{f:G\longrightarrow\mathbb{Z}|f(hgh'')=f(g),\forall h,h'\in H,\forall g\in G\}
\end{equation}  of the $\mathbb{Z}$-valued functions on $G$ that are invariant on the $G$-double cosets of $G$ is \textit{commutative}. Let $G$ be a finite group. If one considers the pair $(G,diag(G))$ where  $diag(G)=\{(g,g)\in G\times G| g\in G \}$ then there is a $\mathbb{Z}$-algebra isomorphism 
\begin{equation}\label{gelfand connection}
\mathcal{H}(G)\simeq \mathcal{H}(G\times G,diag(G)).
\end{equation}
For details on this isomorphism, see \cite[Proposition 1.5.22]{CTS10}. For an extensive study on Gel'fand pairs related to symmetric groups see \cite{CTS00}. Relying on this observation, one can generalize the concepts discussed earlier. 

First notice that, the analogous basis elements in this case are given by the characteristic functions on $H$-double cosets of $G$. More precisely, if $\Theta$ denotes the set of $H$-double cosets of $G$, the elements
\[
K_{\lambda}=\sum_{g\in\lambda}g
\] is an element of $\mathcal{H}(G,H)$ and the set $\{K_{\lambda}|\lambda\in \Theta\}$ constitute a basis for $\mathcal{H}(G,H)$. This means, if $\lambda\mu\in \Theta$ are fixed, then for all $\eta\in \Theta$, there exists unique $c_{\lambda,\mu}^{\eta}\geq 0$ such that
\[
K_{\lambda}\cdot K_{\mu}=\sum_{\eta\in\Theta}
c_{\lambda,\mu}^{\eta}\cdot K_{\eta}.
\]
Consider a sequence of groups $\{G_n\}_{n\in\mathbb{N}}$ and a family of subgroups $\{H_n\leq G_n\}_{n\in\mathbb{N}}$.
Let $G$ (resp. $H$) be the direct limit of $G_n$'s (resp. $H_n$'s). Then $H\leq G$. Let $\mathcal{H}=\mathcal{H}(G,H)$ (resp. $\mathcal{H}_n=\mathcal{H}(G_n,H_n)$) be the Hecke algebra corresponding to $(G,H)$ (resp. $(G_n,H_n)$). Each double coset of $H_n$ in $G_n$ 
extends to a unique $H_{n+1}$ double coset in $G_{n+1}$. If every distinct $H_n$ double cosets in $G_n$ remains distinct in $G_{n+1}$, then we say that the family $(G_n,H_n)$ is \textbf{saturated}.

Let $\Theta$ (resp $\Theta_n$) denote the set of double cosets of $H$ (resp. $H_n$) in $G$ (resp. $G_n$) and $\Theta(n):=\{\theta(n):=\theta\cap G_n:\theta\in\Theta\}$. If  $H$-double cosets of $G$ is $H_n$-saturated than $\Theta(n)=\Theta_n$. For $n\geq 0$, one can then define $K_{\lambda}(n)$ for $\lambda\in \Theta$ in a similar way and introduce the structural functions $c_{\lambda,\mu}^{\eta}(n)$ satisfying
\[
K_{\lambda}(n)\cdot K_{\mu}(n)=\sum_{\eta\in\Theta}
c_{\lambda,\mu}^{\eta}(n)\cdot K_{\eta}(n).
\]
In this setting, study of the structure constants of saturated families of pairs makes sense. The saturated family $(S_{2n},B_n)$ and its structure constants are investigated in the papers \cite{AC}, \cite{saf17} and \cite{tout14}. It turns out that, this family also satisfy the stability property, i.e. the structural functions corresponding to the top coefficients with respect to a suitable filtration are independent of $n$. For a detailed study of the pair $(S_n,B_n)$ see \cite{CTS00}.

Finally, we recall the Frobenious formula which justifies the attention on the structure constants of the center of the integral group rings. The proof of the following theorem can be found in the appendix of \cite{LZ}:

\begin{Theorem*}[Frobenious formula]
	Let $\lambda,\mu,\eta$ be three conjugacy classes of a finite group $G$ and let $\eta^{-1}$ be the conjugacy which consists of elements $x\in G$ where $x^{-1}\in \eta$. Then
	\[
	c_{\lambda,\mu}^{\eta}(G)=\frac{|\lambda||\mu||\eta^{-1}|}{|G|}\sum_{\chi}\frac{\chi(\lambda)\chi(\mu)\chi(\eta^{-1})}{\chi(1)}
	\]where the sum taken over irreducible characters of $G$.
\end{Theorem*}

For an analogue of the Frobenious formula in the setting of Gel'fand pairs, see \cite{tout17}

\textbf{Acknowladgements.} I would like to express my deepest gratitude to Professor Weiqiang Wang. 
He has shown great generosity with both his time and expertise throughout the project. 
I also thank Professor Jinkui Wan for her invaluable comments and feedback that I have benefited a lot while finalizing the work.  
My special thanks are to Professors Brian Conrad, Karl Hofmann and Kazim Büyükboduk for their great support.

\section{Notations and preliminaries}\label{C:2}

In this chapter, we first introduce the notion of saturated family of groups $(G_n)_{n\in \N}$ and then present  a systematic way of calculating structure constants in the center. In the subsequent sections, we introduce a ring, so called Farahat-Higman ring and summarize the work of Farahat and Higman.

\subsection{Center of the group rings and uniformly saturated families of groups}\label{Center of the groups rings and uniformly saturated families of groups}

Let $G$ be a group. Two elements $g_1,g_2\in G$ are said to be conjugate or similar, if there exists $h\in G$ such that $h^{-1}g_1h=g_2$. The similarity relation is an equivalence relation and it is denoted by $\sim_G$. The conjugacy class of an element $g\in G$ is denoted by $g^G$ and the set of conjugacy classes of $G$ is denoted by $\widehat{G}$. If $g\in G$ and $\lambda\in \widehat{G}$ representing the conjugacy class of $g$, then we say that \textbf{type} of $g$ is $\lambda$. The center of the group algebra $\Z[G]$ is denoted by $\mc{H}(G)$. If $\lambda\in \widehat{G}$, then the class sum

\begin{equation}\label{class sum}
K_{\lambda}=\sum_{g\in\lambda}g
\end{equation}
is an element of $\mc{H}(G)$. As $\lambda$ ranges over $\widehat{G}$, the elements $K_{\lambda}$ form a basis of $\mc{H}(G)$ and the non-negative integers $c_{\lambda\mu}^{\eta}$ defined via the equation
\begin{equation}\label{structure constants}
K_{\lambda}\cdot K_{\mu}=\sum_{\eta\in \widehat{G}}c_{\lambda\mu}^{\eta}K_{\eta}
\end{equation}
are called the \textbf{structure constants} of $\mc{H}(G)$. For $A,B,C\subset G$ the \textbf{fiber} of $C$ in $A\times B$ is denoted by $V(A\times B:C)$ and defined by
\begin{equation}\label{fiber of a set}
V(A\times B:C)=\{(a,b)\in A\times B:ab\in C\}.
\end{equation}
\begin{Lemma}\label{structure constants via fibers}
	Let $\lambda,\mu,\eta\in\widehat{G}$ and $z\in \eta$. Then
	\begin{equation}
	c_{\lambda,\mu}^{\eta} = |V(\lambda\times \mu:\{z\})|= |V(\lambda\times \mu:\eta)||\eta|^{-1}.
	\end{equation}
	In particular, $c_{\lambda,\mu}^{\eta}\in \N$.
\end{Lemma}
\begin{proof}
	The first equality follows from the definition of the structure constants and the basis elements
	$K_{\eta}$. In fact, the coefficient $c_{\lambda,\mu}^{\eta}$ is equal to the
	coefficient of $z$ in the expansion of the product
	\[
	(\sum_{x\in\lambda}x)\cdot (\sum_{y\in\mu}y)
	\]
	and it is equal to the number couples $(x,y)\in \lambda\times \mu$
	which satisfy $xy=z$. Therefore $c_{\lambda,\mu}^{\eta}$ equals to the
	number of elements in $V(\alpha\times \beta:\{z\})$, which proves the first
	equality. The second equality follows from the first one and the set theoretic equality
	\[
	V(A\times B:C_1\sqcup C_2)=V(A\times B:C_1) \sqcup V(A\times B: C_2)
	\]
	for $C_1\cap C_2=\emptyset$.
\end{proof}

Let $G_1\subset G_2 \cdots \subset G_n\subset \cdots$ be an ascending chain of finite
groups and let $G$ be the union of $G_n$, for $n\in \N$. If $x\in G_m$ and $m\leq
n$, then the image of $x$ in $G_n$ is denoted by
$x^{\uparrow n}$. The family $(G_n)_{n\in \N}$ is said to be \textbf{saturated} if
for all $x_1,x_2\in G_m$ and for all $n\geq m$.
\begin{equation}\label{saturated definition}
x_1\sim_{G_m} x_2 \Leftrightarrow {x_1}^{\uparrow n}\sim_{G_n} {x_2}^{\uparrow n}
\end{equation}
In other words, the family is said to be saturated if for all $m\in \N$, two non-conjugate elements in $G_m$ remains non-conjugate in $G$. For a fixed saturated family $(G_n)_{n\in \N}$, the algebra $\mc{H}(G_n)$ is simply denoted by $\mc{H}_n$ henceforth.
\begin{Lemma}
	Let $(G_n)_{n\in \N}$ be a saturated family of finite groups. The association $g^{G_m}\longmapsto (g^{\uparrow n})^{G_n}$ defines an injection $\widehat{G_m}\longrightarrow \widehat{G_n}$ for all $m\leq n$, thus defines a direct system. Moreover
	\begin{equation}\label{limits conjugacy classes}
	\widehat{G}=\lim_{\longrightarrow \atop n\in \N} \widehat{G_n}.
	\end{equation}
\end{Lemma}
\begin{proof}
	The fact that
	$\widehat{G_n}\hookrightarrow\widehat{
		G_{n+1}}$ follows directly from
	(\ref{saturated definition}). As each
	conjugacy class of $G$ is uniquely
	determined by an element $x\in G$ and
	each such element is contained in
	$G_n$ for some $n\in \N$ the natural map
	\begin{equation}\label{limits conjugacy classes}
	\lim_{\longrightarrow \atop n\in \N} \widehat{G_n}\longrightarrow\widehat{G}
	\end{equation}is onto. As this map is induced by the limit of injective maps, it is also injective. Hence it is bijective.
\end{proof}

Now we introduce some abstract notation which will have concrete
meanings in each case that will be covered in the later sections. Fix a saturated family $(G_n)_{n\in \N}$. If
$\lambda\in \widehat{G_m}$ then the image of $\lambda$ in
$\widehat{G}$ is denoted by $ \overset{\circ}{\lambda}$. The element $
\overset{\circ}{\lambda}$ is called the \textbf{modification} of $\lambda$, and
elements of
$\widehat{G}$ are called a \textbf{modified types}. Let $\lambda\in
\widehat{G}$ be a fixed modified type. The intersection
$\lambda(n):=\lambda\cap G_n$, if non-empty, determines a conjugacy
class in $\widehat{G_n}$. The minimal integer $l_{\lambda}$ for which
$\lambda(l_{\lambda})\neq \emptyset$ is called the \textbf{level} of
$\lambda$. If $n\geq l_{\lambda}$ then the equality
\[
\overset{\circ}{\lambda(n)}=\lambda
\] is a tautological consequence of the
definitions. Let $\lambda\in \widehat{G}$ be a modified type. The
element $\lambda(l)\in \widehat{G_l}$, where $l$ is the level of
$\lambda$, is called the \textbf{completion} of $\lambda$ and denoted
by $\overline{\lambda}$. For $n\geq l_{\lambda}$, the induced element
$\lambda(n)$ is denoted by $\lambda^{\uparrow n}$ and called the
$n$-th \textbf{completion}. It is clear that $\lambda^{\uparrow n}$ is
equal to the image of $\overline{\lambda}$ in $\widehat{G_n}$ and they are both equal to $\lambda(n)$. The
corresponding basis element (cf. Eq.\eqref{class sum}) of
$\mc{H}_n$  determined by $\lambda(n)$ is denoted by $K_{\lambda}(n)\in \mc{H}_n$ instead of
$K_{\lambda(n)}$.

Let $\lambda,\mu,\eta\in \widehat{G}$ be three modified types and let $m=\min
\{l_{\lambda},l_{\mu},l_{\eta}\}.$ Then for all $n\geq m$, all the three intersections
$\lambda(n)$, $\mu(n)$, $\eta(n)$ are non-empty and determine elements of $\widehat{G_n}$. This means, one can form the multiplication $K_{\lambda}(n)\cdot K_{\mu}(n)$ in $\mc{H}_n$ for all $n\geq m$ and
consider the coefficient $c_{\lambda,\mu}^{\eta}(n)$ of $K_{\eta}(n)\in\mc{H}_n$. We
will call the resulting functions
\begin{equation}
n\longmapsto c_{\lambda,\mu}^{\eta}(n)
\end{equation} the \textbf{structural functions} of $G$.

\begin{Remark}
	Using Lemma \ref{structure constants via fibers}, we know that
	\[
	c_{\lambda,\mu}^{\eta}(n) = |V(\lambda(n)\times \mu(n):\{z\})|
	\]
	where $z\in \eta(m)$. But $V(\lambda(n_1)\times \mu(n_1):\{z\})\subset V(\lambda(n_2)\times \mu(n_2):\{z\})$ for $n_1\leq n_2$. From this, it follows that the structural functions are monotone increasing.
\end{Remark}

Now we present a certain way of calculating the structural constants, which was introduced by Farahat and Higman in \cite{FH} in the context of symmetric groups.
Let $G$ be a fixed group. $G$ acts on $G\times G$ with the
two-fold simultaneous conjugation: For
$h\in G$ and $(x,y)\in G\times G$ we set
$(x,y)^{h}:=(hxh^{-1},hyh^{-1})$.
\begin{Remark}
	Notice that $(xy)^{h}$ is equal to $x^{h}y^{h}$,
	which means the fiber $V=V(\lambda\times\mu:\eta)$ is closed under two-fold conjugation, where $\lambda$, $\mu,$ and $\eta$ stand for conjugacy classes. In fact, let $(x,y)\in V$, i.e. the conjugacy class of $xy$ is $\eta$. Then $(x,y)^h=(hxh^{-1},hyh^{-1})$ and $hxh^{-1}hyh^{-1}=hxyh^{-1}\sim xy$, thus $hxyh^{-1}\in \eta$.
\end{Remark}
A saturated
family of groups $(G_n)_{n\in \N}$ will be
called \textbf{finitely saturated} if for all
$\lambda,\mu,\eta\in \widehat{G}$ the fiber
set $V=V(\lambda\times\mu:\eta)$
admits finitely many orbits with respect to the two-fold
simultaneous action. We write $V(n)$ for $V(\lambda\times\mu:\eta)\cap G_n\times G_n$. If $L$ is an  orbit of $V(\lambda\times\mu:\eta)$ then $L(n)$ indicates the set $L\cap G_n\times G_n$. A finitely saturated family will be called \textbf{uniformly saturated} if there exists $m_L$ such that for all $n\geq m_L$, the set $L(n)$ is a single orbit of simultaneous conjugation action of $G_n$ on $G_n\times G_n$.

Next, let $(G_n)_{n\in \N}$ be a uniformly saturated family of finite groups and
$\lambda,\mu,\eta\in \widehat{G}$ be three
stable conjugacy classes in $G$. Assume that $L_1,\cdots,L_s$
is the totality of orbits in $V=V(\lambda\times\mu:\eta)$, which is finite
as the family is uniformly saturated. Set $m=\min\{l_{\lambda},l_{\mu},l_{\eta}\}$ so that
for any $n\geq m$ the intersections $\lambda(n),\mu(n)$ and $\eta(n)$ are non-empty and hence they
determine elements of $\widehat{G_n}$. For all $n\geq m$ the intersection $V(n)$ of the fiber $V$
with $G_n\times G_n$ is equal to the disjoint union of $L_i(n)$ and hence it follows that
\begin{equation}\label{size of the fiber}
|V(n)|=\sum_{i=1}^{s}|L_i(n)|.
\end{equation} Combining Lemma \ref{structure constants via fibers} and Eq.\eqref{size of the fiber} one can deduce that
\begin{equation}
c_{\lambda,\mu}^{\eta}(n) = \frac{|V(n)|}{|\eta(n)|} = \sum_{i=1}^{s}\frac{|L_i(n)|}{|\eta(n)|}\label{structure constants via orbits}
\end{equation}
Next we deal with the summands in Eq.\eqref{structure constants via orbits}. Let $(x_i,y_i)\in
L_i$. As $x_i,y_i\in G_n$ and $(x_i,y_i)\in V(n)$, the product $z_i:=x_iy_i$ is an element of $\eta\cap G_n$. So $\eta(n)$ is equal to the conjugacy class of $z_i$ in $G_n$, whose size is given by the usual formula:
\[
|\eta(n)|=|(z_i)^{G_n}|=\big|G_n/C_{G_n}(z_i)\big|
\]
where $C_{G_n}(z_i)$ denotes the centralizer of
$z_i$ in $G_n$. On the other hand, the size of
$L_i(n)$ is determined by the formula
$|G_n/Stab_{G_n}(x_i,y_i)|$ where
$Stab_{G_n}(x_i,y_i)$ denotes the stabilizer of
$(x_i,y_i)$ of the simultaneous conjugation
action of $G_n$ on $G_n\times G_n$. But it is
clear that the stabilizer of $(x_i,y_i)$ is
equal to the intersection $C_{G_n}(x_i)\cap
C_{G_n}(y_i)$. Combining all these, we find that $\frac{|L_i(n)|}{|\eta(n)|}=\frac{|C_{G_n}(x_iy_i)|}{|C_{G_n}(x_i)\cap C_{G_n}(y_i)|}$ and hence Eq.\eqref{structure constants via orbits} becomes
\begin{equation}\nonumber
c_{\lambda,\mu}^{\eta}(n)=\sum_{i=1}^s\frac{|C_{G_n}(x_iy_i)|}{|C_{G_n}(x_i)\cap C_{G_n}(y_i)|}
\end{equation}

Let us summarize the findings.

\begin{Proposition}\label{abstract version of determining structure constants}
	Let $(G_n)_{n\in \N}$ be a uniformly saturated family of groups. For each triple $\lambda$, $\mu$, $\eta$ of modified types in $\widehat{G}$, there exists an $m\in \N$ and a finitely many  elements $(x_1,y_1),\cdots,(x_s,y_s)\in \lambda(m)\times \mu(m)$ such that\begin{enumerate}
		\item $x_iy_i\in\eta(l)$ for $i=1,\dots, s$.
		\item For every $n\geq m$ the structural function $c_{\lambda,\mu}^{\eta}$ satisfies the relation below.
		\begin{equation}\label{structure functions via central index}
		c_{\lambda,\mu}^{\eta}(n)=\sum_{i=1}^s\frac{|C_{G_n}(x_iy_i)|}{|C_{G_n}(x_i)\cap C_{G_n}(y_i)|}.
		\end{equation}
		Each summand on the right hand side of the above equation will be referred as the \textbf{growth of the centralizer}.
		\item \label{index of centralizer function item} By the finiteness of the summation above, the growth of the structural function $c_{\lambda,\mu}^{\eta}(n)$ is determined by the growth of the centralizers
		\begin{equation}\label{index of centralizer function}
		n\longmapsto \frac{|C_{G_n}(x_iy_i)|}{|C_{G_n}(x_i)\cap C_{G_n}(y_i)|}
		\end{equation} In particular, if all the functions occurring in Eq.\eqref{index of centralizer function} are polynomials in $n$, then the structural function $c_{\lambda,\mu}^{\eta}(n)$ is also a polynomial in $n$.
	\end{enumerate}
\end{Proposition}

\subsection{Farahat-Higman ring}

In this section, we will consider a uniformly saturated family
$(G_n)_{n\in \N}$ of groups which admits a certain conjugation invariant
sub-additive function.  More precisely, let $(G_n)_{n\in \N}$ be a
uniformly saturated family of groups and assume that $G_n$
possesses a length function $||\cdot||_n$ with the following properties:
\begin{enumerate}
	\item $||\cdot||_n$ is stable under the embedding $G_n\subset G_{n+1}$. That is, if
	$x\in G_m$ and $n\geq m$ then
	\[
	||x^{\uparrow n}||_n=||x||_m.
	\] Hence, $G$ possesses a length function $||\cdot||:G\longrightarrow \N$ so that $||\cdot||_{|G_n}=||\cdot||_n$ for all $n\in \N$.
	\item $||\cdot||$, and hence $||\cdot||_n$, is constant on the conjugacy classes.
	\item $||\cdot||$, and hence $||\cdot||_n$, is sub-additive. That is,
	\[
	||xy||\leq ||x||+||y||.
	\]
\end{enumerate}
We will call such a family a
\textbf{filtered uniformly saturated family}. Notice that, since $||\cdot||$ is
constant on the conjugacy classes, one can transfer the length
function $||\cdot||$ to $\widehat{G}$ by setting $||\eta||:=||x||$ where
$\eta\in \widehat{G}$ and $x\in \eta$ is arbitrary. Following \cite{FH} we introduce the following algebra $\mc{S}'(G)$ defined as follows:  Let $(G_n)_{n\in\N}$ be
filtered uniformly saturated family and assume that the functions $c_{\lambda,\mu}^{\eta}(n)$ are polynomials of $n$ for all $\lambda,\mu,\eta$. Let $\mc{B}$ be the subring of polynomials $f(T)\in \Z[T]$ which maps integers to integers and consider
$\mc{S}'(G):=\mc{B}[K_{\lambda}:\lambda\in \widehat{G}]$, the free polynomial algebra over the ring $\mc{B}$ with the
indeterminates $K_{\lambda}\in \widehat{G}$, where the multiplication is defined as
\begin{equation}\label{free ring multiplication}
K_{\alpha}\cdot K_{\mu}=\sum_{\eta\in\widehat{G}}c_{\alpha,\mu}^{\eta}(T)K_{\eta}.
\end{equation}
Notice that the sum is actually a finite sum, and thus, meaningful. This is an associative and commutative ring and the evaluation map $f(T)\longmapsto f(n)$ induces a surjection from $\mc{S}'(G)$ onto $\mc{H}_n$. Now using the filtration, we define the induced filtered ring, called the \textbf{Farahat-Higman ring} of the uniformly saturated family and denote it by $\mc{S}(G)$ by setting:
\begin{equation}\label{filtered algebra multiplication}
K_{\alpha}\cdot K_{\mu}=\sum_{\eta\in\widehat{G}\atop ||\alpha||+||\mu||=||\eta||}c_{\alpha,\mu}^{\eta}(T)K_{\eta}.
\end{equation}
Following Wan and Wang, we say that the family $(G_n)_{n\in \N}$ satisfies the \textbf{stability property} if the structure constants $c_{\alpha,\mu}^{\eta}(T)$ of the Farahat Higman ring are independent of $T$, i.e. $c_{\alpha,\mu}^{\eta}(T)\in \Z$.

\subsection{An example: The uniformly saturated family $(S_n)_{n\in \N}$}

This section summarizes the work \cite{FH} of Farahat-Higman. The notation introduced below will be used later in the cases of the families $(GL_n(q))_{n\in N}$ and $(Sp_n(q))_{n\in N}$.

We introduce the relevant notation.
	\begin{enumerate}
		\item A \textbf{partition} $\lambda$ is a
		non-increasing sequence of non-negative
		integers $(\lambda_1,\cdots,\lambda_r,\cdots)$
		where almost all $\lambda_i$-s are zero.
		\item The integers
		$\lambda_i$ are called the \textbf{parts} of
		$\lambda$ and the number of non-zero
		$\lambda_i$'s is called the \textbf{length} of
		$\lambda$ and denoted by $l=l(\lambda)$ and we write $\lambda=(\lambda_1,\cdots,\lambda_l)$ and omit the zeros in the tail.
		\item Let $\lambda=(\lambda_1,\cdots,\lambda_r)$ be a partition. If $m_k=|\{i:\lambda_i=k\}|$ then $\lambda$ can be denoted as $(1^{m_1},\cdots,\lambda_1^{m_{\lambda_1}})$.
		\item The
		\textbf{weight}  $||\lambda||$ of a partition $\lambda$ is
		defined to be the integer $\sum_{i\in \N}\lambda_i$, which is well-defined as the sum
		is in fact over a finite set.
		\item If $||\lambda||=n$
		then one says $\lambda$ is a partition of $n$
		and writes $\lambda\vdash n$. The set of
		partitions of $n$ is denoted by $\mc{P}_n$ and
		the set of all partitions is denoted by
		$\mc{P}$ which is the union of
		$(\mc{P}_n)_{n\in\N}$.
		\item For $k>0$, the partition $1_k$ is the unique partition whose non-zero parts are $1$ and weight is $k$. There is a unique partition of $0$, the empty partition $\emptyset$.
		\item For two partitions
		$\lambda,\mu$, their sum $\lambda\cup \mu$ is defined to be the unique
		partition whose parts consists of parts of $\lambda$ and $\mu$.
		\item For a
		partition $\lambda=(\lambda_1,\cdots,\lambda_k)$, the
		\textbf{completion} $\overline{\lambda}$ is the partition
		$(\lambda_1+1,\cdots,\lambda_k+1)$. The weight of the completion of $\lambda$ is clearly equal to $||\lambda||+l(\lambda)$.
		\item For an integer $n\geq
		||\overline{\lambda}||=||\lambda||+l(\lambda)$ the $n$-th
		\textbf{completion} $\lambda^{\uparrow n}$ is the partition
		$\overline{\lambda}\cup 1_r$, where $r=n-||\overline{\lambda}||$.
		\item If $\lambda_i\geq \mu_i$ for all $i\in \N$, then one defines
		$\lambda-\mu$ as the partition whose parts are $\lambda_i-\mu_i$. For
		a partition $\lambda$ with length $r$, the partition
		$\overset{\circ}{\lambda}=\lambda-1_r$ is called the \textbf{modification} of
		$\lambda$.
	\end{enumerate}

If $\lambda$ is the empty partition we still talk of the
modification, completion and $n$-th completion of $\lambda$. The first two are again the empty partition and the $n$-th completion of the empty partition is clearly equal to $1_n$. Later we will introduce
the notion of \textbf{partition valued functions}, and analogous
concepts to weight, completion and modification will be introduced.
\begin{Example}
	Consider $\lambda=(4,3,3,2,1,1,1)$, a partition of
	$15=4+3+3+2+1+1+1$. The length of $\lambda$ is $7$. The
	modification $\overset{\circ}{\lambda}$ of $\lambda$ is
	$(4-1,3-1,3-1,2-1,1-1,1-1,1-1)=(3,2,2,1)$, which is a partition of
	$8$. The completion $\overline{\overset{\circ}{\lambda}}$ of
	$\lambda^{\circ}$ is $(4,3,3,2)$. The $15$-th completions of
	$\overset{\circ}{\lambda}$ and $\overline{\overset{\circ}{\lambda}}$ are both equal to
	$(4,3,3,2,1,1,1)$.
\end{Example}
Let $A$ be a subset of $\N$ and $g$ be a permutation of $A$. The support $[g]$ of $g$ is defined to be the subset $[g]:=\{a\in A:g(a)\neq
a\}$ of $A$. The group of
permutations $g$ of $A$ with finite
support is denoted by $S_A$. For $n\in \N$, let $[n]$ indicates the set
$\{1,2,\cdots,n\}$. When $A=[n]$, we will follow the usual notation and
simply write $S_n$ instead of $S_{[n]}$.
It is well-known that the conjugacy class of an element $g\in
S_n$ is completely determined by the cycle
type of $g$, which determines a unique
partition $\lambda^g=\lambda$ of $n$. The \textbf{reflection length} $l(g)$ of $g\in S_n$ is the minimal number of transpositions whose product is equal to $g$. As transpositions generate symmetric group, this definition of reflection length makes sense.

The symmetric group $S_{n}$ embeds in $S_{n+1}$ in a natural way. The conjugacy classes $\widehat{S_n}$ of $S_n$ are in 1-1 correspondence with $\mc{P}_n$. The family is clear saturated. The union of $S_n$, $n\in \N$, is denoted by $S_{\infty}$, it is the group of permutations of $\N$ whose supports are finite.

\begin{Lemma}\cite{FH} The family $(S_n)_{n\in \N}$ is a saturated family of groups and the bijections $\widehat{S_n}\longrightarrow \mc{P}_n$ induce the commutative diagram below.
	\begin{equation}
	\xymatrix{
		\widehat{S_n} \ar[d]_{{g^{S_n}}\mapsto g^{S_{\infty}}} \ar[r]^{\sim} &\mc{P}_n\ar[d]^{\lambda\mapsto \overset{\circ}{\lambda}}\\
		\widehat{S_{\infty}} \ar[r]^{\sim} & \mc{P}}
	\end{equation}
	In particular, the conjugacy classes of $S_{\infty}$ are in 1-1 correspondence with the set of all partitions.
\end{Lemma}

From the lemma it also follows that the abstract definitions of the concepts of modification, completion and $n$-th completion introduced earlier are consistent with the concrete definitions given in this section.

\begin{Lemma}
	The  reflection length is constant on conjugacy classes and it is sub-additive. It is also stable under the embedding $S_m\hookrightarrow S_n$ for $m\leq n$. Moreover, the reflection length of $g$ is equal to the $||\overset{\circ}{\lambda}||$.
\end{Lemma}

\begin{Example}
	Consider the permutation $g=(345)(78)$. As an element of $S_8$ and
	$S_{10}$, the conjugacy class of $g$ corresponds to the partitions
	$(3,2,1,1,1)$ and $(3,2,1,1,1,1,1)$ respectively. As an element of
	$S_{\infty}$ the conjugacy class of $g$ corresponds to the
	partition $(2,1)$. The completion of $(2,1)$ is $(3,2)$ whose
	weight is $5$. The level of $g$ is also $5$ which is equal to
	$|[g]|$. The reflection length of $g$ is $3$ and it is equal to
	$||\overset{\circ}{(\lambda^g)}||=||(2,1)||$.
\end{Example}

The following lemma is the normal form theorem in the context of symmetric groups whose proof is evident.

\begin{Lemma}\label{normal form FH}
	Let $g,h\in S_n$ and assume that $\big|[g]\cup[h]\big|=m\leq n$. Then there is an element $z$ in $S_n$ so that $(g,h)^z\in G_m\times G_m$.
\end{Lemma}

\begin{Proposition}[\cite{FH}](Farahat-Higman)
	The family $(S_n)_{n\in \N}$ is a uniformly saturated family of groups.
\end{Proposition}
\begin{proof}
	Let $\lambda,\mu,\eta$ be three modified types in $\widehat{S_{\infty}}$ and
	consider $V=V(\lambda\times \mu:\eta)$. For $(g,h)\in V$, the number $|[g]\cup[h]|$ is bounded by $m:=|[g]|+|[h]|$. Hence every orbit has
	a representative in the finite group $G_m\times G_m$, thus there is at most finitely many orbits. (Compare with Lemma \ref{normal form}.)
\end{proof}

\begin{Remark}[Growth of centralizers]\label{centralizers in symmetric group remark}
	If $g,h$ are two elements of $S_n$ then
	\begin{equation}\label{centralizers in symmetric group}
	C_{S_n}(gh)=C_{S_{[gh]}}(gh)\oplus S_{[n]-[gh]}
	\end{equation}
	and hence
	\begin{equation}\label{intersection of centralizers in symmetric group}
	C_{S_n}(g)\cap C_{S_n}(h)=(C_{S_{[g]\cup [h]}}(g)\cap C_{S_{[g]\cup [h]}}(h))\oplus S_{[n]-([g]\cup[h])}.
	\end{equation}
\end{Remark}

\begin{Proposition}[\cite{FH}](Farahat-Higman)
	For all $\lambda,\mu,\eta\in\widehat{S}_{\infty}$, the structural functions $c_{\lambda,\mu}^{\eta}(n)=p_{\lambda,\mu}^{\eta}(n)$ for some polynomial $p_{\lambda, \mu}^{\eta}(t)\in \Z[t]$  for large $n$.
\end{Proposition}
\begin{proof}
	It is clear that, the index of the two groups occurring in Eq.\eqref{intersection of centralizers in symmetric group} and Eq.\eqref{centralizers in symmetric group}  is a polynomial in $n$. In fact, if $|[gh]|=r\leq s=|[g]\cup [h]|$ then
	\begin{eqnarray}
	\frac{C_{S_n}(gh)}{C_{S_n}(g)\cap C_{S_n}(h)} & = & \Big|\frac{C_{S_{[gh]}}(gh)}{C_{S_{[g]\cup [h]}}(g)\cap C_{S_{[g]\cup [h]}}(h)}\Big|\frac{(n-r)!}{(n-s)!}\nonumber\\
	& = & \Big|\frac{C_{S_{[gh]}}(gh)}{C_{S_{[g]\cup [h]}}(g)\cap C_{S_{[g]\cup [h]}}(h)}\Big| (n-r)\cdots (n-s+1).
	\end{eqnarray}  Since the family is uniformly saturated the result
	follows from Remark \ref{abstract version of determining structure constants}/\ref{index of centralizer function item}.
\end{proof}

Notice that, in the above proof, the degree of the polynomial is equal to $|[g]\cup [h]|-|[gh]|$,
which is zero only if $[g]\cup [h]=[gh]$. The next lemma establishes a criteria to guarantee the equality.

\begin{Lemma}\cite{FH}\label{reflection length and support FH}
	Let $g,h\in S_n$. If $||\overset{\circ}{(\lambda^g)}||+||\overset{\circ}{(\lambda^h)}||=\overset{\circ}{(\lambda^{gh})}$ then $[g]\cup [h]=[gh]$.
\end{Lemma}

\begin{Proposition}\cite{FH}
	For $g\in S_{m}$ and $n\geq m$
	\begin{equation}
	\overset{\circ}{((\lambda^g)^{\uparrow n})}=\overset{\circ}{\lambda^g}.
	\end{equation}
	The weight of $\overset{\circ}{(\lambda^g)}$ is equal to the reflection length
	of $g$. Hence, if $|\overset{\circ}{\theta}|>|\overset{\circ}{\alpha}|+|\overset{\circ}{\beta}|$
	then $c_{\alpha,\beta}^{\theta}(n)=0$ for all $n\in \N$. If the
	equality holds, then the polynomial $p_{\alpha,\beta}^{\theta}(t)$ is
	constant.
\end{Proposition}

\begin{Corollary}\cite{FH}
	The uniformly saturated family $(S_n)_{n\in \N}$ satisfies the stability property.
\end{Corollary}

\section{The uniformly saturated family $GL_n(q)$ and the work of Wan and Wang}\label{C:3}

In this chapter, we summarize the work \textit{Stability of the centers
	of group algebras of $GL_n(q)$} of Wan and Wang, \cite{WW18}. In the
first section, following \cite{Mac} and \cite{J74} we review the
general theory of $GL_n(q)$ and parameterize the conjugacy
classes in general linear groups over a finite field. In the second
section, we closely follow \cite{WW18} and construct the uniformly saturated family $GL_n(q)$. In the following sections, we present
the main theorems of Wan and Wang without proofs. Some of the theorems
are divided into smaller pieces because some parts will be used in the
symplectic case.
Some general facts concerning the centralizers of block matrices will
also be discussed in as they are used in the proofs of Wan and Wang and as well as in our study concerning symplectic group rings.

\subsection{Notation and preliminaries}

Let $p$ be a prime and $q$ be a power of $p$. The finite field with $q$ is denoted by $\F_q$. The
set of monic irreducible polynomials $p(t)\in\F_q[t]-\{t\}$ is denoted by $\Phi$.
For an abstract finite dimensional vector space $V$ and $U\in GL(V)$ the
\textbf{residual} $R^{U}$ and \textbf{fixed} space $V^{U}$ of $U$ are defined as
\begin{equation}\label{residual and fixed space}
R^{U}=(U-1_V)V,\qquad V^{U}=\ker (U- 1_V).
\end{equation} An element in $GL(V)$ is called a \textbf{reflection} if
$\dim R^{U}=1$, equivalently, codimension of $V^{U}$ is $1$
by the equality $\dim R^{U}+\dim V^{U}=\dim V$.
The \textbf{reflection length} $l(U)$ of $U\in GL(V)$ is the
minimum number $r$ such that there exists a sequence of reflections of reflections $\tau_1,\cdots,\tau_r$ such that
$U=\tau_1\cdots \tau_r$.

Next we introduce the relevant combinatorial objects. These definitions will be used in symplectic group case as well.

\begin{Definition}\label{notation: partition valuede functions}
	\begin{enumerate}
		\item A \textbf{partition valued function} $\pmb{\lambda}$ on $\Phi$ is a function from $\Phi$ to
		the set of partitions $\mc{P}$ such that for almost all $f\in\Phi$, the image
		$\lambda(f)$ is the empty partition. The image will be sometimes denoted by $\lambda_f$ depending on the convenience.
		\item  The \textbf{weight} $||\pmb{\lambda}||$ of a partition
		valued function $\pmb{\lambda}:\Phi\longrightarrow \mc{P}$ is defined as follows:
		\begin{equation}\label{weight of the function definition}
		||\pmb{\lambda}||=\sum_{f\in \Phi} \deg(f)\cdot ||\lambda_f||
		\end{equation}
		which makes sense as the weight of the empty partition is by definition equal to zero. The set of partition valued functions on $\Phi$ of weight $n$ is denoted by $\mc{P}_n(\Phi)$. The set of all partition valued functions is denoted by $\mc{P}(\Phi)$.
		
		\item The sum $\pmb{\lambda}\cup \pmb{\mu}$ of two partition valued functions $\pmb{\lambda}$ and $\pmb{\mu}$ is defined as the function sending $f$ to $\lambda(f)\cup\mu(f)=\lambda_f\cup\mu_f$.
		
		\item (\cite{WW18}) The \textbf{unipotent part} $\pmb{\lambda}^e$ and non-\textbf{unipotent part} $\pmb{\lambda}^{ne}$ of $\pmb{\lambda}$ are defined as follows. The partition valued function $\pmb{\lambda}^e$ induced by the partition valued function $\pmb{\lambda}$ as follows:
		\[
		\text{$\pmb{\lambda}^e(t-1)=\lambda({t-1}), \qquad $and$\quad \pmb{\lambda}^e(f)=\emptyset, \quad\forall f\neq t-1$.}
		\]
		The non-unipotent part $\pmb{\lambda}^{ne}$ of $\pmb{\lambda}$ is defined as follows:
		\[
		\text{$\pmb{\lambda}^{ne}(t-1)=\emptyset , \qquad $and$\quad \pmb{\lambda}^{ne}(f)=\lambda(f),\quad \forall f\neq t-1$.}
		\]
		It is clear that, for a partition valued function $\pmb{\lambda}$ the equality below holds:
		\[
		\pmb{\lambda}^{e}\cup \pmb{\lambda}^{ne}=\lambda.
		\]
		\item A partition valued function $\pmb{\lambda}$ is called a \textbf{unipotent function} if it is equal to its unipotent part.
	\end{enumerate}
\end{Definition}

\begin{Example}\label{partition valued function first example}
	Let $\alpha\in\F_q$ be a non-square. Define $\pmb{\mu}\in \mc{P}(\Phi)$ by setting
	\begin{equation}
	\text{$\mu({t-1})=(3,2,1,1),\quad$ and $\quad \mu({t^2-\alpha})=(2,2,1)$},
	\end{equation} and for $f\neq t-1,t^2-\alpha$, set $\mu(f)=\emptyset\in\mc{P}_0$. By definition we get \begin{equation}
	||\pmb{\mu}||=1\cdot(3+2+1+1)+2\cdot (2+2+1)=17.
	\end{equation}  The unipotent part $\pmb{\lambda}^e$ is equal to the
	function which assigns $(3,2,1,1)$ to $(t-1)$ and assigns the empty
	partition $\emptyset$ to $f$ for all $f\in \Phi-\{t-1\}$. The non-unipotent
	part $\pmb{\lambda}^{ne}$ of $\pmb{\lambda}$ is the partition valued function that assigns
	$(2,2,1)$ to $t^2-\alpha$ and $\emptyset$ to $f$ for all $f\in \Phi-\{t^2-\alpha\}$.
\end{Example}

The following concepts are introduced in \cite{WW18} as variants of
modification, completion and $n$-th completion.  Recall that the
modification, completion and $n$-th completion of the empty partition
were formally defined.

\begin{Definition}[Wan-Wang]\label{definition of uparrow n}
	Let $\pmb{\mu}\in\mc{P}_n(\Phi)$ be a partition valued function of
	weight $n$. The \textbf{modification} $\overset{\circ}{\pmb{\mu}}$ is
	the partition valued function defined as the unique
	partition valued functions satisfying
	\begin{equation}
	\text{$\overset{\circ}{\pmb{\mu}}(t-1)=\overset{\circ}{\pmb{\mu}(t-1)}\quad $ and $\overset{\circ}{\pmb{\mu}}(f)=\pmb{\mu}(f)\quad$}
	\end{equation}
	for all $f\in \Phi-\{t-1\}$. The \textbf{completion} $\overline{\pmb{\mu}}$ of $\pmb{\mu}$ is
	the partition valued function defined as the unique
	partition valued functions satisfying
	\begin{equation}
	\text{$\overline{\pmb{\mu}}(t-1)=\overline{\pmb{\mu}(t-1)}\quad $ and $\overline{\pmb{\mu}}(f)=\pmb{\mu}(f)\quad$}
	\end{equation}
	for all $f\in \Phi-\{t-1\}$. For $n\geq ||\overline{\pmb{\mu}}||$,
	define the $n$-completion $\pmb{\mu}^{\uparrow n}\in\mc{P}_n(\Phi)$ to be the unique partition valued
	function that satisfies
	\begin{equation}
	\pmb{\mu}^{\uparrow n}(t-1)=\pmb{\mu}(t-1)^{\uparrow r}
	\end{equation}
	where $r=n-||\mu||$ and $\pmb{\mu}^{\uparrow n}(f)=\pmb{\mu}(f)$ for all $f\neq t-1$.
\end{Definition}

Notice that all the operations sending $\pmb{\mu}$ to
$\overset{\circ}{\pmb{\mu}}$, or to $\overline{\pmb{\mu}}$ or to
$\pmb{\mu}^{\uparrow n}$ affects only the unipotent part
$\pmb{\mu}^e$ of $\pmb{\mu}$.

\begin{Example}\label{partition valued function second example}
	Let us observe the effects of the operations just introduced on the partition valued function $\pmb{\mu}$ of Example \ref{partition valued function first example}, which was defined as
	\begin{equation}
	\text{$\pmb{\mu}({t-1})=(3,2,1,1),\quad$ and $\quad \pmb{\mu}({t^2-\alpha})=(2,2,1)$},
	\end{equation} and for $f\neq t-1,t^2-\alpha$, set $\mu(f)=\emptyset\in\mc{P}_0$ where $\alpha\in \F_q$ is a non-square. Then
	\begin{eqnarray}
	\overset{\circ}{\pmb{\mu}}({t-1}) & = & \overset{\circ}{\pmb{\mu}(t-1)}=\overset{\circ}{(3,2,1,1)}=(2,1) \nonumber\\ \overset{\circ}{\pmb{\mu}}(t^2-\alpha) & = & \pmb{\mu}(t^2-\alpha)=(2,2,1) \nonumber \\ \overset{\circ}{\pmb{\mu}}(f) & = & \emptyset
	\end{eqnarray} for all $f\neq t-1,t^2-\alpha$. The following equalities follow from the definitions.
	\begin{equation}
	\text{$\overline{\overset{\circ}{\pmb{\mu}}}({t-1})=\overline{\overset{\circ}{\pmb{\mu}}(t-1)}=(3,2),\quad$ and $\quad \overline{\overset{\circ}{\pmb{\mu}}}(f)=\overset{\circ}{\pmb{\mu}}(f)$},
	\end{equation} for all $f\neq t-1$. The weight of $\overset{\circ}{\pmb{\mu}}$ is $1\cdot (3+2)+2\cdot (2+2+1)=15$. Clearly, $(\overset{\circ}{\pmb{\mu}})^{\uparrow 17}=\pmb{\mu}$.
\end{Example}

\subsection{Conjugacy classes in general linear groups}

Let $U\in GL(V)$. For $v \in V$, the association  
$v\longmapsto U\cdot v$ defines an $\F_q[t]$-action on $V$ in the following way. Define
an $\mathbb{F}_q[t]$-module structure $(V_{U},\cdot_U)$ on $V$ by setting
$t\cdot_U v=U\cdot v$ and extending it linearly.

\begin{Remark}
	
	The most important property of this module is that it characterizes the conjugacy class of the defining element of the $\F_q[t]$-module. Let $U_1,U_2\in GL(V)$ be two $F$-automorphism of $V$ and assume that the elements $U_1$ and $U_2$ are conjugate: $U_1U=UU_2$ for some $U\in GL(V)$, which implies
	\[
	t\cdot_{U_1} (U_1(v))= U_1U(v) = UU_2(v) = U(t\cdot_{U_2} v).
	\]
	As a result $v\longmapsto U(v)$ defines an $F(t)$-module isomorphism from $V_{U_1}$ to $V_{U_2}$. Let us rewrite the last inequality in a more suggestive form:
	\begin{equation}\nonumber
	\xymatrix{
		V \ar@{->}[d]^{U}\ar[r]^{t\cdot_{U_2}} & V \ar@{->}[d]^{U} \\
		V \ar@{->}[r]^{t\cdot_{U_1}} & V
	}
	\end{equation}
	which reads as $V_{U_1}$ and $V_{U_2}$ are isomorphic representation spaces of $F[t]$. Conversely, if $U$ is such a module isomorphism, then it is clearly a linear isomorphism which satisfies $U_1U=UU_2$. As a result we have
	\begin{equation}\label{conjugacy class and isomorphism class are the same concepts}
	U_1^G=U_2^G \Longleftrightarrow V_{U_1}\simeq V_{U_2},
	\end{equation}
	for all $U_1,U_2\in G$. The  Eq.\eqref{conjugacy class and isomorphism class are the same concepts} can be stated
	in terms of representations. The elements $U_1$ and $U_2$ are conjugate if and only if there is an $F[t]$-equivariant isomorphism between $V_{U_1}$ and $V_{U_2}$. This interpretation will allow us to show that an equation of type
	\[
	XA=BX,\, A\in \mat_{n\times n},\,B\in \mat_{m\times m},\, X\in \mat_{m\times n}
	\]
	admits only the trivial solution $X=0$ when $V_A$ and $V_B$ are non-isomorphic simple modules. Of course, this is just a special case of Schur's lemma.
\end{Remark}

Let $U\in GL(V)$, be a fixed linear endomorphism of $V$. Since $\mathbb{F}_q[t]$ is a PID and $V_{U}$ is a finite dimensional module, the elementary
divisor theory applies and $V_U$ admits a decomposition into primary cyclic modules where a primary cyclic $\mathbb{F}_q[t]$ module is by definition in the following form:
\begin{equation}\label{primary cyclic module definition}
N_{f,i}:=\mathbb{F}_q[t]/(f)^i, \quad i>0,f\in \Phi.
\end{equation} It is well known that the decomposition into primary cyclic modules is unique on the isomorphism class of $V_U$ up to permuting the orders of the summands (\cite[Chapter 3]{J74}). Let
\begin{equation}\label{primary cyclic module decomposition}
V_U=\bigoplus_{i=1}^{r_U} M_i
\end{equation}
be a decomposition of $V_U$ into primary cyclic modules and for $f\in \Phi$. For $l\in\N$ define
\[
m^f_{l}=m_l=|\{i:M_i\simeq N_{f,l} \}|,
\] the number
of copies of $N_{f,l}$ in the decomposition of $V_{U}$ into primary cyclic modules. As there
are only finitely many such summands, $m_l=0$ for almost all $l$, in fact, for $l>r_U$ one has $m_l^f=0$. Thus, the decomposition Eq.\eqref{primary cyclic module decomposition} determines a partition $(1^{m_1},\cdots,r^{m_r})$ attached to $f$, as a result one obtains a partition valued function $\pmb{\lambda}^U$ which sends $f$ to the partition $\lambda^U(f)=\lambda^U_f=(\lambda^U_{f,1},\cdots,\lambda^U_{f,l_f})$, which is defined as above. With this notation the above decomposition can be written as
\begin{equation}\label{decomposition of V-U}
V_{U}=\bigoplus_{f\in \Phi}N_{f,\pmb{\lambda}^{U}_f}
\end{equation}
where
\[N_{f,\pmb{\lambda}^{U}_f}=\bigoplus_{i= 1}^{l_f} \mathbb{F}_q[t]/(f)^{\pmb{\lambda}^{U}_{f,i}}.\]
The weight $||\pmb{\lambda^U}||=\dim V$ which follows from the fact that $\dim_{\F_q} N_{f,i}=i\cdot \deg(f)$ together with Eq.\eqref{decomposition of V-U}. Conversely, it can be shown that for each such function $\pmb{\lambda}$, the
corresponding $\mathbb{F}_q[t]$-module is realized by an element $U$ of $GL(V)$. In fact, for a given polynomial $f(t)\in\Phi$ and $m\geq 1$, write $f(t)^m=t^k-a_{k-1}t^{k-1}-\cdots-a_0$, and introduce the \textbf{companion matrix} $J_{f^m}$ of $f^m$ by setting
\begin{equation}\label{companiaon matrix of f}
J_{f^m}=\begin{bmatrix}
0 & 0 & \cdots & & 0 & a_0 \\[0.3em]
1 & 0 &  \cdots & & 0  & a_1 \\[0.3em]
0 & 1 &  \cdots & & 0 & a_2 \\[0.3em]
\vdots &  & \ddots & & \vdots & \vdots \\[0.3em]
0 & 0 &\cdots & & 1 & a_{k-1}
\end{bmatrix}_{k\times k}
\end{equation}
It is well-known that the $\F_q[t]$ module defined by $J_{f^m}$ is isomorphic to
\begin{equation}
\F_q[t]/(f(t)^m)=N_{f,m}.
\end{equation}
So, if $\lambda_f=(\lambda_1,\cdots,\lambda_r)$ and if $J_{\lambda_f}$ denotes the block diagonal matrix $diag(J_{f^{\lambda_1}},\cdots,J_{f^{\lambda_r}})$ then the block diagonal matrix
\begin{equation}\label{conjugacy class representative}
J_{\pmb{\lambda}}:=diag(J_{\lambda_f})_{f\in \Phi}
\end{equation}
is an element of the conjugacy class in $GL_n(\F_q)$ that induces the partition valued function $\pmb{\lambda}$. This finishes the characterization of the conjugacy classes of $GL_n(q)$. Let us summarize.

\begin{Proposition}\label{Conjugacy classes of GL-n}
	The association $U\longmapsto \pmb{\lambda}^{U}$ defines a surjection $GL_n(\F_q)\longrightarrow\mc{P}_n(\Phi)$. Two endomorphisms $U_1,U_2\in GL_n(\F_q)$ define the same partition valued function if and only if they are conjugate in $GL_n(\F_q)$. In particular, $U\longmapsto \pmb{\lambda}^{U}$ induces a bijection
	\[
	\widehat{G_n}\longrightarrow \mc{P}_n(\Phi).
	\]
	The basis elements of $\mc{H}_n$ thus can be indexed by the elements of $\mc{P}_n(\Phi)$.
\end{Proposition}

\begin{Remark}\label{schurs lemma and primary cyclic modules}
	Consider two primary cyclic modules $M_i=\F_q[t]/(f_i^{m_i})$, $i=1,2$ with distinct irreducible monic polynomials $f_1,f_2$. Then $\F_q[t]$-modules $V_1$ and $V_2$ and by Schur's lemma there is no intertwining operator between them.
\end{Remark}

The use of suitable
representatives is particularly important
in calculations done in \cite{WW18} as well as
in the symplectic group case which will
be investigated later. The main importance of choosing a suitable form is that it enables one to compute the functions defined in the form $C(U_1U_2)/C(U_1)\cap C(U_2)$, cf. (\ref{index of centralizer function}), via proving a result similar to the one presented in Remark \ref{centralizers in symmetric group remark}, Eq.\eqref{centralizers in symmetric group}. We recall the basic result in the least explicit form, yet it will be enough for our purposes.

\begin{Lemma}\cite[Chapter 3/10]{J74} Let $U\in End(V)$ and $m_U(t)=\prod m_i(t)^{r_i}$ be the minimal polynomial of $U$, where $\gcd(m_i,m_j)=1$ for $i\neq j$. Then there is a basis $B$ of $V$ such that the matrix of $U$ with respect to $B$ is in block diagonal form $diag(M_1,\cdots,M_r)$ where minimal polynomial of $M_i$ is $m_i^{r_i}(t)$.
\end{Lemma}

The blocks $M_i$'s admits further decomposition into a block diagonal form, where minimal polynomial of each block of $M_i$ is a power of $m_i$. The explicit blocks can be given depending on the minimal polynomial.

\begin{Remark}[Centralizers of block diagonal matrices and Schur's lemma] \label{Centralizers of block diagonal matrices}
	Let $U$ be an $n\times n$ invertible block diagonal matrix
	$diag(U_1,\cdots,U_k)$, where $U_i$ is an $n_i\times n_i$ square matrix and let $D$ be an $n\times n$ matrix. The block structure of $U$ can be used to write $D$ as a block matrix
	$(D_{ij})^k_{i,j=1}$, where $D_{ij}$ is an $n_i\times n_j$ matrix. The matrix $D$ commutes with $U$ if and only if the equation below holds:
	\begin{equation}
	diag(U_1,\cdots,U_k)D=Ddiag(U_1,\cdots,U_k),
	\end{equation}
	which can be written in detail:
	\begin{equation}
	\begin{bmatrix}
	U_1D_{11} & U_1D_{12} & \cdots  & U_1D_{1k} \\
	U_2D_{21}& U_2D_{22} & \cdots & U_2D_{2k }\\
	\vdots & \vdots & \cdots  & \vdots \\
	U_kD_{k1}& U_kD_{k2} & \cdots & U_kD_{kk } \\
	\end{bmatrix}=\begin{bmatrix}
	D_{11}U_1 & D_{12}U_2 & \cdots  & D_{1k}U_k \\
	D_{21}U_1& D_{22}U_2 & \cdots & D_{2k}U_k\\
	\vdots & \vdots & \cdots  & \vdots \\
	D_{k1}U_1& D_{k2}U_2 & \cdots & D_{kk}U_k \\
	\end{bmatrix}
	\end{equation}
	So, $D$ commutes with $U$ if and only if
	\begin{equation}\label{block diagonal centralizer}
	U_iD_{ij}U^{-1}_j=D_{ij}
	\end{equation}$\forall i,j=1,\cdots,k$. Now assume that, each $U_i$ is of the form $J_{\lambda(f_i)}$ where $f_i$ and $f_j$ are distinct irreducible polynomials for $i\neq j$. Writing Eq.\eqref{block diagonal centralizer} as $U_iD_{ij}=D_{ij}U_j$, we see that $D_{ij}$ defines an intertwining operator between $N_{f_1,\lambda(f_1)}$ and $N_{f_2,\lambda(f_2)}$. Such an operator must be zero if $f_1\neq f_2$ according to the Remark \ref{schurs lemma and primary cyclic modules}. As a consequence, we obtain the following direct sum decomposition of the centralizer of $diag(J_{\lambda_f})_{f\in \Phi}$:
	\begin{equation}\label{centralizer of a block diagonal matrix admits a direct sum decomposition}
	C(diag(J_{\lambda_f})_{f\in \Phi})\simeq\bigoplus_{f\in \Phi} C(J_{\lambda_f}).
	\end{equation}
\end{Remark}

\begin{Remark}\label{remark definition of S-n}
	There are other rational forms that represent conjugacy classes. The following one will be useful in the context of symplectic groups. For $n\in \N$, the matrix \[
	S_{n}=\begin{blockarray}{ccccccc}
	e_1 & e_2 & \cdots & e_{n-1} & e_n \\
	\begin{block}{[cccccc]c}
	1 &  &  & & &  & e_1 \\
	1 & 1 & & &  & & e_2 \\
	1 & 1 & \ddots & & & & e_2 \\
	&  \ddots & \ddots & \ddots & & & \vdots \\
	&    & \ddots & 1 & 1 & & e_{n-1}\\
	1 &  1  & \cdots & 1 & 1 & 1 & e_n\\
	\end{block}
	\end{blockarray},\qquad S_n^{-1}=\begin{blockarray}{ccccccc}
	e_1 & e_2 & \cdots & e_{n-1} & e_n \\
	\begin{block}{[cccccc]c}
	1 &  &  & & &  & e_1 \\
	-1 & 1 & & &  & & e_2 \\
	0 & -1 & \ddots & & & & e_2 \\
	&  \ddots & \ddots & \ddots & & & \vdots \\
	&    & \ddots & -1 & 1 & & e_{n-1}\\
	0 &  0  & \cdots & 0 & -1 & 1 & e_n\\
	\end{block}
	\end{blockarray}
	\]
	is an element of $GL_n(q)$. Its minimal polynomial is equal to $(t-1)^n$ and as an $\F_q[t]$-module, $V_{S_n}$ is isomorphic to $\F_q[t]/(t-1)^n=N_{t-1,n}$. Thus, the induced partition valued function $\pmb{\lambda}$ assigns the partition $(n)$ to $t-1$ and the empty partition to $f\in \Phi-\{t-1\}$. The fixed space $V^{S_n}$ of $S_n$is generated by $e_n$, in particular, dimension of the fixed space of $S_n$ is $1$.
\end{Remark}

\subsection{Uniformly saturated family $GL_n(q)$.}

In this section, following \cite{WW18} we construct the uniformly saturated family $(GL_n(q))_{n\in \N}$.

\begin{Definition}\cite{WW18}\label{chain of V-n}
	For $m\leq n$ consider the embedding $V_m \longrightarrow  V_n$ defined by the rule
	\begin{equation}
	(v_1,\cdots,v_m) \longmapsto (v_1,\cdots,v_m,\underbrace{0,\cdots,0}_{\text{$n-m$ many}})
	\end{equation}
	and identify $V_m$ with its image in $V_n$. Denote
	\begin{equation}
	V_{[n]-[m]}=\{(\underbrace{0,\cdots,0}_{\text{$m$ many}},w_1,\cdots,w_{n-m}):w_i\in\F_q\}
	\end{equation}
	which implies $V_n=V_m\oplus V_{[n]-[m]}$. For $U\in GL(V_m)=GL_m(\F_q)$ the injection $U^{\uparrow n}\in GL(V_n)$ is defined by setting $U\oplus I_{V_{[n]-[m]}}$.
	\begin{equation}
	U^{\uparrow n}= \begin{bmatrix}
	U & 0 \\
	0 & I_{n-m}
	\end{bmatrix}.
	\end{equation}
	The group $G_{\infty}=GL_{\infty}(q)$ is defined to be the union of $(GL_n(q))_{n\in \N}$.
\end{Definition}


We collect numerous results of Wan and Wang in the following lemma.

\begin{Lemma}\cite{WW18}\label{GL-n is a saturated family}The following hold:
	\begin{enumerate}
		\item The family $(GL_n(q))_{n\in \N}$ is a saturated family.
		\item The map $U\longmapsto \overset{\circ}{(\pmb{\lambda}^U)}$ induces a bijection between the conjugacy classes of $GL_{\infty}$ and $\mc{P}(\Phi)$, the set of all partition valued functions. The partition $\overset{\circ}{(\pmb{\lambda}^U)}$ is called the \textbf{modified type} of $U$.
		\item Let $\pmb{\lambda}$ be a partition valued function. Then $GL_n(q)$ contains an element whose modified type is $\pmb{\lambda}$ if and only if $||\overline{\pmb{\lambda}}||\leq n$.
		\item Let $\pmb{\lambda}$ be a partition valued function such that $||\overline{\pmb{\lambda}}||= m$ and let $U\in GL_m(q)$ be an element whose stable type is $\pmb{\lambda}$. If $n\geq m$ then
		\begin{equation}
		\pmb{\lambda}^{\uparrow n}=\pmb{\lambda}^{(U^{\uparrow n})}
		\end{equation}
	\end{enumerate}
\end{Lemma}
\begin{proof}
	All of the statements follows from the characterizations of conjugacy classes with partition valued functions and the definitions.
\end{proof}


\begin{Example}
	Let us reconsider the Example \ref{partition valued function second example}. Recall that the partition valued function $\pmb{\mu}$ was defined
	by setting
	\begin{equation}
	\text{$\pmb{\mu}({t-1})=(3,2,1,1),\quad$ and $\quad \pmb{\mu}({t^2-\alpha})=(2,2,1)$},
	\end{equation} and for $f\neq t-1,t^2-\alpha$, set $\mu(f)=\emptyset\in\mc{P}_0$ where $\alpha\in \F_q-\F_q^2$.
	We already observed that $||\pmb{\mu}||=1\cdot(3+2+1+1)+2\cdot (2+2)=17$. Let
	$\pmb{\lambda}=\overset{\circ}{\pmb{\mu}}$. More precisely
	\[\text{$\pmb{\lambda}(t-1)=(2,1)$, $\:\pmb{\lambda}(t^2-\alpha)=(2,2,1)$ and
		$\pmb{\lambda}(f)=\emptyset$,}\] for all $f\neq t-1,t^2-\alpha$. The completion
	$\overline{\pmb{\lambda}}$ of $\pmb{\lambda}$ differs from $\pmb{\lambda}$ only on the
	image of $t-1$. Applying Definition \ref{definition of uparrow n} we have $\overline{\pmb{\lambda}}(t-1)=\overline{\pmb{\lambda}(t-1)}=\overline{(2,1)}=(3,2)$. The weight of $\overline{\pmb{\lambda}}$ is $1\cdot (3+2)+ 2\cdot (2+2+1)=15$. As a result, for all $n\geq 15$, there is an element in $GL_n(q)$ whose modified type is equal to $\pmb{\lambda}$. Let $U\in GL_{15}(q)$ be an element whose modified type is equal to $\pmb{\lambda}$. Then, the partition valued function defined by $U^{\uparrow 17}$ is equal to $\pmb{\mu}$. If we denote the matrix of $U$ in $GL_{15}(q)$ again by $U$ then
	\[
	U^{\uparrow 17}=\begin{bmatrix}
	U & 0 \\
	0 & I_2\\
	\end{bmatrix}
	\]
\end{Example}

For a modified type $\pmb{\lambda}\in\mc{P}(\Phi)$, let ${\pmb{\lambda}}(n)$ be the intersection ${\pmb{\mu}}\cap GL_n(\F_q)$, which is non-empty if and only if $||\overline{\pmb{\lambda}}||\leq n$ and let
\begin{equation}
K_{\pmb{\lambda}}(n)=\sum_{U\in {\pmb{\lambda}}(n)}U.
\end{equation}
The sum $K_{\pmb{\lambda}}(n)$ is an element of $\mc{H}_n=\mc{H}(GL_n(\F_q))$, the center of the integral group algebra $\Z[GL_n(q)]$, as pointed earlier in the general setting of Eq.\eqref{class sum}. Notice that, if ${\lambda}(n)=\emptyset$ then the above sum is over the empty set and hence equal to $0$.

\begin{Lemma}\cite[Lemma 2.3]{WW18}\label{basis for center in the general linear case}
	The set $\{K_{\pmb{\lambda}}(n):\pmb{\lambda}\in \mc{P}(\Phi),\, K_{\pmb{\lambda}}(n)\neq 0\}$ forms the class sum $\Z$-basis for the center $\mc{H}_n$, for each $n\geq 0$.
\end{Lemma}

\subsection{The growth of the centralizers}

We have seen in Section \ref{Center of the groups rings and uniformly saturated families of groups}, Proposition \ref{abstract version of determining structure constants}, that in order to determine the structural functions $c_{\alpha,\beta}^{\theta}(n)$ one needs to study the growth of the centralizer of a fixed element as the groups enlarge. So, one needs a variant of Eq.\eqref{centralizers in symmetric group}. 

\begin{Remark}
	Recall that if $g\in S_m$ which has no fixed points and $n\geq m$ then
	\begin{equation}\label{centralizers in symmetric group without fixed points}
	C_{S_n}(g^{\uparrow n})=C_{S_m}(g)\oplus S_{n-m}.
	\end{equation}
	where, as before, $g^{\uparrow n}$ is the image of $g$ under the natural identification of $S_m$ in $S_n$.
\end{Remark}

\begin{Remark}\label{dimension of the fixed space and the length of mu(t-1)}
	Let $U\in GL_n(\F_q)$ and $\pmb{\lambda}^{U}$ be its non-modified type. Then $\dim V^U=l(\lambda(t-1))$. This can be seen directly from the fact that only the  companion matrices belonging to $t-1$ contributes to the $1$-eigenspace and for each block, the contribution to the dimension is incremented by $1$ (cf. Remark \ref{definition symplectic transformation}).
\end{Remark}

Let $\pmb{\mu}\in\mc{P}(\Phi)$, $m=||\overline{\pmb{\mu}}||$. Assume that $U\in GL_m(q)$ whose type is $\pmb{\mu}$. For the matrix $U$, the following is the variant of Eq.\eqref{centralizers in symmetric group without fixed points}. Let $\dim V^U=l(\mu_{t-1})=d$.

\begin{Proposition}\cite[Proposition 2.5]{WW18} \label{centralizer growth in gl}
	Let $n\geq m=||\pmb{\mu}||+d=||\overline{\pmb{\mu}}||$. Then, the centralizer $C_{GL_n(q)}(U^{\uparrow n})$ of $U^{\uparrow n}\in GL_n(\F_q)$ is given by
	\begin{equation}
	C_{GL_n(q)}(U^{\uparrow n})=\Big\{\begin{bmatrix}
	A & B \\
	C & D
	\end{bmatrix}: A\in C_{GL_m(q)}(U), D\in GL_{n-m}(\F_q), UB=B, CU=C\Big\}.
	\end{equation}
	In particular, $A$ and $D$ are invertible and hence
	\begin{equation}
	|C_{GL_n(q)}(U^{\uparrow n})|=|C_{GL_n(q)}(U)|\cdot |GL_{n-m}|\cdot q^{2\cdot d}
	\end{equation}
\end{Proposition}

\begin{proof}
	The second equality directly follows from the first equality and Remark \ref{dimension of the fixed space and the length of mu(t-1)}. Conditions on $B$ and $C$ follows from the equality
	\[
	\begin{bmatrix}
	U & 0 \\
	0 & I_{n-m}
	\end{bmatrix}\begin{bmatrix}
	A & B \\
	C & D
	\end{bmatrix}=\begin{bmatrix}
	A & B \\
	C & D
	\end{bmatrix}\begin{bmatrix}
	U & 0 \\
	0 & I_{n-m}
	\end{bmatrix}
	\]
	The proof of the invertibility of $A$ and $D$ can be found in \cite{WW18}. There, the authors in fact prove that
	\[
	\det(\begin{bmatrix}
	A & B \\
	C & D
	\end{bmatrix})=\det(A)\det(D)
	\]
	whenever $\begin{bmatrix}
	A & B \\
	C & D
	\end{bmatrix}$ is in the centralizer of $U$.
\end{proof}

\subsection{Reflection length, modified type and the main theorems of Wan and Wang}

The following Lemma is due to \cite{HLR17}. It is the analogue of Lemma \ref{reflection length and support FH} and used in \cite{WW18} to prove a similar result to Theorem \ref{normal form FH} in the case of $GL_n(q)$.

\begin{Lemma}\cite[Proposition 2.9, 2.16]{HLR17} \label{reflection length and residual dimension}
	\begin{enumerate}
		\item For $U\in GL_n(q)$, the reflection length and residual dimension are equal:
		$l(U)=\dim R^{U}=\codim V_n^{U}$.
		\item The reflection length is sub-additive: i.e. for $U_1,U_2\in GL_n(q)$
		\[
		l(U_1U_2)\leq l(U_1)+l(U_2).
		\]
		\item \label{intersection equality for the top coefficients} If $l(U_1U_2)=l(U_1)+l(U_2)$ then
		\begin{equation}
		\text{$V_n^{U_1}\cap V_n^{U_2}=V_n^{U_1U_2} \qquad$ and $\qquad V_n=V_n^{U_1}+V_n^{U_2}$.}
		\end{equation}
	\end{enumerate}
\end{Lemma}

\begin{Lemma}\cite[Lemma 3.2]{WW18}\label{reflection length and stable type}
	The reflection length is stable under the embedding $G_m\subseteq G_{n}$ for all $m,n\in \N$ satisfying $m\leq n$. Moreover:
	\begin{enumerate}
		\item\label{reflection length equal to weight} If the modified type of $U$ is $\mu$, then $l(U)=||\mu||$.
		\item If the modified types of $U_1,U_2,U_1U_2\in G_{\infty}$ are $\lambda,\mu,\nu$ respectively, then
		\[
		||\lambda||+||\mu||\leq ||\nu||.
		\]
	\end{enumerate}
\end{Lemma}

Proposition \ref{centralizer growth in gl}, Lemma \ref{reflection length and stable type} and Lemma \ref{reflection length and residual dimension} are sufficient to prove that the index function
\begin{equation}\label{the index function}
n\longmapsto \frac{|C_{GL_n(q)}(U_1U_2)|}{|C_{GL_n(q)}(U_1)\cap C_{GL_n(q)}(U_2)|}
\end{equation}
is independent of $n$ if
\[
||\lambda||+||\mu||= ||\eta||,
\] where $\lambda,\mu$ and $\eta$ are stable types of $U_1,U_2$ and $U_1U_2$, respectively. However, to
prove that the structural function $c_{\lambda,\mu}^{\eta}(n)$ is indeed independent of $n$
requires to know that there are only finitely many index functions which contribute to the
structural function $c_{\lambda,\mu}^{\eta}(n)$ and this is equivalent to show that the fibers
$V(\lambda\times \mu:\eta)$ admits only finitely many orbits with respect to the simultaneous
conjugation. Such result relies on the \textit{normal form} results of Wan and Wang:

\begin{Lemma}\label{normal form lemma}\cite{WW18}
	Let $U_1,U_2\in GL_n(q)$ and $l(U_1U_2)=l(U_1)+l(U_2)$. Moreover, let $T\in GL_n(q)$ be such that
	\begin{equation}
	TU_1U_2T^{-1}=\begin{bmatrix}
	\overline{U_1U_2} & 0 \\
	0 & I_{n-l(U_1U_2)}\\
	\end{bmatrix}
	\end{equation}
	then
	\begin{equation}
	\text{$TU_1T^{-1}=\begin{bmatrix}
		\overline{U_1} & 0 \\
		0 & I_{n-l(U_1U_2)}\\
		\end{bmatrix},\quad$ and $\quad TU_2T^{-1}=\begin{bmatrix}
		\overline{U_2} & 0 \\
		0 & I_{n-l(U_1U_2)}\\
		\end{bmatrix}$.}
	\end{equation}
\end{Lemma}

\begin{Remark}
	Wan and Wang do not present this last lemma as an isolated entity but produce it as a by product of the proof of the proposition below. We, instead, present it independently because we will use it in the context of symplectic groups.
\end{Remark}

\begin{Proposition}[Normal Form Theorem]\cite[Proposition 3.3]{WW18}\label{normal form}
	Let $U_1,U_2,U_1U_2\in G_{\infty}$ and $\pmb{\lambda},\pmb{\mu},\pmb{\eta}$ be their modified types respectively. Suppose $||\pmb{\eta}||=||\pmb{\lambda}||+||\pmb{\mu}||$ and set $m=||\nu||+l(\nu(t-1))$. Then there exists $T\in GL_n(q)$ and $\overline{U_1},\overline{U_2}\in G_k$ such that
	\begin{equation}
	TU_1T^{-1}=\begin{bmatrix}
	\overline{U_2} & 0 \\
	0 & I_{n-m}
	\end{bmatrix}, \qquad
	TU_2T^{-1}=\begin{bmatrix}
	\overline{U_2} & 0 \\
	0 & I_{n-m}
	\end{bmatrix},\qquad
	TU_1U_2T^{-1}=\begin{bmatrix}
	\overline{U_1}\overline{U_2} & 0 \\
	0 & I_{n-m}
	\end{bmatrix}.
	\end{equation}
\end{Proposition}

\begin{Corollary}
	The simultaneous conjugation admits finitely many orbits. Hence $(GL_n(\F_q))_{n\in \N}$ is a uniformly saturated family.
\end{Corollary}


The following theorem is the stability property of the uniformly saturated family $(GL_n(q))_{n\in \N}$ and it is proved using the previous results as analogs of them used to prove the stability result for the uniformly saturated family $(S_n)_{n\in \N}$.

\begin{Theorem}[Stability Theorem]\cite[Theorem 3.4]{WW18}\label{key}
	Let $\pmb{\lambda}$, $\pmb{\mu}$, $\pmb{\eta}$ be three elements of $\mc{P}(\Phi)$. If $||\pmb{\eta}||=||\pmb{\lambda}||+||\pmb{\mu}||$, then $c_{\pmb{\lambda},\pmb{\mu}}^{\pmb{\eta}}(n)$ is a non-negative integer independent of $n$.
\end{Theorem}

\section{The case of symplectic groups: $Sp_n(q)$}\label{C:4}

In this chapter, we start dealing with the case of symplectic groups. In the first section the basics of symplectic spaces and alternating forms are discussed. In the subsequent section a detailed review of conjugacy in symplectic groups is presented. The results of the second section are used to obtain a rational form for the unipotent symplectic matrices. In the fourth section the reviewed theory is used to construct the uniformly saturated family $Sp_n(q)$. Finally, the main theorem, the stability property of center of the symplectic group rings is proved assuming Theorem \ref{growth of centralizer in symplectic group- main statement} whose proof is deferred to the next chapter.


\subsection{Review of symplectic groups}

This section presents the basic properties of the symplectic groups $Sp_{n}(q)$ over finite field with $q$ elements. The main reference for this section are the books \textit{Symplectic Groups} by O.T.
O'meara \cite{Om78} and \textit{Linear Algebra and Geometry, a seconds course}, by I. Kaplansky, \cite{Kap74},

Let $V$ be an $\F_q$ vector space of dimension $n$, where $q$ is an odd prime power. An
\textbf{alternating form} (or symplectic form) $Q(\cdot,\cdot)$ on $V$ is a map $V\times V\longmapsto \F_q$ such that for all $u,v, w\in V$ and $a\in \F_q$, the equalities
\begin{enumerate}
	\item $Q(v,w)=-Q(w,v)$, (alternating property)
	\item $Q(av+u,w)=aQ(v,w)+Q(u+w)$, (bilinearity)
\end{enumerate} hold. If $Q$ is an alternating form on $V$ then the pair $(V,Q)$ is called a \textbf{symplectic space}. Given
two symplectic spaces $(V_i,Q_i)$, $i=1,2$, over $\F_q$ are called \textbf{equivalent} if there is a bijective linear map $\phi:V_1\longrightarrow V_2$
such that
\[
Q_2(\psi(v),\psi(w))=Q_1(v,w),
\] for all $v,w\in V_1.$ In the case of equality $V_1=V_2$, one speaks of the equivalency of $Q_1$ and $Q_2$ and drop the underlying vector space from the notation. As done for all bilinear forms, the
effect of $Q(\cdot,\cdot)$ on $V\times V$ can be written in terms of matrices. Let
$B=\{e_1,\cdots,e_n\}$ be a fixed ordered basis of $V$ and let $[S_Q]_B$ be the $n\times n$ matrix
$(s_{ij})_{i,j=1}^n$ where
\[
s_{ij}=Q(e_i,e_j).
\]
The matrix $[S_Q]_B$ is a skew symmetric in the sense that, $S_Q^{tr}=-S_Q$, as a consequence of the fact that
$Q$ is alternating. Let $v,w\in V$ be two elements that are considered as column
vectors written with respect to the ordered basis $\{e_1,\cdots,e_n\}$. Then it is easily seen that
\[
Q(v,w)=v^{tr}\cdot [S_Q]_B\cdot w.
\]
Two elements $v,w\in V$ are said to be \textbf{orthogonal} to each other, denoted as $v\perp w$, if $Q(v,w)=0$. Similarly, two
subspaces $W_1,W_2\subset V$ are said to be \textbf{orthogonal} to each other if for all $w_1\in
W_1$, $w_2\in W_2$, $Q(w_1,w_2)=0$. The orthogonality of subspaces again denoted by the notation $W_1\perp
W_2$. For a subspace $W\subset V$, the subspace of elements that are orthogonal to $W$ is
$W^{\perp}:=\{v\in V:v\perp w,\forall w\in W\}$. A symplectic space $(V,Q)$ is said to be \textbf{non-degenerate} if $V^{\perp}=0$. The non-degeneracy of a
form $Q$ is equivalent to non-vanishing of $\det(S_Q)$, which is independent of the
chosen basis.
A \textbf{hyperbolic pair} $(e,f)$ with respect to $Q$ is an element of $V\times V$ with the property $Q(e,f)=1$. In this case $e$ will be referred as the \textbf{positive} part and $f$ will be referred as the \textbf{negative} part of the hyperbolic pair.

\begin{Lemma}\cite[Theorem 1.1.13]{Om78}\label{basics of symplectic spaces}
	Let $(V,Q)$ be a symplectic space. Then the following are equivalent:
	\begin{enumerate}
		\item $Q$ is non-degenerate.
		\item $V$ admits an ordered basis $\{e_1,e_2,\cdots,e_n,f_n,f_{n-1},\cdots f_1\}$ where $(e_i,f_i)$ is a hyperbolic pair for $i\in\{1,\cdots, n\}$, such that $H_i\perp H_j$ for $i\neq j\in\{1,\cdots, n\}$, where $H_i=\langle e_i,f_i\rangle$ is the subspace generated by the hyperbolic pair $(e_i,f_i)$. With respect to this basis the matrix of $Q$ is equal to the block diagonal matrix
		\[
		Q=\begin{blockarray}{ccccccccc}
		e_1 & e_2 & \cdots & e_n & f_n & \cdots & f_2 & f_1 &\\
		\begin{block}{[cccccccc]c}
		&  &  &   &  &  &  & 1 & e_1 \\
		&  &   &  & & & 1 &  & e_2 \\
		&  &  &  &  & \iddots & & &  \\
		&    &  &  & 1 & & & & e_n\\
		&  &  &  -1 &  &  &   &    & f_n\\
		&  & \iddots &   & &  &   &   & & \\
		&  -1 & &    &  &  &  &   & f_2\\
		-1 &    &  &  &   & &  &  & f_1 \\
		\end{block}
		\end{blockarray}
		\]
	\end{enumerate}
	In particular, non-degenerate symplectic spaces must be even dimensional and if $Q_1$ and $Q_2$ are two non-degenerate symplectic forms on $V$ then they are equivalent.
\end{Lemma}

A basis $B$ satisfying 2. of Lemma \ref{basics of symplectic spaces} is called a \textbf{hyperbolic basis}. In this case $e_i$ and $f_i$ are said to be \textbf{hyperbolic conjugates of each other}. If $B$ is an hyperbolic basis, then $B^+$ denote the positive parts of hyperbolic pairs in $B$, and $B^-$ denote the negative parts of hyperbolic pairs in $B$.

Let $(V,Q)$ be a non-degenerate symplectic space. An element of $U\in GL(V)$ is said to be a \textbf{symplectic transformation} if
\begin{equation}\label{definition symplectic transformation}
Q(Uv,Uw)=Q(v,w)
\end{equation}for all $v,w\in V$. The set of symplectic transformations form a group
which is called the symplectic group and denoted by $Sp(V)$. It is contained in the
special linear group $SL(V)$ of linear transformations with determinant $1$
(\cite{Om78}, Thm. 2.1.110). For an element $U\in GL(V)$, whether or not $U$ is a symplectic
transformation can be checked via hyperbolic bases. Let
$\{e_1,f_1,\cdots,e_n,f_n\}$ be a hyperbolic basis for $(V,Q)$ and $U\in GL(V)$. Then
$U$ is an element of $Sp(V)$ if and only if $\{Ue_1,Uf_1,\cdots,Ue_n,Uf_n\}$ is a
hyperbolic basis.

\subsection{Conjugacy classes in $Sp_{n}$}

In this section, the references that we follow are \textit{On isometries of inner product space} by J. Milnor  \cite{Mil69}, and \textit{Conjugacy Classes} by Springer-Steinberg in \cite{Bor70}. Since these results are not comprehensively covered in text books, we will present a thorough discussion.

Let $(V,Q)$ be a non-degenerate symplectic space of dimension $2n$. By Proposition \ref{Conjugacy classes of GL-n}, conjugacy classes of $GL(V)$
are parameterized by the partition valued functions $\pmb{\lambda}:f\longmapsto
\lambda(f)=(\lambda_1,\cdots,\lambda_{r_f})$ on the $\Phi$, which are of weight $2n$:
\begin{equation}\label{weight of a partition valued function in sp-2n}
2n=||\pmb{\lambda}||=\sum_{f\in \Phi}\deg f\cdot ||\lambda(f)||= \sum_{f\in \Phi}\deg f\cdot \big( \sum_{i=1}^{r_f} \lambda_i\big)
\end{equation}
However, if one considers elements $U\in Sp(V)$, then one can not realize all the partition valued functions of weight $2n$. This is not the only obstacle. Namely, let $U_1,U_2$ be two isometries and assume that
$\lambda^{U_1}=\lambda^{U_2}$. Then it is known that $U_ 1$ and $U_ 2$ are conjugate only over a suitable
extension $F$ over $\mathbb{F}_q$, (cf. \cite{Kap74}, Theorem 70, pg. 79), which means for a fixed $m$, the family $(Sp_m(q^n))_{n\in\N}$ is not saturated.

Let $U\in Sp(V)$ and $V_U$ denotes $\mathbb{F}_q[t]$-module whose underlying space is $V$, on which $t$ acts
as $U$. i.e. $t\cdot v=Uv$. Let $m_{U}(t)$ denotes the minimal polynomial of $U$ and introduce the module $A(U)=\mathbb{F}_q[t]/(m_U(t))$. From  the fact that
$Q(Uv,w)=Q(v,U^{-1}w)$ and the bilinearity of $Q$ it follows that for every polynomial
$f\in\mathbb{F}_q(t)$ one has
\begin{equation}\label{polynomial adjointness}
Q(f(U)v,w)=Q(v,f(U^{-1})w).
\end{equation}
Substituting $m_U$ in the equation Eq.\eqref{polynomial adjointness} one gets
\[
0=Q(0\cdot v,w)=Q(m_U(U) v,w)=Q(v,m_U(U^{-1})w),
\]
$\forall v,w\in V$. Since the form $V$ is non-degenerate, it follows that $m_U(U^{-1})=0$ and thus
the minimal polynomial of $U^{-1}$ divides that of $U$. By symmetry and the fact that both polynomials
are monic, it follows that $m_U(t)=m_{U^{-1}}(t)$. As a result, the map
\begin{equation}\label{involution of A(U)}
\sigma:U\longmapsto U^{-1}
\end{equation} induces an
isomorphism on $A(U)=\mathbb{F}_q[t]/(m_U(t))$, which is obviously an involution.

\begin{Definition}\label{dual definition}
	For $f=a_0+a_1t+\cdots+t^d\in \Phi$, introduce the \textbf{dual} $\overline{f}\in\mathbb{F}_q(t)$ by
	\begin{equation}\label{dual of a polynomial}
	\overline{f}(t)=\sum_{i=0}^{d}(a_ia_0^{-1})t^{d-i}.
	\end{equation} A self-dual polynomial $f$ is called
	\textbf{dual-irreducible} if f is either irreducible or
	$f=g\overline{g}$ where $g$ is an irreducible polynomial that is not self-dual.
	Denote the set of dual irreducible polynomials with $\Phi^s$.
\end{Definition}

\begin{Remark}
	It is straightforward that $\overline{fg}=\overline{f}\overline{g}$, hence, if $f$ is an irreducible polynomial then its dual $\overline{f}$ is also irreducible. It is also clear that a self-dual polynomial is a product of dual-irreducible polynomials.
\end{Remark}

\begin{Lemma}
	If $U\in Sp_n$ then the minimal polynomial $m_U(t)$ of $U$ is self-dual. In particular, $m_{U}(t)$ is a product of dual-irreducible polynomials.
\end{Lemma}

\begin{proof}
	We start with noticing the following relation between the automorphism $\sigma$ of $ A(U)$ sending $U$ to $U^{-1}$, and the dual operation defined on polynomials (cf. Eq.\eqref{dual of a polynomial}):
	\begin{eqnarray}
	\sigma(f(U)) & = & f(U^{-1})\nonumber\\
	& = & \Big(\sum_{i=0}^{d}a_iU^{-i}\Big)(a_0^{-1}U^d)(a_0U^{-d})\nonumber\\
	& = & a_0U^{-d}\sum_{i=0}^{d}(a_ia_0^{-1})U^{d-i}\nonumber\\
	& = & a_0U^{-d}\overline{f}(U).\label{composition of the involution with dual}
	\end{eqnarray}
	Invoking this observation in Eq. \eqref{polynomial adjointness} and taking $f(t)=m_U(t)$ yields
	\begin{equation}\label{minimal polynomials of isometries are self-dual}
	0=Q(m_U(U)v,w)=Q(v,a_0U^{-d}\overline{m_U}(U)w)=Q(U^{d}\cdot v,a_0\overline{m_U}(U)\cdot w).
	\end{equation}
	As $U$ is invertible and $Q$ is non-degenerate, it follows that $\overline{m_U}(U)=0$. The desired equality now follows from the equality of the degrees.
\end{proof}

\begin{Lemma}
	If $f_1$, $f_2$ are distinct monic irreducible factors of
	$m_U$, the minimal polynomial of $U\in Sp_n(q)$, then the generalized eigenspaces $V_{f_i}=\{v\in V: f_i
	^k(U)v=0,\; \text{for large $k$} \}$ for $i=1,2$ are orthogonal to each other
	unless $\overline{f_1}=f_2$.
\end{Lemma}
\begin{proof}
	Let $k$ be such that $f_1
	^k(U)v=0$ for all $v\in V_{f_1}$. Then, for all $v_i\in V_{f_i}$, $i=1,2$ one
	gets
	\[
	0=Q(0,v_2)=Q(f_1^k(U)v_1,v_2)=Q(v_1,f_1^k(U^{-1})v_2)=Q(v_1,a_0^kU^{-dk}\overline{f_1}^k(U)v_2).
	\]
	Next we assume that $\overline{f_1}\neq f_2$. As $\overline{f_1}, f_2$ are both irreducible, it follows that $\overline{f_1}^k$ and $f_2$ are coprime and there exist $h_1,h_2\in \F_q[t]$ such that $h_1\overline{f_1}^k+h_2f_2=1\in F_q[t]$. As the action of $h_2f_2(U)$ on $V_{f_2}$ is zero, it follows that, on $V_{f_2}$ we have $h_1\overline{f_1}^k(U)=1$, in particular it acts as an automorphism of $V_{f_2}$, so does $U^{-dk}\overline{f_1}^k(U)$. This finishes the proof.
\end{proof}

Let $U\in Sp_n(q)$. Let $f(t)$ be a dual-irreducible divisor of $ m_U(t)$. If $f$ is irreducible, set $W_f$ to be $V_f$ (the generalized eigenspace of $f$) and if $f=g\overline{g}$ for some irreducible non-self-dual polynomial $g$, then set $W_f$ as the subspace $V_g\oplus V_{\overline{g}}$. With this notation, the above findings can be packed into the following proposition. Recall that $\Phi^s$ is defined to be the set of dual-irreducible polynomials in $\F_q[t]-\{t\}$. 

\begin{Lemma}\cite{Mil69}\label{orthogonal decomposition of an isometry}
	For each dual-irreducible divisor $f$ of $m_U(t)$, the subspace $W_f$ is a
	non-degenerate symplectic space and $V$ is equal to the orthogonal sum of $W_f$'s,
	as $f$ ranges over dual-irreducible factors of $m_U(t)$. In particular, the
	restriction $U_{|W_f}$ is an isometry of $W_f$ and $V$ admits the following orthogonal sum of invariant subspaces:
	\begin{equation}
	V=\bigoplus_{\substack{f(t)\in \Phi^s \\ f(t)|m_U(t)}}W_f.
	\end{equation}
\end{Lemma}

\begin{Proposition}\cite{Mil69}\label{orthogonal decomposition of V-n}
	Let $U_1,U_2$ be two isometries of $V$. The isometries $U_1$ and $U_2$ are
	conjugate in $Sp_n(q)$ if and only if
	\begin{enumerate}
		\item $\lambda_{U_1}=\lambda_{U_2}$,
		\item The isometries $(U_1)_{|W_f}$ and $(U_2)_{|W_f}$ are conjugate in
		$Sp(W_f)$, for $f=t\pm 1$.
	\end{enumerate}
	In particular, the $Sp$ conjugacy class of $W_f$ for $f\neq t\pm 1$ is completely determined by the Jordan form.
\end{Proposition}
\begin{proof}
	For $f\neq t\pm 1$ self-dual, see the proof of Theorem 3.2 in \cite{Mil69}. For $f$ non-self-dual, see the second paragraph following Theorem 3.4 in ibid.
\end{proof}

This reduces the study of conjugacy classes into the study of conjugacy classes of elements $U$ such that the polynomial $m_U(t)$ is a power of $(t\pm 1)$.

\begin{Theorem}\cite[Theorem 3.2]{Mil69}\label{milnor orthogonal splitting}
	Let $U$ be an isomorphism , and $W_{t\pm 1}$ be as in Lemma \ref{orthogonal decomposition of an isometry}. The space $W_{t\pm 1}$ admits an orthogonal decomposition \[V_U=W_{t\pm1}^1\perp \cdots
	\perp W_{t\pm1}^{r}\]
	where $W_{t\pm1}^{i}$ is a free $\F_q[t]/(t\pm
	1)^{m_i}$-module and $\lambda(t\pm 1)=(1^{m^ {\pm}_1},\cdots,r^{m^{\pm}_r})$.
\end{Theorem}

\begin{proof}(Sketch)
	Consider a not necessarily orthogonal decomposition of $V_{U}$ as in statement of the lemma. Then the restriction $Q_{\mid W_{t\pm 1}^r}$  of the
	inner product $Q$ to  $W_{t \pm 1}^r$ is non-degenerate \cite[Lemma
	1.4.6]{Wa63}, \cite[Theorem 3.2]{Mil69}. So we can consider the orthogonal decomposition of
	$V_{U}=W_{t\pm 1}^r\oplus (W_{t \pm 1}^r)^{\perp}$ and continue by induction.
\end{proof}

\begin{Theorem}\cite[Theorem 3.4]{Mil69} \label{orthogonal decomposition of V-n}
	We keep the notation and the assumptions of the previous Theorem.
	\begin{enumerate}
		\item For each $i=1,\cdots, r$, there exists a vector space $H_i^{\pm}$ and a bilinear form $h^{\pm}_i$ on $H_i^{\pm}$, called the Wall form.
		\item The dimension of $H^{\pm}_i$ is $m^{\pm}_i$,
		where $h^{\pm}_{i}$  is a non-degenerate symplectic form for odd $i$, and $h^{\pm}_{i}$ is a
		symmetric bilinear form for even $i$.
		\item The equivalence classes of $(h_i^{\pm})_{i}$ completely determine the $Sp_m(q)$ conjugacy
		classes of $x_{|W_{t\pm1}}$.
	\end{enumerate}
\end{Theorem}
\begin{Remark}\label{construction of the Wall forms}
	Following Milnor (cf. \cite[Section 3]{Mil69}), we will recall the construction of the vector spaces $H_i^{-}$ and the
	definition of the Wall forms $h_i^-$ for a fixed $i$, hence we restrict ourselves to
	the case $m_U(U)=(t-1)^{s}$, i.e. to the unipotent $U$ case. Let $A(U):=\F_q[t]/(t-1)^{s}$ and $\Delta=t-t^{-1}$,
	where $t$ is the image of $U$ in $A(U)$. Introduce
	$H_i^-:=W^i_{t-1}/(U-I)W^i_{t-1}$. The subspace $W_{t-1}^i$ is a free $A(U)$-module, hence equal to direct sum
	of cyclic modules $T^i_1,\cdots, T^i_r$, for some $r>0$. Since $T_j$ is a cyclic
	module, there exists $v_j\in T_j$ such that the translates $v_j,Uv_j,U^2v_j,\cdots $
	generate $T_j$. Then, it follows that $\overline{T_j}\subset
	H_i^-=W^i_{t-1}/(U-I)W^i_{t-1}$ is generated by $\overline{v_j}$, and hence
	\[H_i^-=\oplus_{j=1}^k\langle \overline{v_j}\rangle.\] The association
	\[
	h_i^-(\overline{v},\overline{w})=Q(\Delta^{i-1}v,w), \forall \overline{v},\overline{w}\in H_i^-
	\]
	is well-defined and defines bilinear form on $H^-_i$. According to the theorem, it is a symplectic non-degenerate form for odd $i$ and symmetric non-degenerate form for even $i$. As, over a given vector space, all non-degenerate symplectic forms are isomorphic, one can take $h^-_i=-1$ for $i$ odd. Likewise, as non-degenerate symmetric bilinear forms over $\F_q$ are parameterized by $\F_q^{\times}/(\F_q^{\times})^2$, for even $i$ we have $h^-_i$ is equal to $+1$ or $-1$.
\end{Remark}

\begin{Definition}
	\begin{enumerate}
		\item A \textbf{signed partition} is a couple $(\lambda,h)$ such that
		$\lambda=(\lambda_1,\cdots,\lambda_r)$ is an ordinary partition and $h=(h_1,\cdots,h_r)\in \{-1,+1\}^r$ satisfying the following property: if $\lambda_i=\lambda_j$ then $h_{i}=h_j$. 
		\item  The \textbf{weight} $||(\lambda,h)||$ of a signed partition $(\lambda,h)$ is defined
		as the weight $||\lambda||$ of the underlying partition.
	\end{enumerate}
	
\end{Definition}

\begin{Remark}
	One can write a signed partition in
	the form $\lambda=(1^{(m_1,-)},2^{(m_2,\pm)},\cdots)$. For example, if $(\lambda,h)=((6,6,2,2,2,2,1,1,1),(-,-,+,+,+,+,-,-,-))$ then one can write $(\lambda,h)$ as
	$(1^{(3,-)},2^{(4,+)},6^{(2,-)})$. Also observe that
	the weight of a symplectic partition is always an even integer.
\end{Remark}

\begin{Definition}
	\begin{enumerate}
		\item A signed-partition $(1^{(m_1,h_1)},2^{(m_2,h_2)},\cdots)$ is called a \textbf{symplectic
			partition} if for odd $i$, $m_i$ is even and $h_i=-1$. The set of symplectic partitions is denoted by $\mathcal{P}^s$. 
		\item A symplectic partition valued function (simply, a symplectic function) is a triple $(\pmb{\lambda},h^+,h^-)$, where
		$\pmb{\lambda}$ is a partition valued function defined on $\Phi^s$, and $(\pmb{\lambda}(t- 1),h^{-})$, $(\pmb{\lambda}(t+ 1),h^{+})$
		are symplectic partitions. The weight of such a function is defined as the weight of
		the underlying partition valued function. The set of symplectic partition valued functions of weight $2m$ is denoted by $\mc{P}^s_{2m}(\Phi^s)$ and the set of all symplectic partition valued functions is denoted by $\mc{P}^s(\Phi^s)$.
	\end{enumerate}
\end{Definition}

With this notation, we can rephrase Theorem \ref{orthogonal decomposition of V-n} as follows.

\begin{Corollary}\label{symplectic parametrization of conjugacy classes}
	\cite[Theorem 1.20]{Shi80}   The conjugacy classes in $Sp_m(q)$
	are parameterized by the symplectic partition valued functions of
	weight $2m$. If $(\pmb{\lambda},h^+,h^-)$ is the symplectic partition
	valued function that corresponds to the isometry $U$, then the
	underlying partition valued function $\pmb{\lambda}$ is equal to
	$\pmb{\lambda}^U$, when viewed as an element of $GL_{2m}(q)$. The symplectic function $(\pmb{\lambda},h^+,h^-)$ is called the \textbf{symplectip type} of $U$.
\end{Corollary}

\subsection{Rational forms for unipotent blocks in $Sp_{n}(q)$}

Following \cite{GLO}, we introduce a family of matrices what will serve as rational forms for unipotent matrices in the symplectic groups. Introduce the matrices $S_m$ for $m\in \N$ are defined as follows. First recall that the matrices


\begin{equation}
S_m:=\begin{bmatrix}
1 &  &  &     \\
1 & 1 &  &   \\
\vdots &  \vdots & \ddots &    \\
1 & 1  & \cdots & 1 \\
\end{bmatrix}, \qquad S_m^{-1}=\begin{bmatrix}
1 &  &     &  \\
-1 & 1 &     &  \\
& \ddots & \ddots &   \\
&   & -1 & 1
\end{bmatrix}.
\end{equation}
were defined earlier. Clearly, the minimal and characteristic polynomials of $S_m$ and $S_m^{-1}$ are equal to $(t-1)^n$.  Now introduce the matrices
\[
J_{2m}=\begin{blockarray}{ccccccccc}
e_1 & e_2 & \cdots & e_m & f_m & \cdots & f_2 & f_1 &\\
\begin{block}{[cccccccc]c}
1 &  &  &   &  &  &  &  & e_1 \\
1 & 1 &   &  & & & &  & e_2 \\
\vdots &  \vdots & \ddots &  &  & & & &  \\
1 &  1  & \cdots & 1 & & & & & e_m\\
&  &  &  & 1 &  &   &    & f_m\\
&  & &   & -1 & 1 &   &   & & \\
&   & &    &  & \ddots & \ddots &   & f_2\\
&    &  &  &   & & -1 & 1 & f_1 \\
\end{block}
\end{blockarray}
\]
and for $\epsilon\neq 0$
\[
J_{2m,\epsilon}=\begin{blockarray}{ccccccccc}
e_1 & e_2 & \cdots & e_m & f_m & f_{m-1} & \cdots & f_1 &\\
\begin{block}{[cccccccc]c}
1 &  &  &   &  &  &  &  & e_1 \\
1 & 1 &   &  & & & &  & e_2 \\
\vdots &  \vdots & \ddots &  &  & & & & \vdots \\
1 &  1  & \cdots & 1 & & & & & e_m\\
\epsilon & \epsilon & \cdots  & \epsilon & 1 &  &   &    & f_m\\
&  & &   & -1 & 1 &   &   & f_{m-1} \\
&   & &    &  & \ddots & \ddots &   & \vdots \\
&    &  &  &   & & -1 & 1 & f_1 \\
\end{block}
\end{blockarray}
\]
written with respect to the ordered hyperbolic basis
$\{e_1,e_2,\cdots,e_m,f_m,\cdots,f_2,f_1\}$. The matrices of the form $J_{2m}$ will be called $2m$-\textbf{dimensional symplectic blocks} and matrices of the form $J_{2n,\epsilon}$ will be called an $2m$-\textbf{dimensional orthogonal blocks}. The matrices $J_{2m}$ and
$J_{2m,\epsilon}$ are elements of the symplectic group, which can be readily seen by checking the equality
\[Q(C_u(J_{2m}),C_v(J_{2m}))\] as $u,v$ ranges over
$B$. The minimal polynomial of $J_{2m}$ is equal to the minimal polynomial
$m_{S_m}(t)=m_{S_m^{-1}}(t)=(t-1)^m$ of $S_m$ and the minimal polynomial of
$J_{2m,\epsilon}$ is equal to $(t-1)^{2m}$. In particular, $1$ is the unique eigen-value in both
cases. Notice also that $J_2=I_2$ and no other $J_{2m,\epsilon}$ satisfies such an equality.

\begin{Remark}\label{Remark eigen vectors of unipotent blocks}
	When $U$ is an $m\times m$ matrix, we will view $U$ as a linear operator of $V=\F_q^m$ in
	two ways: Let $v=(v_1,\cdots, v_m)\in \F_q^n$
	\begin{enumerate}
		\item The association $v\longmapsto v\cdot U$ is called the \textbf{right action} of $U$. The fixed space of this action is denoted by $^UV$. The following identities are obvious:
		\[  ^{J_{2m,0}}V=\langle e_1,f_m\rangle,\quad ^{J_{2m,\epsilon}}V=\langle e_1\rangle\]
		\item The association $v^t\longmapsto  U\cdot v^t$ is called the \textbf{left action} of $U$. The fixed space of this action is denoted by $V^U$. The following identities are obvious:
		\[  V^{J_{2m,0}}=\langle e_m,f_1\rangle,\quad V^{J_{2m,\epsilon}}=\langle f_1\rangle \]
	\end{enumerate}
	In case of a symplectic block, the space $V$ splits off into two cyclic spaces with cyclic vectors $e_1$ and $f_m$. And in case of an orthogonal block, the space $V$ contains $e_1$ as a cyclic vector.
\end{Remark}

\begin{Remark} When the rows and columns of a matrix are labeled with bases elements, then we consider the matrix as a linear operator in two different ways, as described in the previous remark. In this case, we will consider both rows and columns of the matrix as vectors of the appropriate vector space determined by the bases.
\end{Remark}

Our next aim is to show that each symplectic unipotent conjugacy class is realized as the orthogonal sums of suitable symplectic and orthogonal blocks. To this end, we will investigate the $\F_q{[t]}$-module structures on $V$ that are induced by $J_{2m}$ and $J_{2m,\epsilon}$. More precisely, we will investigate the induced bilinear forms $h_i$, as explained in Remark \ref{construction of the Wall forms}.

Let $U=J_{4k+2,0}$, which acts on the symplectic space $V_{4k+2}$. The minimal
polynomial of $U$ is $(t-1)^{2k+1}$ and $V_{4k+2}$ is equal to the direct sum of two
cyclic $\F_q[t]/(t-1)^{2k+1}$-modules $T_1:=\langle e_1,\cdots,e_{2k+1}\rangle $ and
$T_2:=\langle f_1,\cdots,f_{2k+1}\rangle $. So, $W^{2k+1}_{t-1}=V_{4k+2}$ and $W^{i}_{t-1}=0$ for
$i\neq 2k+1$. The subspace $T_1$ (resp. $T_2$) is generated by the $U$ translates of
$e_1$ (resp. $f_{2k+1}$). Recall that $\Delta$ is defined as $U-U^{-1}$. Thus we
have\[\Delta=
\begin{blockarray}{ccccccccc}
e_1 & \cdots & e_{2k} & e_{2k+1} & f_{2k+1} & \cdots & f_2 & f_1 &\\
\begin{block}{[cccccccc]c}
0 &  &  &   &  &  &  &  & e_1 \\
2 & 0 &   &  & & & &  & e_2 \\
\vdots &  \ddots & \ddots &  &  & & & & \vdots \\
1 &  \cdots  & 2 & 0 & & & & & e_{2k+1}\\
&  &   &  & 0 &  &   &    & f_{2k+1}\\
&  & &   & -2 & 0 &   &   &  \\
&   & &    & \vdots & \ddots & \ddots &   & \vdots \\
&    &  &  &  -1 & \cdots & -2 & 0 & f_1 \\
\end{block}
\end{blockarray}
\]
and hence
\[
\Delta^{2k}=\begin{blockarray}{ccccccc}
e_1 & \cdots & e_{2k+1} & f_{2k+1} & \cdots & f_1 &\\
\begin{block}{[cccccc]c}
&  &  &     &  &  & e_1 \\
&   &  &  &  & & \vdots \\
2^{2k} &    &  &   & & & e_{2k+1}\\
& &  & &  &      & f_{2k+1}\\
&   &   &  &  &   & \vdots \\
&    &  &  2^{2k}  &  &  & f_1 \\
\end{block}
\end{blockarray}
\]
The space $H^-_{2k+1}$ is generated by $\{2^{-2k}\overline{e_1},\overline{f_{2k+1}}\}$ and
\[h^-_{2k+1}(2^{-2k}\overline{e_1},\overline{f_{2k+1}})=Q(\Delta^{2k}(2^{-2k}e_1),f_{2k+1
})=Q(e_m,f_m)=1.\] This means $H^-_{2k+1}$ is a non-degenerate symplectic space with
hyperbolic basis $\{2^{-2k}\overline{e_1},\overline{f_{2k+1}}\}$. In particular, the symplectic type
$(\pmb{\lambda}^U,h^+,h^-)$ of $U$ can be described as follows. For $f\neq t-1$,
$\pmb{\lambda}(f)=\emptyset$, the empty partition, and
$\pmb{\lambda}(t-1)=(2k+1,2k+1)=((2k+1)^2)$. As $\pmb{\lambda}(t+1)$ is the empty
partition, $h^+_i$ is a sequence of length zero. The sign corresponding to $2k+1$ is
$-1$ as the sign is determined by the isomorphism class of $h^-_{2k+1}$, which is a non-degenerate symplectic form. So,
$\pmb{\lambda}(t-1)=((2k+1)^{(2,-)})$. With this point of view, if $U=\oplus_{i=1}^r
a_iJ_{4i+2}$, where the direct sum is the usual orthogonal sum and $a_i$'s are allowed
to be zero, then $\pmb{\lambda}(t-1)=(1^{(2a_1,-)},\cdots,(2r+1)^{(2a_{2r+1},-)})$.

Next we consider the case $U=J_{2k,\epsilon}$, where $\epsilon\in \F_q^{\times}$, with its action on $V_{k}$. The minimal polynomial of $U$ is $(t-1)^{2k}$ and thus $W^{i}_{t-1}=0$ for $i\neq 2k$, consequently, $W^{2k}$ is equal to the ambient space $V_{k}$. The space $V_{k}=W^{2k}_{t-1}$ is generated by the $U$ translates of the cyclic vector $e_1$, so $H^-_{2k}$ is generated by the image of $e_1$ in $H^-_{2k}$. We also have
\[\Delta=
\begin{blockarray}{ccccccccc}
e_1 & \cdots & e_{k-1} & e_{k} & f_{k} & \cdots & f_2 & f_1 &\\
\begin{block}{[cccccccc]c}
0 &  &  &   &  &  &  &  & e_1 \\
2 & 0 &   &  & & & &  & e_2 \\
\vdots &  \ddots & \ddots &  &  & & & & \vdots \\
1 &  \cdots  & 2 & 0 & & & & & e_{k}\\
\epsilon & \epsilon & \epsilon  & 2\epsilon  & 0 &  &   &    & f_{k}\\
&  & &  \epsilon & -2 & 0 &   &   &  \\
&   & & \vdots   & \vdots & \ddots & \ddots &   & \vdots \\
&    &  & \epsilon &  -1 & \cdots & -2 & 0 & f_1 \\
\end{block}
\end{blockarray}
\]
and
\[
\Delta^{2k-1}=\begin{blockarray}{ccccccc}
e_1 & \cdots & e_{k} & f_{k} & \cdots & f_1 &\\
\begin{block}{[cccccc]c}
&  &  &     &  &  & e_1 \\
&   &  &  &  & & \vdots \\
&    &  &   & & & e_{k}\\
& &  & &  &      & f_{k}\\
&   &   &  &  &   & \vdots \\
\epsilon'(k) &    &  &    &  &  & f_1 \\
\end{block}
\end{blockarray}
\]
where $\epsilon'(k):=-\epsilon(k):=(-1)^{k-1}2^{2k-1}\epsilon$. As a result, we have
\[
h^-_{2k}(\overline{e_1},\overline{e_1})=Q(\Delta(e_1),e_1)=Q(-\epsilon(k)f_1,e_1)=\epsilon(k)=(-1)^{k-1}2^{2k-1}\epsilon\neq 0.
\]
whose image $\overline{\epsilon(k)}$ in $\F_q^{\times}/(\F_q^{\times})^2=\{\pm 1\}$ is equal to the discriminant of the symmetric bilinear form $h^-_{2k}$, consequently, $h^-_{2k}$ is non-degenerate. By taking $\epsilon$ to be a $1$ or a non-square, one can obtain both possible discriminant values in $\F_q^{\times}/(\F_q^{\times})^2$. This means, the symplectic type $(\pmb{\lambda},h^+,h^-)$ of $U$ is defined as follows: $\lambda(f)=\emptyset$ for $f\neq t-1$ and $\lambda(t-1)=((2k)^{(1,\overline{\epsilon(k)})})$. In order to generalize as done above, consider $U=\oplus_{i=1}^r (\oplus_{j=1}^{a_i} J_{2i,\epsilon_{ij}})$, where, as before $a_i$'s are allowed to be zero. Then
\[
\lambda(t-1)=(2^{(a_1,\prod_{j=1}^{a_1} \overline{\epsilon_{1j}})},\cdots,2r^{(a_r,\prod_{j=1}^{a_r} \overline{\epsilon_{rj}})})
\]
One can combine the investigated situations immediately and derive the following proposition:

\begin{Proposition}\label{rational forms for unipotent blocks}
	\cite[Proposition 2.3]{GLO} Let $U\in Sp(V_m)$ be a unipotent
	matrix. Then there is a hyperbolic basis  $V_m$ so that the matrix of
	$U$ in this basis is equal to the orthogonal sum of suitable
	symplectic and orthogonal unipotent blocks.
\end{Proposition}

\begin{Remark}
	When considering matrices, we will always label rows and columns with basis
	elements, hence each matrix will determine a unique endomorphism. So, if $M$ is a matrix, then one can check whether $M$ is an isometry or not by checking the equality
	\[
	Q(u,v)=Q((M(u),M(v))
	\]
	where $u,v$ range over the basis set that is used to label rows and columns. One can
	also decide whether $M$ is an isometry or not, by considering the matrix
	of $Q$ (again denoted by $Q$) with respect to the basis used to label $M$. Indeed, the question of $M$ being an isometry is equivalent to the equality
	\[
	M^tQM=Q.
	\]
	The matrix of $Q$ with respect to the basis used in the definition of symplectic/orthogonal blocks is the following:
	\[
	\small{\begin{blockarray}{ccccccc}
		e_1 & \cdots & e_{k} & f_{k} & \cdots & f_1 \\
		\begin{block}{[cccccc]c}
		&  &  &     &  &  1 & e_1 \\
		& &  &    & \iddots  & & \vdots \\
		&   &    & 1  & & & e_{k}\\
		&    &  -1 &  &     &    & f_{k}\\
		&  \iddots & &       &  &   & \vdots \\
		-1 &    &  &  &   &  & f_1 \\
		\end{block}
		\end{blockarray}}
	\]
\end{Remark}

\subsection{The uniformly saturated family $(Sp_{n}(q))_{n\in \N}$}

Let $V_{\infty}$ be an infinite dimensional $\F_q$-vector space. We will consider $V_{\infty}$ with the
ordered basis $\mathcal{B}=\{e_1,f_1,\cdots,e_m,f_m,\cdots\}$ and 
the subspace generated by $\mathcal{B}_m=\{e_1,f_1,\cdots,e_m,f_m\}$ will be denoted as $V_m$. The
hyperbolic conjugate of $w\in \mathcal{B}$ is denoted by $w'$. We endow $V$
with the unique symplectic structure where $V_m$ is a non-degenerate
symplectic space and $\mathcal{B}_m=\{e_1,f_1,\cdots,e_n,f_n\}$ is a
hyperbolic basis. For $m\leq n$, the
orthogonal complement $V_m^{\perp}$ of $V_m$ in $V_n$ is denoted by $V_{m,n}$
and its hyperbolic basis $\{e_{m+1},f_{m+1},\cdots,e_n,f_n\}$ is denoted by
$\mathcal{B}_{m,n}$. The inclusion $V_m\subset V_{n}$ induces an embedding
from
\begin{eqnarray}
(\cdot)^{\uparrow\uparrow n}:GL_{2m}(q) & \longrightarrow & GL_{2n}(q)\nonumber\\
U & \longmapsto & U^{\uparrow\uparrow n}:= \begin{bmatrix}
U & 0 \\
O & I_{2n-2m}
\end{bmatrix}
\end{eqnarray}
which carries $Sp_m(q)$ into $Sp_{n}(q)$ and thus defines a direct system of
groups. The direct limit of this system will be denoted by $Sp_{\infty}(q)$
and referred as the infinite symplectic group. The similar map from $GL_m(q)$
to $GL_n(q)$ is defined in \cite{WW18} and it is denoted by $U\longmapsto
U^{\uparrow n}$. It is clear that the map $(\cdot)^{\uparrow 2n}$ from
$GL_{2m}(q)$ to $GL_{2n}(q)$ coincides with the map
$(\cdot)^{\uparrow\uparrow n}$ defined above. The group $GL_{\infty}(q)$ is
defined in the same manner.

Recall that the weight of a symplectic function on $\Phi^s$ was defined as the weight of the underlying
partition valued function.The modification operation $\overset{\circ}{\cdot}$, completion $\overline{\cdot}$
and $n$-th completion $^{\uparrow n}$ are defined in a similar way. In particular, let
$(\pmb{\lambda},h^+,h^-)$ be a symplectic function.

\begin{Definition}
	The \textbf{weight}
	$||(\pmb{\lambda},h^+,h^-)||$ of $(\pmb{\lambda},h^+,h^-)$ is by definition
	\begin{equation}\label{weight of a signed-partition valued function}
	(\pmb{\lambda},h^+,h^-)=||\pmb{\lambda}||.
	\end{equation}
	The set of symplectic functions of weight $2n$ is denoted by $\mc{P}^s_{2n}(\Phi^s)$. The set of all symplectic functions is denoted by $\mc{P}^s(\Phi^s)$.
	For $(\pmb{\lambda},h^+,h^-)\in \mc{P}^s_{2n}(\Phi^s)$ the modification $\overset{\circ}{(\pmb{\lambda},h^+,h^-)}$ is defined by setting
	\begin{equation}\label{modification of a signed-partition valued function}
	\overset{\circ}{(\pmb{\lambda},h^+,h^-)}=(\overset{\circ}{\pmb{\lambda}},\overset{\circ}{h^+},\overset{\circ}{h^-})
	\end{equation}
	where $\overset{\circ}{h^+}=h^+_i$ and $\overset{\circ}{h^-}$ is defined as follows. First recall that $(\lambda(t-1),h^-)$ is by definition
	a symplectic partition. As a result, it can be written as
	$(1^{(m_1,\epsilon_1)},2^{(m_2,\epsilon_2)},\cdots, r^{(m_r,\epsilon_r)})$ where $\epsilon_i=\pm 1$
	and for odd $i$, $m_i$ is even and $\epsilon_i=1$. The modified partition $\overset{\circ}{\lambda}$
	is then equal to $(1^{m_2},\cdots, (r-1)^{m_r})$. So we define
	$\overset{\circ}{h^-}=(h_1,\cdots,h_{r-1})$ where $h_i=\epsilon_{i+1}$ for $i=1,\cdots,r-1$. In
	particular, the resulting signed partition $(\overset{\circ}{\lambda},\overset{\circ}{h^-})$ can be
	written as $(1^{(m_2,\epsilon_2)},\cdots, (r-1)^{(m_r,\epsilon_r)})$. Clearly, the resulting signed
	partition is in general not a symplectic partition.
	Likewise,
	\begin{equation}\label{completion of signed partition valued function}
	\overline{(\pmb{\lambda},h^+,h^-)}=(\overline{\pmb{\lambda}},\overline{h^+},\overline{h^-})
	\end{equation} where $\overline{h^+}={h^+},\overline{h^-}={h^-}$. Finally, the $n$-completion $(\pmb{\lambda},h^+,h^-)^{\uparrow\uparrow n}$ of $(\pmb{\lambda},h^+,h^-)$ is defined by the rule
	$(\pmb{\lambda},h^+,h^-)^{\uparrow\uparrow n}=(\pmb{\lambda}^{\uparrow 
		2n},{h^+}^{\uparrow\uparrow},{h^-}^{\uparrow\uparrow})$ where 
	${h^+}^{\uparrow\uparrow}={h^+}$ and ${h^-}^{\uparrow\uparrow}$ is defined 
	similarly. In fact, consider 
	$(\lambda(t-1),h^-)=(1^{(m_1,\epsilon_1)},2^{(m_2,\epsilon_2)},\cdots, 
	r^{(m_r,\epsilon_r)})$. Then we define ${h^-}^{\uparrow\uparrow n}$ the 
	sequence $\pm 1$ so that the equality $(\lambda(t-1)^{\uparrow\uparrow 
		n},{h^-}^{\uparrow\uparrow n})=(1^{(m_0,-1)},2^{(m_1,\epsilon_1)},\cdots, 
	(r+1)^{(m_r,\epsilon_r)})$ holds, where $r=n-||\overline{\lambda}||$.
	The \textbf{unipotent} and \textbf{non-unipotent} blocks are defined analogously.
\end{Definition}

\begin{Remark}
	Note that, unlike the maps
	$\pmb{\lambda}\longmapsto\overline{\pmb{\lambda}}$ and
	$\pmb{\lambda}\longrightarrow \pmb{\lambda}^{\uparrow 2n}$, the modification
	operator $\overset{\circ}{(\cdot)}$ does not map $\mc{P}(\Phi^s)$ to itself as
	the weight of the resulting function may fail to be even. The set of
	\textbf{modified symplectic functions} $\mc{P}^{st}$ is defined as the image
	$\overset{{\circ}}{\mc{P}^{s}(\Phi)}$ of $\pmb{\lambda}\longmapsto
	\overset{\circ}{\pmb{\lambda}}$. Clearly in this case $\pmb{\lambda}\longmapsto
	\overline{\pmb{\lambda}}$ maps the modified symplectic functions to the
	symplectic functions.
\end{Remark}

If $U\in Sp_m(q)$ and $^s\pmb{\lambda}^U=(\pmb{\lambda},h^+,h^-)$ is
the symplectic type of $U$, then it follows that
\begin{equation}\label{behaviour of symplectic type under the embedding}
^s\pmb{\lambda}^{(U^{\uparrow \uparrow n})}=(\pmb{\lambda}^{\uparrow 2n},h^+,{h^-}^{\uparrow \uparrow n}).
\end{equation}
where the operation $\pmb{\lambda}\longmapsto\pmb{\lambda}^{\uparrow 2n}$ for partition valued functions was
described in Remark \ref{definition of uparrow n}. Relying on this observation we follow the idea
of the definition given in \cite{WW18} and introduce the map
\begin{equation}
U\longmapsto \overset{\circ}{^s\pmb{\lambda}^U}=\in \mc{P}^{st}(\Phi)
\end{equation}
and called the image function \textbf{modified symplectic type} of $U$.

\begin{Remark}[Reflection length]
	Let $G$ be an abstract group and $R\subset G$ be a set of elements that generates $G$ as a monoid. The length $l(g)$ of $g\in G$ with respect to $R$ is defined to be the minimum of
	\[
	\{l \in\N:g=r_1r_2\cdots r_l,\; r_i\in R\}.
	\] Such a function is clearly a sub-additive function. If $R$ is closed under conjugation then $l$ is invariant on the conjugacy classes. In the case of symplectic groups, the set $R$ is taken to be transvections in general, which are by definition reflections of determinant $1$. In this case, the relation between reflection length and residual space of an element $g\in Sp_n(q)$ is as follows, (cf \cite{Om78}, Thm. 2.1.11):
	\begin{enumerate}
		\item If $g$ is an involution then $l(g)=\dim R^g+1$.
		\item If $g$ is not involution then $l(g)=\dim R^g$.
	\end{enumerate}
	This means, the reflection length on $Sp_n(q)$ induced by transvections is not consistent with the weight of the stable type. As a result, we will be considering $Sp_n(q)$ with the reflection length induced from $GL_{2n}(q)$.
\end{Remark}

\begin{Lemma}\label{Sp-n is a saturated family}
	\begin{enumerate}
		\item The family $(Sp_n(q))_{n\in \N}$ is a saturated family.
		\item The map $U\longmapsto \overset{\circ}{^s\pmb{\lambda}^U}$ induces a bijection between the conjugacy classes of $Sp_{\infty}(q)$, and the set of all stabilized symplectic functions $\mc{P}^{st}$.
		\item Let $\pmb{\lambda}\in \mc{P}^{st}$ be a modified  symplectic function. Then $Sp_m(q)$ contains an element whose symplectic stable type is $\pmb{\lambda}$ if and only if $||\overline{\pmb{\lambda}}||\leq 2m$.
		\item Let $\pmb{\lambda}\in \mc{P}^{st}$ be a modified symplectic function such that $||\overline{\pmb{\lambda}}||= 2m$. Let $U\in Sp_m(q)$ be an element whose modified type is $\pmb{\lambda}$ and $n$ be an integer greater than $m$. Then
		\begin{equation}
		\pmb{\lambda}^{\uparrow n}=\pmb{\lambda}^{(U^{\uparrow\uparrow n})}
		\end{equation}
		where $\pmb{\lambda}^{\uparrow n}$ denotes the image of $\pmb{\lambda}$ in $\widehat{Sp_n}(q)$.
		\item Reflection length remains unchanged under the embedding $Sp_m(q)\hookrightarrow Sp_m(q)$ and it is equal to the weight of the stable type.
	\end{enumerate}
\end{Lemma}

\begin{proof}
	\begin{enumerate}
		\item By Eq.\eqref{behaviour of symplectic type under the embedding} one can see that non-conjugate elements in $Sp_m(q)$ remain non-conjugate in $Sp_n(q)$ for $m\leq n$ which proves the first claim.
		\item The fact that $U\longmapsto \overset{\circ}{^s\pmb{\lambda}^U}$ defines a well-defined map from $\widehat{Sp_{\infty}}(q)$ to $\mc{P}^{st}(\Phi)$ follows from
		Eq.\eqref{behaviour of symplectic type under the embedding} and the rest follows from Theorem \ref{symplectic parametrization of conjugacy classes}.
		\item and 4. are formal consequences of the definitions.
		\item[5.] Follows from the fact that the weight of the symplectic stable type is equal to the weight of the stable type and Lemma \ref{reflection length and stable type}/\ref{reflection length equal to weight}.
	\end{enumerate}
	
\end{proof}

The following two lemmas are symplectic analogous of Lemma \ref{reflection length and stable type} and Lemma \ref{reflection length and residual dimension}.

\begin{Lemma}\cite[Proposition 2.9, 2.16]{HLR17}\label{symplectic reflection length and residual dimension}
	\begin{enumerate}
		\item For $U\in Sp_n(q)$ the reflection length
		$rl(U)$ is equal to the $\codim V_n^{U}$.
		\item The reflection length is sub-additive: i.e. , the inequality $rl(U_1U_2)\leq rl(U_1)+rl(U_2)$ holds for all $U_1,U_2\in Sp_n(q)$.
		\item \label{intersection equality for the top coefficients} If $rl(U_1U_2)=rl(U_1)+rl(U_2)$ then $V_n^{U_1}\cap V_n^{U_2}=V_n^{U_1U_2}$ and $V_n=V_n^{U_1}+V_n^{U_2}$.
	\end{enumerate}
\end{Lemma}

\begin{Lemma}\cite[Lemma 3.2]{WW18}\label{symplectic reflection length and modified type}
	The reflection length is stable under the embedding $Sp_m(q)\subseteq Sp_n(q)$ for all $n,m\in \N$ satisfying $m\leq n$. Moreover:
	\begin{enumerate}
		\item\label{symplectic reflection length equal to weight} If the modified type of $U$ is $\pmb{\lambda}$, then $rl(U)=||\pmb{\lambda}||$.
		\item If the modified type of $U_1,U_2,U_1U_2\in Sp_{\infty}(q)$ are $\pmb{\lambda},\pmb{\mu},\pmb{\nu}$ then
		\[
		||\pmb{\lambda}||+||\pmb{\mu}||\leq ||\pmb{\nu}||.
		\]
	\end{enumerate}
\end{Lemma}

\begin{proof}(of \ref{symplectic reflection length and residual dimension} and \ref{symplectic reflection length and modified type})
	Use Lemma \ref{reflection length and residual dimension} and Lemma \ref{reflection length and stable type} and the fact that the reflection length on $Sp_m(q)$ is the reflection length induced by $GL_{2m}(q)$ and along with the fact that weight of a symplectic function is equal to the weight of the underlying partition valued function.
\end{proof}

We end this section following the lines of \cite{WW18} in the context of symplectic groups. Let $\pmb{\lambda}=(\pmb{\lambda},h^+,h^-)$ be a stabilized symplectic function and let $\pmb{\lambda}$ also denote the conjugacy class in $Sp_{\infty}(q)$ which corresponds to $\pmb{\lambda}$. Let $n$ be a positive integer. Then
\begin{equation}\label{minimal realization of a conjugacy class}
{\pmb{\lambda}}(n):=Sp_{n}\cap \pmb{\lambda}\neq \emptyset \qquad \Longleftrightarrow \qquad ||\overline{\pmb{\lambda}}||\leq 2n,
\end{equation}
in which case we set
\begin{equation}
K_{\pmb{\lambda}}(n)=\sum_{g\in {\pmb{\lambda}}(n)}g.
\end{equation}
$K_{\pmb{\lambda}}(n)$ is an element of $\mc{H}_n:=\mc
H(Sp_{n}(q))$, the center of the integral group algebra $\Z[Sp_n(q)]$. Notice that if ${\pmb{\lambda}}(n)=\emptyset$ then the above sum is over the empty set and hence equal to $0$.

\begin{Lemma}\label{basis for center in the symplectic group case}
	The set $\{K_{\pmb{\lambda}}(n)\neq 0:\pmb{\lambda}\in \mc{P}(\Phi)\}$ forms the class sum $\Z$-basis for the center $\mc{H}_n$, for each $n\geq 0$.
\end{Lemma}

\subsection{Structure constants of $\mc{H}_n$ and the main theorems}

We start with proving the normal form theorem (cf. Proposition \ref{normal form}) in the context of symplectic groups. This will allow us to deduce that the simultaneous conjugation admits finitely many orbits. 

\begin{Proposition}[Normal Form Theorem]\label{normal form for symplectic case}
	Let $U_1,U_2,U_1U_2\in Sp_{n}(q)$ and $\pmb{\lambda}, \pmb{\mu}, \pmb{\eta}$ be their modified symplectic types respectively. Suppose $||\pmb{\eta}||=||\pmb{\lambda}||+||\pmb{\mu}||$ and  $||\overline{\pmb{\eta}}||=2m$.There exists $T\in Sp_{n}(q)$ and $\overline{U_1},\overline{U_2}\in Sp_{m}(q)$ such that
	\begin{eqnarray}
	TU_1T^{-1}=\begin{bmatrix}
	\overline{U_1} & 0 \\
	0 & I_{2n-2m}
	\end{bmatrix},  \qquad 
	TU_2T^{-1}=\begin{bmatrix}
	\overline{U_2} & 0 \\
	0 & I_{2n-2m}
	\end{bmatrix}
	\end{eqnarray}
	and
	\begin{equation}
	TU_1U_2T^{-1}  =  \begin{bmatrix}
	\overline{U_1}\overline{U_2} & 0 \\
	0 & I_{2n-2m}
	\end{bmatrix}.
	\end{equation}
\end{Proposition}
\begin{proof}
	We will use Lemma \ref{normal form lemma} as it is used in
	the proof of Prop. \ref{normal form} in \cite{WW18}. Since
	the modified symplectic type of $U_1U_2$ is $\pmb{\eta}$, and $||\overline{\pmb{\eta}}||=2m$, it
	follows that there exists a symplectic transformation $U_{\pmb{\eta}}\in Sp_m(q)$ which is conjugate to $U_1U_2$, hence there exists an element $T$ in
	$Sp_n(q)$ so that the matrix of $TU_1U_2T^{-1}$ is equal to the
	matrix $U_{\pmb{\eta}}^{\uparrow\uparrow n}$:
	\begin{equation}
	TU_1U_2T^{-1}=U_{\pmb{\eta}}^{\uparrow\uparrow n}=\begin{bmatrix}
	U_{\overline{\eta}} & 0\\
	0 & I_{2n-2m}\\
	\end{bmatrix}.
	\end{equation}
	Considering
	$U_1$, $U_2$, $U_1U_2$ as elements of $GL_{2n}(q)$ and using the fact
	that the weight of the symplectic partition valued function
	and the weight of the ordinary partition valued function
	defined by the same element are equal, we may apply Lemma
	\ref{normal form lemma} to the triple $U_1$, $U_2$, $U_1U_2$, from which
	the result follows.
\end{proof}

Let $Z=Z(\pmb{\lambda}\times \pmb{\beta}:\pmb{\eta})$ be the set of elements $(U_1,U_2)\in \pmb{\lambda}\times \pmb{\beta}$ such that $U_1U_2\in \pmb{\eta}$. The group $Sp_{\infty}(q)$ acts on $Z$ by simultaneous conjugation, which is defined by the rule $T\cdot (U_1,U_2):= (TU_1T^{-1},TU_2T^{-1})$, for $T\in Sp_{\infty}(q)$.

\begin{Corollary}
	The set $Z$ admits finitely many orbits with respect to the simultaneous conjugation.
\end{Corollary}
\begin{proof}
	Follows directly from the normal form theorem as each orbit contains a representative in $Sp_m(q)$, which is a finite set.
\end{proof}
By the proposition, up to conjugation, we may assume that $U_1$, $U_2$ and $U=U_1U_2$ are all contained in $Sp_m(q)$. Let $d$ be
the dimension of the fixed space of $U_{\pmb{\eta}}$.

\begin{Corollary}\label{structure constants corollary}
	Let $L_1,\cdots,L_k$ be the totality of orbits in $Z=Z(\pmb{\lambda}\times \pmb{\beta}:\pmb{\eta})$ and $(U_{1i},U_{2i})\L_i\cap Sp_{m}\times Sp_{m}$. Let $(U_{1i},U_{2i})\in L_i$ and $U_i=U_{1i}U_{2i}$ for $i=1,\cdots,k$. Then for $n\geq m$
	\begin{equation} \label{structure constants corollary equation}
	c_{\pmb{\lambda},\pmb{\mu}}^{\pmb{\eta}}(n)=\sum_{i=1}^k\frac{C_{Sp_n(q)}(U^{\uparrow\uparrow n})}{C_{Sp_{n}}(U_{1i}^{\uparrow\uparrow n})\cap C_{Sp_{n}}(U_{2i}^{\uparrow\uparrow n})}
	\end{equation}
	where $c_{\pmb{\lambda},\pmb{\mu}}^{\pmb{\eta}}(n)\geq 0$ is the coefficient of $K_{\pmb{\eta}}(n)$ satisfying
	\[
	K_{\pmb{\lambda}}(n)\cdot K_{\pmb{\mu}}(n)=\sum _{\pmb{\eta}\in \mc{P}^{st}(\Phi)}c_{\pmb{\lambda},\pmb{\mu}}^{\pmb{\eta}}(n)\cdot K_{\pmb{\eta}}(n)
	\] 
\end{Corollary}
\begin{proof}
	For $i,j=1,\cdots,k$, the elements $U_i$ and $U_j$ are conjugate to each other and together conjugate to $U$, so one can take $U_i=U$. This means,
	$Z(n):=Z\cap Sp_n(q)\times Sp_n(q) $ is in fact the set of $(x,y)\in\pmb{\lambda}\times \pmb{\mu}$ such that
	$xy\in (U^{\uparrow\uparrow n})^{Sp_n}$, hence
	$c_{\pmb{\lambda},\pmb{\mu}}^{\pmb{\eta}}(n)=\frac{Z(n)}{|\pmb{\eta}|}$. Order of
	the orbit of $(U_{1i},U_{2i})$ is equal to $Sp_n(q)/Stab(U_{1i},U_{2i})$, where
	$Stab(U_{1i},U_{2i})$ is the stabilizer of $(U_{1i},U_{2i})$ under the simultaneous conjugation. The cardinality of the stabilizer is
	clearly equal to $C_{Sp_{n}}(U_{1i}^{\uparrow\uparrow n})\cap
	C_{Sp_{n}}(U_{2i}^{\uparrow\uparrow n})$.
\end{proof}

\begin{Theorem}[Growth of centralizers]\label{growth of centralizer in symplectic group- main statement}
	For $m\leq n$ the following equalities hold:
	\begin{equation}\label{growth of centralizer in symplectic group- main equation 1}
	|C_{Sp_n(q)}(U^{\uparrow\uparrow n})|=|C_{Sp_m(q)}(U)|\cdot |Sp_{n-m}(q)|\cdot q^{2(n-m)d}
	. \end{equation}
	and
	\begin{equation}\label{growth of intersection of centralizers in symplectic group- main equation 2}
	|C_{Sp_{n}}(U_1^{\uparrow\uparrow n})\cap C_{Sp_{n}}(U_2^{\uparrow\uparrow n})|=|C_{Sp_{m}}(U_1)\cap C_{Sp_{m}}(U_2)|\cdot |Sp_{n-m}(q)|\cdot q^{2(n-m)d}.
	\end{equation}
\end{Theorem}
\begin{proof}
	See the next chapter.
\end{proof}

The following theorem is the stability theorem in the case of symplectic groups. We present it in the form given in \cite{WW18}.

\begin{Theorem}[Stability Theorem]\label{stability theorem for symplectic group}
	Let $\pmb{\lambda},$ $\pmb{\mu}$, $\pmb{\eta}$ be three modified symplectic functions and assume that $||\pmb{\eta}||=||\pmb{\lambda}||+||\pmb{\mu}||$. Then $c_{\pmb{\lambda},\pmb{\mu}}^{\pmb{\eta}}(n)$ is a non-negative integer independent of $n$.
\end{Theorem}

\begin{proof}
	
	Substituting the order formulas \eqref{growth of centralizer in symplectic group- main equation 1} and \eqref{growth of intersection of centralizers in symplectic group- main equation 2} in the equation given in Corollary \ref{structure constants corollary} we see that each summand in the right hand side of the Eq.  \eqref{structure constants corollary equation} is equal to
	\begin{equation}
	\frac{|C_{Sp_m(q)}(U_i)|}{|C_{Sp_{m}}(U_{1i})\cap C_{Sp_{m}}(U_{2i})|}
	\end{equation}
	which is independent of $n$.

\end{proof}

\section{Proof of Centralizer growth theorem}\label{C:5}

In this chapter, we will prove the Theorem \ref{growth of centralizer in symplectic group- main statement}, which was the main ingredient of the proof of the Theorem \ref{stability theorem for symplectic group}.

\subsection{Generic matrices and symplectic equations}

Let $F$ be an arbitrary field and $n,m\in \N$ be positive integers. The set of $n\times m$ matrices whose entries are in $F[x_{ij}],$ $i=1,\cdots,n;j=1,\cdots,m$ is called the $n\times m$ \textbf{generic matrices}. Let
$S=\{i_1j_1,\cdots, i_rj_r\}$ be a set of indices. A \textbf{generic matrix with free indices in $S$}
is a generic $n\times m$ matrix $D(S)=D=(d_{ij})_{i,j}$ such that $d_{ij}=x_{ij}$ if $(i,j)\in S$ and
$d_{ij}\in F$ if $(i,j)\notin S$. By substituting elements from $F$ to the variables in
$S$, each generic matrix $D(S)$ with free variables in $S$ defines a function from $F^S$ to $\mat_{n\times n}(F)$.
If $\overline{\alpha}\in F^S$, the image of $\overline{\alpha}$ under this map is denoted by $D(\overline{\alpha})$ and each matrix in the image of a generic
matrix $D$ is called a \textbf{realization} of $D$. In the case of
$S=\{(i,j)=i=1,\cdots, n;j=1,\cdots,m\}$ there
is a unique generic matrix, the
\textbf{universal generic matrix} $X$. For example, if $S=\{(1,1),(2,2)\}$, then
\[
\begin{bmatrix}
x_{11} & 3 \\
2 & x_{22}
\end{bmatrix}
\]
is a generic $2\times 2$ matrix with respect to $S$. Then the realization $D(5,7)$ of $D$ is
\[
\begin{bmatrix}
5 & 3 \\
2 & 7
\end{bmatrix}
\]
Let $f$ be a function of the entries of $D$. Then one can define a function $f^D$ on the set of realizations of $D$. For example $\det^D$ for $D$ introduced above is given by the following formula:
\begin{equation}
{\det}^D(x_{11},x_{22})=x_{11}x_{22}-6.
\end{equation}
Recall our conventions on the labeling of the rows and columns of matrices. We now insist on the condition that when the matrix is square, the labeling of rows and columns will be assumed to be done with respect to the same ordered basis.
For example if $X$ is  the $2n\times 2n$ generic matrix and $B=\{e_1,f_1\cdots,e_n,f_n\}$ is an hyperbolic basis for $V$, then columns and rows of the $X$ are indexed by the basis elements preserving their orders. So, an entry of $X$ is of the following form: $x_{uv}$ where $u,v\in B$. To be even more concrete, we present the following example.

\begin{Example}
	Assume that $X$ is the $4\times 4$ universal generic matrix and the indexing of its columns (and hence its rows) is $e_1,e_2,f_2,f_1$. Then we write the universal matrix $X$ as
	\[
	X=\begin{blockarray}{ccccc}
	e_1 & e_2 & f_2 & f_1 & \\
	\begin{block}{(cccc)c}
	x_{e_1e_1} & x_{e_1e_2} & x_{e_1f_2} & x_{e_1f_1} & {e_1} \\
	x_{e_2e_1} & x_{e_2e_2} & x_{e_2f_2} & x_{e_2f_1} & {e_2} \\
	x_{f_2e_1} & x_{f_2e_2} & x_{f_2f_2} & x_{f_2f_1} & {f_2} \\
	x_{f_1e_1} & x_{f_1e_2} & x_{f_1f_2} & x_{f_1f_1} & {f_1} \\
	\end{block}
	\end{blockarray}
	\]
\end{Example}
The $uv$-th \textbf{symplectic equation} $SE(u,v,B)$ with respect to the fixed hyperbolic basis with a prescribed ordering, which concerns the entries of $u$-th and $v$-th columns of $X$, is defined as follows:
\begin{eqnarray}\label{symplectic equations}
\sum_{i=1}^n x_{e_i u}x_{f_i v}-\sum_{i=1}^n x_{f_i u}x_{e_i v}=Q(u,v).
\end{eqnarray}
Observe that the left hand side of the equation is nothing but the formal
image of $Q(C_u(X),C_v(X))$.  In fact, by considering matrices with labeled
rows and columns, we will view the columns of matrices as elements in the
image vector space, and we will often identify the column and the vector
defined by the column (depending on the labeling). For example symplectic
equation $SE(e_1,f_2)$ for $X$ above can be calculated by treating the
entries as coefficients of basis vectors. That is
\begin{eqnarray}
0 & = & Q(e_1,f_2)\nonumber\\
& = & Q(X(e_1),X(f_2))\\
& = & Q(x_{e_1e_1}e_1+x_{e_2e_1}e_2+x_{f_2e_1}f_2+x_{f_1e_1}f_1,x_{e_1f_2}e_1+x_{e_2f_2}e_2+x_{f_2f_2}f_2+x_{f_1f_2}f_1)\nonumber\\
& = & x_{e_1e_1}x_{f_1f_2}+x_{e_2e_1}x_{f_2f_2}-x_{f_2e_1}x_{e_2f_2}-x_{f_1e_1}x_{e_1f_2}
\end{eqnarray}
The set of all symplectic equations $SE(u,v,B)$, $u,v\in B$ is called the \textbf{symplectic equations} with respect to $B$ and denoted by $SE(B)$.

\begin{Remark}
	Symplectic equations can be considered for generic matrices with free variables. For example, consider the the following $4\times 4$ generic matrix $D(S)$ with free variables in $S=\{(e_{2},e_{1}),(e_2,f_2)(f_1,f_2)\}$
	\[
	D(S)=\begin{blockarray}{ccccc}
	e_1 & e_2 & f_2 & f_1 & \\
	\begin{block}{(cccc)c}
	3 & 0 & 4 & 0 & {e_1} \\
	x_{e_2e_1} & 3 & x_{e_2f_2} & -4 & {e_2} \\
	0 & 0 & 6 & 0 & {f_2} \\
	1 & 0 & x_{f_1f_2} & 6 & {f_1} \\
	\end{block}
	\end{blockarray}
	\]
	Then the symplectic equations with respect to $D(S)$ are obtained by specifying entries of $D(S)$ in the symplectic equations and the will be denoted again by $E(u,v)$ when the basis $B$ and $D$ are fixed.
	\begin{enumerate}
		\item The equation $SE(e_1,f_2)$ is obtained by considering the equality \[Q(e_1,f_2) = Q(C_{e_1}(D(S),C_{f_2}(D(S))\]
		hence $E(e_1,f_2)$ is $0 = 3x_{x_{f_1f_2}}+6x_{e_2e_1}-4$, or simply
		\[
		4 = 3x_{x_{f_1f_2}}+6x_{e_2e_1}.
		\]
		\item The equations $SE(e_1,e_2)$ and $E(f_2,f_1)$ can be computed similary and they are simply $0=0$.
		\item Finally, the equation $SE(e_1,f_1)$ is
		\[
		1=Q(e_1,f_1)=3\cdot 6 + x_{e_2e_1}\cdot 0 - 0\cdot (-4) - 1\cdot 0=18.
		\]
		This means that there is no symplectic realization $M$ of $D(S)$.
	\end{enumerate}
	
\end{Remark}

Using this terminology, there is a tautological result concerning the symplectic transformations which we record as the next lemma. It will be beneficial in the calculation of the growth of the centralizers of unipotent elements.  

\begin{Lemma}
	Let $(V,Q)$ be a non-degenerate symplectic space and $B$ be an hyperbolic basis with a prescribed order. Let $U\in GL(V)$. Then, $U\in Sp(V)$ if and only if the columns of $U$ satisfy the symplectic equations $SE(B)$.
\end{Lemma}

We end this section with inducing the question of the growth of the centralizer of a general symplectic matrix $U$ case to the unipotent $U$ case:

\begin{Remark}[Growth depends on the unipotent block]\label{Growth depends on the unipotent part} 
	Let $U$ be a symplectic transformation whose 
	non-modified type is the 
	symplectic partition valued function $(\pmb{\lambda},h^+,h^-)$ of weight $2m$. 
	Then, by Lemma \ref{orthogonal decomposition of an 
		isometry}, we may assume that $U$ has the form 
	\begin{equation}
	U=\begin{bmatrix}
	U_{\pmb{\lambda}^{ne}} & 0\\
	0 & U_{\pmb{\lambda}^{e}} \\
	\end{bmatrix}
	\end{equation}
	where the type of $U_{\pmb{\lambda}^{ne}}$ is $\pmb{\lambda}^{ne}$, the type of $U_{\pmb{\lambda}^{e}}$ is $\pmb{\lambda}^{e}$, and the diagonal sum of the matrices is an orthogonal sum. From this we conclude that that the minimal polynomial of $U_{\pmb{\lambda}^{e}}$ is a power of $t-1$ and the minimal polynopmial of $U_{\pmb{\lambda}^{ne}}$ is coprime to $t-1$. Now we consider
	the embedding of $U$ into $Sp_n(q)$ for some $n>m$ and and
	an element $D$ from the centralizer of $U^{\uparrow\uparrow n}$ and writing it in the block form of $U^{\uparrow\uparrow n}$ yields the following eaulity:
	\begin{equation}\begin{bmatrix}
	U_{\pmb{\lambda}^{ne}} & 0 & 0\\
	0 & U_{\pmb{\lambda}^{e}} & 0 \\
	0 & 0 & I_{2n-2m} \\
	\end{bmatrix}
	\begin{bmatrix}
	D_{11} & D_{12} & D_{13} \\
	D_{21} & D_{22} & D_{23} \\
	D_{31} & D_{32} & D_{33} \\
	\end{bmatrix}=	\begin{bmatrix}
	D_{11} & D_{12} & D_{13} \\
	D_{21} & D_{22} & D_{23} \\
	D_{31} & D_{32} & D_{33} \\
	\end{bmatrix}\begin{bmatrix}
	U_{\pmb{\lambda}^{ne}} & 0 & 0\\
	0 & U_{\pmb{\lambda}^{e}} & 0 \\
	0 & 0 & I_{2n-2m} \\
	\end{bmatrix}
	\end{equation}
	Then one obtains the following equality of matrices:
	\begin{equation}
	\begin{bmatrix}
	U_{\pmb{\lambda}^{ne}}D_{11} & U_{\pmb{\lambda}^{ne}}D_{12} & U_{\pmb{\lambda}^{ne}}D_{13} \\
	U_{\pmb{\lambda}^{e}}D_{21} & U_{\pmb{\lambda}^{e}}D_{22} & U_{\pmb{\lambda}^{e}}D_{23} \\
	D_{31} & D_{32} & D_{33} \\
	\end{bmatrix}=
	\begin{bmatrix}
	D_{11}U_{\pmb{\lambda}^{ne}} & D_{12} U_{\pmb{\lambda}^{e}} & D_{13} \\
	D_{21}U_{\pmb{\lambda}^{ne}} & D_{22} U_{\pmb{\lambda}^{e}} & D_{23} \\
	D_{31}U_{\pmb{\lambda}^{ne}} & D_{32} U_{\pmb{\lambda}^{e}} & D_{33} \\
	\end{bmatrix}
	\end{equation}
	From this, it follows that each $D_{ij}$ is an intertwining operator between $\F_q[t]$-modules. However, as
	pointed out earlier in Remark \ref{schurs lemma and primary cyclic modules} and Remark
	\ref{centralizer of a block diagonal matrix admits a direct sum decomposition}, an
	intertwining operator between two modules with distinct primary cyclic parts must be zero.
	Since the primary cyclic parts of the modules defined by $U_{\pmb{\lambda}^{u}}$ and $I_{2n-2m}$ are
	all of type $\F_q[t]/(t-1)^r$ for some $r\geq 1$ and the primary cyclic parts of the modules
	defined by $U_{\pmb{\lambda}^{nu}}$ are all of type $\F_q[t]/(f)^r$ for some $f\neq t-1$ and
	$r\geq 1$ it follows that the intertwining operators $D_{12},D_{13},D_{21},D_{31}$ are all
	zero. As a result
	\begin{equation}
	D=\begin{bmatrix}
	D_{11} & 0 & 0\\
	0 & D_{22} & D_{23}\\
	0 & D_{32} & D_{33}\\
	\end{bmatrix}
	\end{equation}
	where $D_{11}$ is in the centralizer of $U_{\pmb{\lambda}^{ne}}$ and
	$\begin{bmatrix}
	D_{22} & D_{23}\\
	D_{32} & D_{33}\\
	\end{bmatrix}$ is in the centralizer of $U_{\pmb{\lambda}^{e}}$. This means, in order to investigate the growth of the centralizer of a symplectic matrix $U$ under the embedding $U\longmapsto U^{\uparrow\uparrow n}$, it is sufficient to consider the same question for the unipotent block of $U$.
\end{Remark}

\subsection{Unipotent Matrix Actions}

In this section, we introduce an action of $\mat_n\times \mat_m$ on $\mat_{n\times m}$ as follows. For every square matrix $A\in \mat_{n\times n}$, $B\in\mat_{m\times m}$ and $M\in \mat_{n\times m}$ put
\[
(A,B)\cdot M=AMB
\]
We will introduce some terminology concerning the fixed points of a fixed $(A,B)\in \mat_{n\times n}\times \mat_{m\times m}$ which is similar to the concept of symplectic equations introduced earlier. Taking $M$ as the generic matrix $X$ and writing
\begin{equation}\label{fixed point homogeneous}
AXB-X=0
\end{equation}
induces a homogeneous system of linear equations in the variables $x_{ij}$,
$i=1,\cdots,n,j=1,\cdots,m$, which will be denoted by $E(A,B;\overline{x_{ij}})$. Clearly, each solution of the
system $E(A,B;\overline{x_{ij}})$ defines a fixed point of $(A,B)$. An index $(r,k)$ is called a \textbf{free
	index} with respect to $(A,B)$, if $x_{rk}$ does not appear in the system $E(A,B;x_{ij})$ of linear equations induced by Eq.\eqref{fixed point homogeneous}, in which case we refer to $x_{rk}$ as a \textbf{free variable} with respect to $(A,B)$, or simply a free variable. This means, if $M\in \mat_{m\times n}$ then the condition of $M$
being a fixed point can be checked without knowing $m_{rk}$, so the following definition makes
sense: A \textbf{generic fixed point} of $(A,B)$ with respect to a set $S$ of free indices is a
generic matrix $D(S)$ with free variables in $S$ where $D(\overline{\alpha})$ is a fixed point of $(A,B)$ for every $\overline{\alpha}\in F^S$.

\begin{Example}
	Let $A=B=\begin{bmatrix}
	1 & 0\\
	1 & 1
	\end{bmatrix}$. Then the equation Eq.\eqref{fixed point homogeneous} reads as
	\[
	\begin{bmatrix}
	1 & 0\\
	1 & 1
	\end{bmatrix} \begin{bmatrix}
	x_{11} & x_{12}\\
	x_{21} & x_{22}
	\end{bmatrix} \begin{bmatrix}
	1 & 0\\
	1 & 1
	\end{bmatrix}
	=\begin{bmatrix}
	x_{11} & x_{12}\\
	x_{21} & x_{22}
	\end{bmatrix}
	\]
	Direct multiplication yields
	\[
	\begin{bmatrix}
	x_{11}+x_{12} & x_{12}\\
	x_{21}+x_{22}+x_{12}+x_{11} & x_{22}+x_{12}
	\end{bmatrix}=\begin{bmatrix}
	x_{11} & x_{12}\\
	x_{21} & x_{22}
	\end{bmatrix}
	\]
	Therefore, the induced homogeneous system $E(A,B;\overline{x_{ij}})=E(A,B;x_{11},x_{12},x_{21},x_{22})$ of linear equations is
	\begin{eqnarray}
	x_{12} & = & 0 \nonumber\\
	x_{11}+ x_{22} & = & 0 \nonumber
	\end{eqnarray}
	This means, the only free index with respect to $(A,B)$ is $(2,1)$. The matrix
	\[
	D(x_{12})=\begin{bmatrix}
	1 & 0\\
	x_{21} & -1
	\end{bmatrix}
	\]
	is thus a generic fixed point of $(A,B)$ and the realization $D(2)=\begin{bmatrix}
	1 & 0\\
	2 & -1
	\end{bmatrix}$ of $D$ is an actual fixed point of $(A,B)$.
\end{Example}

\begin{Lemma}
	Let $A\in \mat_{n\times n}(\F_q),B\in \mat_{m\times m}(\F_q)$ and let $S$ be the set of free indices induced by $(A,B)$. If $GFix(A,B)$ denotes the set of generic fixed points of $(A,B)$ and $Fix(A,B)$ denotes the set of fixed points of $(A,B)$ then
	\begin{equation}
	|Fix(A,B)|=|GFix(A,B)|\cdot q^{|S|}.
	\end{equation}
\end{Lemma}
\begin{proof}
	Follows from the definitions.
\end{proof}

The last lemma will be useful when considering the growth of the centralizer of elements under the natural embedding $Sp_{m}(q)\hookrightarrow Sp_{n}(q)$ for $m\leq n$, where
the next lemma will be useful when considering the intersection of centralizers of two matrices. An $n\times m$ matrix whose only non-zero is $1$ and placed at the $(r,k)$ will be denoted by $1_{rk}$. Observe that in the notation there is no reference to the size, but in each case, it will be determined by the context.

\begin{Lemma}\label{free index in terms of 1 solution}
	An index $(r,k)$ is a free index with respect to $(A,B)$ if and only if the matrix ${1}_{rk}$ is a fixed point of $(A,B)$.
\end{Lemma}
\begin{proof}
	$(\Rightarrow)$ Assume that $(r,k)$ is a free index. Then the linear system of equations $E(A,B;\overline{x_{ij}})$ induced by $(A,B)$ is homogeneous and $x_{rk}$ does not appear in these equations. As every homogeneous system of linear equations admits the trivial solution, $1_{rk}$ is a fixed point of $(A,B)$. \\
	$(\Leftarrow)$ Assume that $(r,k)$ is not free and let
	\[
	\alpha x_{rk}+\text{other terms with variousvariables $x_{ij}$}=0
	\] where $\alpha\neq 0$. But in this situation the previous equation becomes $1=0$ as the all the variables are equal to zero except $x_{rk}$, which is absurd.
\end{proof}

Now we will restrict the previous action to a certain subset
$\U_n$ of unipotent matrices in $\mat_{n\times n}$ for which we
will be able to determine the free indices precisely. We define
$\mathcal{U}_n$ as the set of unipotent matrices $U$ of size
$n$ which satisfy the following properties: $U$ is lower
triangular and the subdiagonal entries of $U$ are all non-zero.
Hence, elements of $\U_n$ are of the following form:
\begin{equation}
U=
\begin{bmatrix}
1 & 0 & 0 & \cdots & 0 \\
u_{21} & 1 & 0 & \ddots & 0 \\
u_{31} & u_{32} & 1 &  \ddots & 0 \\
\vdots & \vdots & \ddots  &  \ddots & \vdots \\
u_{n1} & u_{n2} & \cdots &  u_{nn-1} & 1 \\
\end{bmatrix}
\end{equation}
where $u_{ii-1}\neq 0$ for $i=2,\cdots,n$.

\begin{Remark}\label{eigen values of general unipotetnt blocks}
	\begin{enumerate}
		\item Let $\mathcal{B}=\{e_1,e_2,\cdots,e_n\}$ be a basis and suppose that the rows and
		columns of the matrix $U\in \mathcal{U}_n$ are indexed by $\mathcal{B}$. Then $V^U=\langle e_n\rangle$ and $^UV=\langle e_1\rangle$.
		\item Moreover, a symplectic block $J_{4m+2}$ is a diagonal sum of two matrices from $\mathcal{U}_{2m+1}$ and an orthogonal block $J_{2m,\epsilon}$ is an element of $\mathcal{U}_{2m}$.
	\end{enumerate}
\end{Remark}

For $n,m\in \mathbb{N}$, one can restrict the previous action to $\U_n\times \U_m$. This action will be called the \textbf{unipotent action}. We are interested in the free indices of $(U_1,U_2)$ with $U_1\in\U_n$, $U_2\in \U_m$ the unipotent action. So let us fix $U_1$ and $U_2$. Observe that $\mc
U_m$ is closed under inversion and hence $U_2^{-1}\in \mc{U}_m$.
So we may write
\begin{equation}
U_1=\begin{blockarray}{cccccc}
e^1_1 & e^1_2 & e^1_3 & \cdots & e^1_n &\\
\begin{block}{(ccccc)c}
1 & 0 & 0 & \cdots & 0 & e^1_1\\
u_{21} & 1 & 0 & \ddots & 0 & e^1_2 \\
u_{31} & u_{32} & 1 &  \ddots & 0 & e^1_3 \\
\vdots & \vdots & \ddots  &  \ddots & \vdots & \vdots \\
u_{n1} & u_{n2} & \cdots &  u_{nn-1} & 1 & e^1_n \\
\end{block}
\end{blockarray}
\qquad U_2^{-1}=\begin{blockarray}{cccccc}
e^2_1 & e^2_2 & e^2_3 & \cdots & e^2_m &\\
\begin{block}{(ccccc)c}
1 & 0 & 0 & \cdots & 0 & e^2_1\\
v_{21} & 1 & 0 & \ddots & 0 & e^2_2 \\
v_{31} & v_{32} & 1 &  \ddots & 0 & e^2_3 \\
\vdots & \vdots & \ddots  &  \ddots & \vdots & \vdots \\
v_{m1} & v_{m2} & \cdots &  v_{mm-1} & 1 & e^2_m\\
\end{block}
\end{blockarray}
\end{equation}

Consider an $n\times m$ matrix $M$. Then the rows of $M$ will be labeled with $\mc{B}_1$ and the columns of $M$ will be labeled with $\mc{B}_2$

\begin{Lemma}\label{unique free index of a fixed point}
	The index $(n,1)$ is the unique free index of the unipotent pair $(U_1,U_2)$. In general, $(e^1_n,e^2_1)$ is the unique free index.
\end{Lemma}
\begin{proof}
	Let $X$ be the generic $n\times m$ matrix. By direct multiplication we calculate the $ij$-th entry of $U_1X$ and $XU_2^{-1}$ and obtain
	\begin{eqnarray}
	u_{i1}x_{1j}+ \cdots +u_{ii-1}x_{i-1j}+x_{ij} & = & (U_1X)_{ij}\nonumber\\ & = & (XU_2^{-1})_{ij} \nonumber\\ & = & x_{ij}+x_{ij+1}v_{j+1j}+\cdots x_{im}v_{mj}\label{equations of unipotent fixed points}
	\end{eqnarray}
	As the subdiagonal entries of $U_1$ and $U_2^{-1}$ are non-zero, it follows that, in the linear equation induced by the $ij$-th position, the coefficients of $x_{i-1j}$ and $x_{ij+1}$ are non-zero, hence they can not be free. On the other hand, the equation \eqref{equations of unipotent fixed points} shows that, in the equation induced by the $ij$-th position, none of the entries below or on the right of $ij$-th position occurs. This proves the claim concerning the index $(n,1)$.
\end{proof}

\begin{Remark}\label{eigen value and free variable}
	The claim that the index $(e^1_n,e^2_1)$ is free can be proved using the description of the eigen-vectors of $U_1$ and $U_2^{t}$, which were determined in Remark \ref{eigen values of general unipotetnt blocks}. Thus we have
	\begin{equation}
	\begin{bmatrix}
	1 & 0 & 0 & \cdots & 0 \\
	u_{21} & 1 & 0 & \ddots & 0 \\
	u_{31} & u_{32} & 1 &  \ddots & 0 \\
	\vdots & \vdots & \ddots  &  \ddots & \vdots \\
	u_{n1} & u_{n2} & \cdots &  u_{nn-1} & 1 \\
	\end{bmatrix} \begin{bmatrix}
	0 & 0 & 0 & \cdots & 0 \\
	0 & 0 & 0 & \ddots & 0 \\
	0 & 0 & 0 &  \ddots & 0 \\
	\vdots & \vdots & \ddots  &  \ddots & \vdots \\
	1 & 0 & \cdots &  0 & 0 \\
	\end{bmatrix}= \begin{bmatrix}
	0 & 0 & 0 & \cdots & 0 \\
	0 & 0 & 0 & \ddots & 0 \\
	0 & 0 & 0 &  \ddots & 0 \\
	\vdots & \vdots & \ddots  &  \ddots & \vdots \\
	1 & 0 & \cdots &  0 & 0 \\
	\end{bmatrix}\nonumber,
	\end{equation}
	\begin{equation}
	\begin{bmatrix}
	0 & 0 & 0 & \cdots & 0 \\
	0 & 0 & 0 & \ddots & 0 \\
	0 & 0 & 0 &  \ddots & 0 \\
	\vdots & \vdots & \ddots  &  \ddots & \vdots \\
	1 & 0 & \cdots &  0 & 0 \\
	\end{bmatrix}\begin{bmatrix}
	1 & 0 & 0 & \cdots & 0 \\
	v_{21} & 1 & 0 & \ddots & 0 \\
	v_{31} & v_{32} & 1 &  \ddots & 0 \\
	\vdots & \vdots & \ddots  &  \ddots & \vdots \\
	v_{m1} & v_{m2} & \cdots &  v_{mm-1} & 1 \\
	\end{bmatrix}= \begin{bmatrix}
	0 & 0 & 0 & \cdots & 0 \\
	0 & 0 & 0 & \ddots & 0 \\
	0 & 0 & 0 &  \ddots & 0 \\
	\vdots & \vdots & \ddots  &  \ddots & \vdots \\
	1 & 0 & \cdots &  0 & 0 \\
	\end{bmatrix}\nonumber.
	\end{equation}
	This means $1_{n1}$ is a solution of $U_1XU_2=X$.  By the Lemma \ref{free index in terms of 1 solution}, $(n,1)$ is a free index. This observation, i.e. proving an index is free by means of $1$-eigen-vectors, will be useful when considering the intersection of two centralizers in the symplectic group.
\end{Remark}

\begin{Lemma}\label{first row lemma}
	In a generic fixed point $D$ of unipotent action (hence in all fixed points), the first row is zero, except possibly for the first entry. This row is called the \textbf{leading row} of $D$. The basis element $e^2_1$ corresponding to this row is called the \textbf{leading basis} element.
\end{Lemma}
\begin{proof}
	The first row of $U_1XU_2$ can be directly computed, hence we can consider the first row of $U_1XU_2$ and $X$. By doing so, one obtains the following system of equations that a generic fixed point must satisfy:
	\begin{eqnarray}
	x_{1m} & = & x_{1m}\nonumber\\
	x_{1m-1} & = & x_{1m-1} + v_{mm-1}x_{1m}\nonumber\\
	x_{1m-2} & = & x_{1m-2} + v_{m-1m-2}x_{1m-1}+v_{mm-2}x_{1m}\nonumber\\
	& \vdots & \nonumber\\
	x_{12} & = & x_{12}+\sum_{j=3}^{m}v_{j2}x_{1j}\nonumber\\
	x_{11} & = & x_{11}+\sum_{j=2}^{m}v_{j1}x_{1j}\nonumber
	\end{eqnarray}
	Since the subdiagonal entries are non-zero, it follows from the second equation that $x_{1m}=0$. Using this fact in the third equation yields
	\[
	x_{1m-2}  =  x_{1m-2} + v_{m-1m-2}x_{1m-1}.
	\]
	As $v_{m-1m-2}$ is a subdiagonal entry, it is non-zero and hence $x_{1m-1}=0$. Clearly, this procedure can be iterated until the last equation, which proves the lemma.
\end{proof}

As a result, a generic fixed point $D(x_{n1})$ of $(A,B)\in\mc{U}_n\times \mc{U}_m$ is of the following form:
\[
D(x_{n1})=\begin{blockarray}{ccccc}
\underbrace{\text{leading column}} & & & \\
\begin{block}{[cccc]c}
d_{11} & 0 & \cdots & 0 & \leftarrow \text{leading row}\\
d_{21} & d_{22} & \cdots & d_{2m} &\\
\vdots & \vdots & \cdots & \vdots & \\
d_{n-1,1} & d_{n-1,2} & \cdots & d_{n-1,m} &\\
x_{n1} & d_{n2} & \cdots & d_{nm} & \leftarrow \text{pivotal  row} \\
\end{block}
\end{blockarray}
\]
where for every $\alpha\in \F_q$ the matrix $D(\alpha)$ obtained by substituting $\alpha$ in $x_{n1}$ is a fixed point of $(A,B)$
under the unipotent action. The row (resp. column) containing the free index will be called the
\textbf{pivotal row} (resp. \textbf{leading column}). For a generic fixed point $D$, the element in the intersection
of the leading row and leading column will be called the \textbf{leading element}. Hence, in the above example, the leading element of $D(x_{n1})$ is $d_{11}\in F$.

Now we generalize these notions to the diagonal sum of matrices. Let $A=diag(A_1,\cdots,A_{r_1})$ and $B=diag(B_1,\cdots,B_{r_1})$ be two $n\times n$ matrices where each block $A_i$ (resp. $B_i$) of $A$ (resp. $B$) are contained in $\mc{U}=\bigcup_{s\geq 1}\mc{U}_i$. A fixed point $D$ of $(A,B)$ is subject to the homogeneous system of linear equations $E$, which is defined by the following equation:
\[
AXB=X.
\]
Let the sizes of $A_i$ and $B_j$ be $a_i$ and $b_j$ respectively, for $i=1,\cdots,r_1; j=1,\cdots,r_2$. And let $X_{ij}$ be the block form of $X$ that is induced from the block forms of $A$ and $B$. More precisely, the $X_{ij}$ is an $a_i\times b_j$ matrix. It is then clear that, the homogeneous system of equations $E$ is equal to the union of homogeneous system of equations $E_{ij}$ defined by the equation.
\begin{equation}
A_iX_{ij}B_j=X_{ij}.
\end{equation}
But this means, if $D$ is a fixed point of $(A,B)$ then each $D_{ij}$ is a fixed point of a certain unipotent action, and hence, one can talk about pivotal row, leading column and leading row of $X_{ij}$. It is also clear that each $E_{ij}$ contains distinct variables, as a result, an indeterminate $x_{uv}$ can occur in at most one system of equations $E_{ij}$. In particular, the set equality concerning linear equations below holds:
\[
E=\bigsqcup_{\substack{i=1,\cdots, r_1\\ j=1,\cdots, r_2}}E_{ij}.
\]
It is also clear that each $E_{ij}$ contains distinct variables, as a result, an indeterminate $x_{uv}$ can occur in at most one system of equations $E_{ij}$.
Call this system of equations $E(x_{uv})$. It is then clear that $x_{uv}$ does not occurs in the homogeneous system of linear equations induced from $AXB-X=0$ if and only if it does not appear in $E(x_{uv})$, i.e. it is a free variable of the equation $E(x_{uv})$. Relying on this observation, we define the \textbf{set of free variables} of $E$ as the union of the set of free variables of $E_{ij}$.

From our previous work, we know that the unique free variable of $A_iX_{ij}B_j$ is the the variable placed in the position $(a_i,1)$. So, if we consider two blocks $X_{i_1j},X_{i_2j}$ in the same column, then, their free variables are contained in the same column of $X$, i.e. leading column of $X_{i_1,j}$ and $X_{i_2j}$ are contained in the same column of $X$. As a result, one can talk about the leading columns of $X$. In fact, the same kind of work can be done for leading rows and pivotal rows as well. Finally, a matrix $D$ is called a \textbf{generic fixed point} of $(A,B)$, if $D_{ij}$ is a generic fixed point of $(A_i,B_j)$.

\subsection{Centralizers of unipotent elements}

In this section, we start working with our original setting. Let $U$ be a unipotent matrix in
$Sp_m(q)$ where $\pmb{\eta}$ is the modified symplectic type $U$ and
$2m=||\overline{\pmb{\eta}}||$. By Theorem \ref{milnor orthogonal splitting}, it follows that $V_m=E_1\perp
\cdots
\perp E_r$, where $E_i$'s are non-degenerate symplectic spaces that are invariant under $U$. Moreover, Proposition \ref{rational forms for unipotent blocks} allows us, up to conjugation we may assume 
\begin{equation}\label{diagonal form of U}
U=diag(U_1,\cdots,U_r)
\end{equation}
and $U_{|E_i}=U_i\neq I$ and that  $U_1,\cdots,U_k$ are
symplectic unipotent blocks and $U_{k+1},\cdots,U_r$ are orthogonal unipotent blocks.
The ordered basis of $E_i$ that is used to index
the columns and rows of $U_i$ is
$B_i=\{e_{i1},\cdots,e_{in_i}f_{in_i},\cdots f_{i1}\}$. The set $\mathcal{B}_m=\cup_{i=1}^r B_i$ forms a hyperbolic basis for $V_m$. We also fix $X=(x_{uv})_{u,v\in\mathcal{B}_m}$, the $2m\times 2m$ matrix where $x_{uv}$ is an indeterminate over $\F_q$.  As in the previous section, we consider $X$ as a block matrix $(X_{ij})_{i,j=1,\cdots,r}$, which is induced by the block form of $U$.

Note that the matrix $U^{-1}$ is an element of $\mc{U}_n$, and it is a again a block diagonal matrix with the same block diagonal structure. Clearly the splitting $V_m=E_1\perp \cdots \perp E_r$ is preserved by $U^{-1}$. We will label the rows and columns of $U^{-1}$ again labeled with the elements of $\mc{B}_m$. A generic fixed point $D$ of $(U,U^{-1})$ will be called a \textbf{generic centralizer} of $U$. Finally, let $d$ be the dimension $\dim V^U=\dim {^UV}$.

\begin{Proposition}\label{shape of a generalized fixed point}
	Let $D$ be a generic centralizer of $U$ and let $D_{ij}$ be the blocks of $D$ induced by the block structure of $U$. Then:
	\begin{enumerate}
		\item If $U_i$ and $U_j$ are both orthogonal, then the block $D_{ij}$ of the generic solution is of the following form:
		\begin{equation}\label{orthogonal orthogonal}
		D_{ij}=\begin{blockarray}{cccccc}
		e_{j1} & \cdots & \cdots & f_{j1} \\
		\begin{block}{[cccc]cc}
		a_{e_{i1}e_{j1}} & 0 & \cdots & 0 & e_{i1} &\leftarrow \text{leading row}\\
		* & * & \cdots & * & e_{i2} & \\
		\vdots & \vdots & \cdots & \vdots & \vdots & \\
		* & * & \cdots & * & f_{i2} &\\
		x_{f_{i1}e_{j1}} & * & \cdots & * & f_{i1}& \leftarrow \text{pivotal row} \\
		\end{block}
		\overbrace{\text{l. cl.}}
		\end{blockarray}
		\end{equation}
		where $a_{e_{i1}e_{j1}}$ is the leading term of $D_{ij}$.
		\item If  $U_i=J_{2s},U_j=J_{2r}$ are both symplectic, then the block $D_{ij}$ is of the following form
		\begin{equation}\label{symplectic symplectic}
		D_{ij}=\begin{blockarray}{cccccccccc}
		e_{j1} & \cdots & \cdots & e_{jn_j} & f_{jn_j} & \cdots & \cdots & f_{j1}\\
		\begin{block}{[cccccccc]cc}
		a_{e_{i1}e_{j1}} & 0 & \cdots & 0 & a_{e_{i1}f_{jn_j}} & 0 & \cdots  & 0 &  e_{i1}  & \leftarrow \text{leading row}\\
		* & * & \cdots & * & * & * & \cdots & * &   \\
		\vdots & \vdots &  & \vdots & \vdots & \vdots &  & \vdots & \vdots\\
		* & * & \cdots & * & * & * & \cdots & * &   \\
		x_{e_{in_i}e_{j1}} & * & \cdots & * & x_{e_{in_i}f_{jn_j}} & * & \cdots & * & e_{in_i} &  \leftarrow \text{pivotal row}\\
		a_{f_{in_i}e_{j1}} & 0 & \cdots & 0 & a_{f_{in_i}f_{jn_j}} & 0 & \cdots  & 0 & f_{in_i} &  \leftarrow \text{leading row}\\
		* & * & \cdots & * & * & * & \cdots & * \\
		\vdots & \vdots &  & \vdots & \vdots & \vdots &  & \vdots &\vdots\\
		* & * & \cdots & * & * & * & \cdots & * &   \\
		x_{f_{i1}e_{j1}} & * & \cdots & * & x_{f_{i1}f_{jn_j}} & * & \cdots & * & f_{i1} &  \leftarrow \text{pivotal row}\\
		\end{block}
		\overbrace{\text{l. cl.}} & & & & \overbrace{\text{l. cl.}}
		\end{blockarray}
		\end{equation}
		\item If  $U_i$ is symplectic and $U_j$ is orthogonal, then the block $D_{ij}$ is of the form:
		\begin{equation}\label{symplectic orthogonal}
		D_{ij}=\begin{blockarray}{cccccc}
		e_{j1} & \cdots & \cdots & f_{j1} \\
		\begin{block}{[cccc]cc}
		a_{e_{i1}e_{j1}} & 0 & \cdots & 0 & e_{i1} &\leftarrow \text{leading row}\\
		* & * & \cdots & * & e_{i2} & \\
		\vdots & \vdots & \cdots & \vdots & \vdots & \\
		* & * & \cdots & * &  &\\
		x_{e_{in_i}e_{j1}} & * & \cdots & * & e_{in_i}& \leftarrow \text{pivotal row} \\
		a_{f_{in_i}e_{j1}} & 0 & \cdots & 0 & f_{in_i} &\leftarrow \text{leading row}\\
		* & * & \cdots & * &  & \\
		\vdots & \vdots & \cdots & \vdots & \vdots & \\
		* & * & \cdots & * & f_{i2} &\\
		x_{f_{i1}e_{j1}} & * & \cdots & * & f_{i1}& \leftarrow \text{pivotal row} \\
		\end{block}
		\overbrace{\text{l. clm.}}
		\end{blockarray}
		\end{equation}
		and if $U_i$ is orthogonal and $U_j$ is symplectic, then the block $D_{ij}$ is of the form:
		\begin{equation}\label{orthogonal symplectic}
		D_{ij}=\begin{blockarray}{cccccccccc}
		e_{j1} & \cdots & \cdots & e_{jn_j} & f_{jn_j} & \cdots & \cdots & f_{j1}\\
		\begin{block}{[cccccccc]cc}
		a_{e_{i1}e_{j1}} & 0 & \cdots & 0 & a_{e_{i1}f_{jn_j}} & 0 & \cdots  & 0 &  e_{i1}  & \leftarrow \text{leading row}\\
		* & * & \cdots & * & * & * & \cdots & * &   \\
		\vdots & \vdots &  & \vdots & \vdots & \vdots &  & \vdots & \vdots\\
		* & * & \cdots & * & * & * & \cdots & * &   \\
		x_{f_{i1}e_{j1}} & * & \cdots & * & x_{f_{i1}f_{jn_j}} & * & \cdots & * & f_{i1} &  \leftarrow \text{pivotal row}\\
		\end{block}
		\overbrace{\text{l. clm.}} & & & & \overbrace{\text{l. clm.}}
		\end{blockarray}
		\end{equation}
	\end{enumerate}
\end{Proposition}

\begin{proof} As pointed out earlier, the homogeneous system of equations induced by 
	the equality $UXU^{-1}-X=0$ is equal to the disjoint union of the homogeneous 
	system of equations induced by $U_iX_{ij}U_j^{-1}-X_{ij}=0$. So, one can consider 
	blocks individually. All cases are similar. We will just prove the last two cases. 
	Let $U_i=J_{2s}=diag(S_s,S_s^{-1})$ and $U_j=J_{2r,\epsilon}$. Recall that, for $s>0$, the matrix $S_s$ is defined as follows.
	\begin{equation}\nonumber
	S_s:=\begin{bmatrix}
	1 &  &  &     \\
	1 & 1 &  &   \\
	\vdots &  \vdots & \ddots &    \\
	1 & 1  & \cdots & 1 \\
	\end{bmatrix}
	\end{equation}
	The blocks $D_{ij}$ and $D_{ji}$ are 
	subject to the equations
	\begin{equation}\label{block entries in the centralizer}
	D_{ij}=U_iD_{ij}U_j^{-1},\qquad D_{ji}=U_jD_{ji}U_i^{-1}.
	\end{equation}
	Write the matrices $D_{ij}$ and $D_{ji}$ as block matrices as follows:
	\[
	D_{ij}=\begin{bmatrix}
	A_1 \\A_2
	\end{bmatrix} \qquad D_{ji}=\begin{bmatrix}
	B_1 & B_2
	\end{bmatrix}
	\]
	where $A_i$'s are $s\times 2r$ matrices and $B_i$'s are $2s\times r$ matrices. Using the fact that $U_i$ is a block diagonal matrix, one can write equation \eqref{block entries in the centralizer} as follows:
	\[
	\begin{bmatrix}
	A_{1} \\
	A_{2}
	\end{bmatrix}=\begin{bmatrix}
	S_{s} & 0 \\
	0 & S_s^{-1}
	\end{bmatrix}\begin{bmatrix}
	A_{1} \\
	A_{2}
	\end{bmatrix}J_{2r,\epsilon}^{-1}
	=\begin{bmatrix}
	S_sA_1J_{2r,\epsilon}^{-1} \\
	S_s^{-1}A_2J_{2r,\epsilon}^{-1}
	\end{bmatrix}\]
	and
	\[
	\begin{bmatrix}
	B_{1} &
	B_{2}
	\end{bmatrix}=J_{2r,\epsilon}\begin{bmatrix}
	B_{1} &
	B_{2}
	\end{bmatrix}\begin{bmatrix}
	S_{s} & 0 \\
	0 & S_s^{-1}
	\end{bmatrix}^{-1}
	=\begin{bmatrix}
	J_{2r,\epsilon}B_1S_s^{-1} &
	J_{2r,\epsilon}B_2S_s
	\end{bmatrix}.\]
	This means, $A_1,A_2,B_1,B_2$ are all fixed points of the unipotent action. As a result, the top rows of $A_1,A_2,B_1,B_2$ are zero except possibly for the first entries. The claim concerning the indices of the free variables follows from Lemma \ref{unique free index of a fixed point} and Lemma \ref{first row lemma}.
\end{proof}

\begin{Definition}
	The set of basis elements that corresponds to a leading row (resp. pivotal row) is called a \textbf{leading basis} (resp. \textbf{pivotal basis}) element. The set of leading (resp. pivotal) basis elements is denoted with $B_{lead}$ (resp. $B_{pivot}$). In detail:
	\[
	B_{lead}=\{e_{i1}:i=1,\cdots,k,k+1,\cdots r\}\bigcup \{f_{in_i}:i=1,\cdots,k\}\subset \mathcal{B}.
	\]
	and
	\[
	B_{pivot}=\{f_{i1}:i=1,\cdots,k,k+1,\cdots r\}\bigcup \{e_{in_i}:i=1,\cdots,k\}\subset \mathcal{B}.
	\]
\end{Definition}

Bearing in mind the block form of $U$ and using Remark \ref{Remark eigen vectors of unipotent blocks} we see that the subset $B_{lead}$ is a basis of the fixed subspace $^UV_m$, i.e. the fixed space of the map defined by
multiplication by $U$ on the right. Likewise, the subset $B_{pivot}$
is a basis of the fixed subspace $V^U_m$, i.e. the fixed space of the map defined by
multiplication by $U$ on the left, equivalently, the fixed space of the map
defined by multiplication by $U^t$ on the left. The subspace of $V_m$ generated by $B_{lead}\cup B_{pivot}$ is denoted by $F_U$.

\begin{Lemma}\label{F-u and B-p, B-l are hyperbolic conjugates} Keeping the notation $U=diag(U_1,\cdots,U_r)$, cf. Eq. \ref{diagonal form of U}, we have the following.
	\begin{enumerate}
		\item The subspaces $^UV$ and $V^U$ are generated by $B_{lead}$ and $B_{pivot}$. 
		\item The set hyperbolic conjugates of elements of $B_{pivot}$ is equal to $B_{lead}$ and the cardinality of both of these sets are equal to $d$, dimension of the fixed space of $U$.
		\item The subspaces $^UV$ and $V^U$ are totally isotropic.
		\item The subspace
		$F_U=V^U\oplus{^UV} $ is a non-degenerate symplectic space, and it splits in
		$V_m$:
		\[
		V_m=F_U\perp (F_U)^{\perp}
		\]
	\end{enumerate}  We will write $F_{U^{\perp}}$ in place of $(F_U)^{\perp}$. As a result, if $C\in V_m$ then $C=C^{F_U}+C^{F_{U^{\perp}}}$, where $C^{F_U}\in F_U$, $C^{F_{U^{\perp}}}\in F_{U^{\perp}}$ and $Q(C^{F_U},C^{F_{U^{\perp}}})=0$.
\end{Lemma}
\begin{proof}
	\begin{enumerate}
		\item The fact that the subspaces $^UV$ and $V^U$ are generated by $B_{lead}$ and $B_{pivot}$ is already discussed in the previous paragraph.
		\item This follows from the explicit determination of the blocks of a generic element $D$ in the centralizer of $U$, as given in Proposition \ref{shape of a generalized fixed point}.
		\item[3,4] Follows from 2.
	\end{enumerate}
	
\end{proof}

\begin{Remark}
	Notice that $|B_{lead}|=|B_{pivot}|=\dim V^U=\dim {^UV}$.
	We also observe that, the set of leading basis elements is equal to the set of basis elements that corresponds to the leading columns. From this we conclude that, an index $(u,v)$ is a free index if and only if $(u,v)\in B_{pivot}\times B_{lead}$.
\end{Remark}

\begin{Definition}
	\begin{enumerate}
		\item A $2m\times 2m$ matrix $D=(d_{uv})_{u,v\in\mathcal{B}}$ will be called a \textbf{primitive
			matrix} if $d_{uv}=x_{uv}$ for $(u,v)\in B_{pivot}\times B_{lead}$, and $d_{uv}\in \F_q$ for $(u,v)\notin B_{pivot}\times B_{lead}$. In particular, if $v\notin B_{lead}$ then the column $C_v(D)$ defines a unique element of $V_m$.
		\item A square matrix whose entries are indexed by $B_{pivot}\times B_{lead}$ will be called a \textbf{free-index} matrix.
		\item For a free-index matrix
		$A=(a_{uv})_{(u,v)\in B_{pivot}\times B_{lead}}$, substituting $a_{uv}$ for $x_{uv}$ defines an element $Mat_{2m\times 2m}(\F_q)$ which is denoted by $D(A)$. The matrix $D(A)$ is called a \textbf{realization} of $D$.
		\item The map given by the rule $M=(m_{uv})_{u,v\in \mathcal{B}}\longmapsto M_{pivot}:=(m_{uv})_{\substack{(u,v)\in B_{pivot}\times
				B_{lead}}}$ is denoted by $M\longmapsto M_{pivot}$. The submatrix $M_{pivot}$ of $M$ will be referred as the \textbf{pivotal submatrix} of
		$M$.
		\item The \textbf{leading submatrix} $M_{lead}$ of
		a matrix $M=(m_{uv})_{u,v\in \mathcal{B}}$ (which can be a primitive matrix as well) is defined as the matrix $M_{lead}=(m_{uv})_{u,v\in B_{lead}}$.
		If $M$ is a realization of $D$
		then $M_{lead}=D_{lead}$ and $D_{pivot}=(x_{uv})_{(u,v)\in B_{pivot}\times B_{lead}}$. Entries of
		$D_{lead}$ (or $M_{lead}$) will be referred as \textbf{leading entries} of $D$ (or $M$).
		\item The column $C_v$ of $M$ or $D$ will be called a \textbf{leading column} for $v\in B_l$.
		\item If $A=(a_{uv})_{u\in B_p,v\in B_l}$ is a free-indexed $d\times d$ matrix, then $\overline{A}=(\overline{a_{uv}})_{u,v \in \mathcal{B}}$ where $\overline{a_{uv}}=a_{uv}$ if $(u,v)\in B_{pivot}\times B_{lead}$ and $\overline{a_{uv}}=0$ if $(u,v)\notin B_{pivot}\times B_{lead}$.
		\item Let $u,v$ be two basis elements and $D$ be a primitive centralizer of $U$. We introduce the notation \[\text{$\phi_{uv}=(Q(C_u(D),C_v(D)))_U$ and $\omega_{uv}=(Q(C_u(D),C_v(D)))_{U^{\perp}}$}\] where $(Q(C_u(D),C_v(D)))_U$ is an element of the symplectic space $F_U=\langle B_{pivot}\rangle \oplus \langle B_{lead}\rangle$ and  $(Q(C_u(D),C_v(D)))_{U^{\perp}}$ is an element of the orthogonal complement $(F_U)^{\perp}$ of $F_U$.
	\end{enumerate}
\end{Definition}

\begin{Remark}\label{A-overline eigen vector}
	Let $A$ be a free-index-matrix and consider $\overline{A}$. Then by definition of free indices and Lemma \ref{F-u and B-p, B-l are hyperbolic conjugates} it follows that the columns of $\overline{A}$ are eigen-vectors of $U$ and rows of $\overline{A}$ are eigen-vectors of $U^{t}$.
\end{Remark}

\begin{Lemma}
	Let $D$ be a primitive matrix with respect to $U$. If $A$ is a free-index-matrix such that $D(A)$ is in the centralizer of $U$ then $D(B)$ is in the centralizer of $U$ for all free-index-matrix $B$.
\end{Lemma}
\begin{proof}
	This follows directly from the definition of a free index. That is, the entries $m_{uv}$ of $M=D(A)$ do not occur in the equations $UMU^{-1}-M=0$ for $(u,v)\in B_{pivot}\times B_{lead}$.
\end{proof}

A primitive matrix $D$ is called a \textbf{primitive centralizer} of $U$ if a realization $D(A)$ (hence all realizations) of $D$ commutes with $U$.

\begin{Lemma}\label{leading row zeros}
	Let $D$ be a primitive centralizer of $U$, $u\in B_{lead}$ be a leading basis element and $R_u$ be the row of $D$ corresponding to $u$. Then all the entries of $R_u$ is zero except the leading entries $d_{uv}$, i.e. $d_{uv}=0$ for $v\notin B_{lead}$. In short, if $u\in B_{lead}$ and $v\notin B_{lead}$ then $d_{uv}=0$.
\end{Lemma}
\begin{proof}
	This is a reformulation of Lemma \ref{first row lemma}.
\end{proof}

\begin{Example}\label{centralizer shape example}
	Consider the block diagonal matrix $U$ whose diagonal entries are $J_6$ and $J_{4,\epsilon}$ with $\epsilon\neq 0$ and let $D$ be a primitive centralizer of $U$. Write $
	D=\begin{bmatrix}
	D_{11} & D_{12} \\
	D_{21} & D_{22}
	\end{bmatrix}
	$
	where $D_{11}$ is a $6\times 6$ matrix. Then $UXU^{-1}=X$ implies
	
	\begin{eqnarray}
	\begin{bmatrix}
	D_{11} & D_{12} \\
	D_{21} & D_{22}
	\end{bmatrix}& = & \begin{bmatrix}
	J_6 & 0 \\
	0 & J_{4,\epsilon}
	\end{bmatrix}
	\begin{bmatrix}
	D_{11} & D_{12} \\
	D_{21} & D_{22}
	\end{bmatrix}
	\begin{bmatrix}
	J_6^{-1} & 0 \\
	0 & J_{4,\epsilon}^{-1}
	\end{bmatrix}\nonumber
	=  \begin{bmatrix}
	J_6D_{11}J_6^{-1} & J_6D_{12}J_{4,\epsilon}^{-1} \\
	J_{4,\epsilon}^{-1}D_{21}J_6^{-1} & J_{4,\epsilon}D_{22}J_{4,\epsilon}^{-1}
	\end{bmatrix}
	\end{eqnarray}
	By the Proposition \ref{shape of a generalized fixed point} it follows that $D$ is of the following type:
	\[\footnotesize{
		D=\begin{blockarray}{cccccccccccc}
		e_{11} & e_{12} & e_{13} & f_{13} & f_{12} & f_{11} & e_{21} & e_{22} & f_{22} & f_{21}\\
		\begin{block}{[cccccccccc]cc} \pmb{a_{11}} & 0 & 0 & \pmb{a_{12}} & 0 & 0 & \pmb{a_{13}} & 0 & 0 & 0 & e_{11} & \leftarrow \text{leading basis} \\
		d_{21} & d_{22} & d_{23} & d_{24} & d_{25} & d_{26} & d_{27} & d_{28} & d_{29} & d_{2,10}  & e_{12} \\
		\pmb{x_{11}} & d_{32} & d_{33} & \pmb{x_{12}} & d_{43} & d_{53} & \pmb{x_{13}} & d_{38} & d_{39} & d_{3,10} & e_{13} & \leftarrow \text{pivotal basis} \\
		\pmb{a_{21}} & 0& 0& \pmb{a_{22}} & 0 & 0& \pmb{a_{23}}&  0 &0 & 0  & f_{13} & \leftarrow \text{leading basis} \\
		d_{51} & d_{52} & d_{53} & d_{54} & d_{55} & d_{56} & d_{57} & d_{58} & d_{59} & d_{5,10} & f_{12} \\
		\pmb{x_{21}} &  d_{62} & d_{63} & \pmb{x_{22}} & d_{65} & d_{66}& \pmb{x_{23}} & d_{68} & d_{69} & d_{6,10} & f_{11} & \leftarrow \text{pivotal basis} \\
		\pmb{a_{31}} & 0 & 0 & \pmb{a_{32}} & 0 & 0 & \pmb{a_{33}} & 0 & 0 & 0 & e_{21} & \leftarrow \text{leading basis} \\
		d_{81} & d_{82} & d_{83} & d_{84} & d_{85} & d_{86} & d_{87} & d_{88} & d_{89} & d_{8,10}  & e_{22}  \\
		d_{91} & d_{92} & d_{93} & d_{94} & d_{95} & d_{96} & d_{97} & d_{98} & d_{99} & d_{9,10}  & f_{22} \\
		\pmb{x_{31}} & d_{10,2} & d_{10,3} & \pmb{x_{32}} & d_{10,5} & d_{10, 6} & \pmb{x_{33}} & d_{10,8} & d_{10,9} & d_{10,10}  & f_{21} & \leftarrow \text{pivotal basis} \\
		\end{block}
		\overbrace{\text{leading column}} & & & \overbrace{\text{l. cl.}} & & & \overbrace{\text{l. cl.}}\\
		\end{blockarray}}
	\]
	where, for each choice of $x_{ij}$, the resulting matrix commutes with $U$. Clearly, the set of pivotal basis elements is $B_{pivot}=\{e_{13},f_{11}, f_{21}\}$, and the set of leading basis elements is $B_{lead}=\{e_{11},f_{13},e_{21}\}$.
	Consider the vectors $C_{f_{13}}(D)$ and $C_{e_{22}}(D)$.  Then we have the following equalities:
	\begin{eqnarray}
	C_{f_{13}}^F & = & a_{12}e_{11}+x_{12}e_{13}+a_{22}f_{13}+x_{22}f_{11}+a_{32}e_{21}+x_{32}f_{21}\nonumber \\
	C_{f_{13}}^{F^{\perp}} & = & d_{24}e_{12}+d_{54}f_{12}+d_{84}e_{22}+d_{94}f_{22} \nonumber
	\end{eqnarray}
	Likewise we have the following equalities:
	\begin{eqnarray}
	C_{e_{22}}^F & = & 0e_{11}+d_{38}e_{13}+0f_{13}+d_{68}f_{11}+a_{33}e_{21}+d_{10,8}f_{21} \nonumber \\
	C_{e_{22}}^{F^{\perp}} & = & d_{28}e_{12}+d_{58}f_{12}+d_{88}e_{22}+d_{98}f_{22} \nonumber
	\end{eqnarray}
	This means
	\begin{eqnarray}\nonumber
	Q(C_{f_{13}},C_{e_{22}}) & = & Q(C_{f_{13}}^F,C_{e_{22}}^F) + Q(C_{f_{13}}^{F^{\perp}},C_{e_{22}}^{F^{\perp}}) \in \F_q
	\end{eqnarray} as
	\begin{eqnarray}
	Q(C_{f_{13}}^F,C_{e_{22}}^F)  & = & a_{12}d_{68}+x_{12}0-a_{22}d_{38}-x_{22}0+a_{32}d_{10,8}-x_{32}0 = a_{12}d_{68}-a_{22}d_{38}+a_{32}d_{10,8})\in \F_q\nonumber
	\end{eqnarray}
	and
	\begin{eqnarray}
	Q(C_{f_{13}}^{F^{\perp}},C_{e_{22}}^{F^{\perp}})
	& = & (d_{24}d_{58}-d_{54}d_{28}+d_{84}d_{98}-d_{94}d_{88})\in \F_q .  \nonumber  
	\end{eqnarray}
	Consider the matrices $D_{pivot}$ and $D_{lead}$ along with the matrix $\sigma$ which is introduced as:
	\[
	\text{$D_{lead}=\begin{bmatrix}
		a_{11} & a_{12} & a_{13} \\
		a_{21} & a_{22} & a_{23} \\
		a_{31} & a_{32} & a_{33} \\
		\end{bmatrix},\quad \sigma=\begin{bmatrix}
		1 & 0 & 0 \\
		0 & 0 & -1 \\
		0 & 1 & 0
		\end{bmatrix},\quad$ and $\quad
		D_{pivot}=\begin{bmatrix}
		x_{11} & x_{12} & x_{13} \\
		x_{21} & x_{22} & x_{23} \\
		x_{31} & x_{32} & x_{33} \\
		\end{bmatrix}$.}\] where instead of labeling elements w.r.t the corresponding pivotal basis elements $e_{11},f_{13},e_{21}$; the usual labeling of entries are used. We observe that
	\begin{eqnarray}
	Q(C_{f_{13}},C_{e_{21}}) & = & \overbrace{a_{12}x_{23}+x_{12}a_{23}-a_{22}x_{13}-x_{22}a_{13}+a_{32}x_{33}-x_{32}a_{33}\nonumber}^{Q(C^{F}_{f_{13}},C^{F}_{e_{21}})}+{Q(C^{F^{\perp}}_{f_{13}},C^{F^{\perp}}_{e_{21}})} \nonumber\\
	& = & \underbrace{a_{12}x_{23}-a_{22}x_{13}+a_{32}x_{33}}_{(D^t_{lead}\sigma D_{pivot})_{23}}-\underbrace{(a_{13}x_{22}-a_{23}x_{12}+a_{33}x_{32})}_{(D^t_{lead}\sigma D_{pivot})_{32}}\nonumber+\omega_{f_{13}e_{21}}
	\end{eqnarray}
\end{Example}

Each realization $M$ of a primitive centralizer $D$ of $U$ is a true centralizer of $U$.
However, it is not always the case that $M\in Sp_m(q)$. Even existence of a realization
$M$ of $D$ which is an element of $Sp_m(q)$ is not guaranteed as the conditions for being an
isometry involves equations with the indeterminates $x_{uv}$. As a result, we introduce
the concept of \textbf{primitive symplectic centralizer} of $U$. First we make some
observations. In order to simplify the notation, we will use $B_l$ and $B_p$ instead of $B_{lead}$ and $B_{pivot}$ respectively.

\begin{Remark}\label{inner product of free parts of leading columns}
	Let $D$ be a primitive matrix and for $w\in \mathcal{B}$, denote
	the hyperbolic conjugate of $w$ with $w'$. For $u\in \mathcal{B}$, using Lemma \ref{F-u and B-p, B-l are hyperbolic conjugates}, we write
	$C_u(D)=C_u^F(D)+C_u^{F^{\perp}}(D)$, where the summands are orthogonal to each other. If $u\in B_l$ then
	\begin{eqnarray}
	C_u^F & = & \sum_{w\in B_l} d_{wu}\cdot w + \sum_{w\in B_p}x_{wu}\cdot w\nonumber\\
	& = & \sum_{w\in B_l} d_{wu}\cdot w + \sum_{w\in B_l}x_{w'u}\cdot w'\in \F_q[\overline{x_{ij}}]-\F_q
	\end{eqnarray}
	and
	\[ C_u^{F^{\perp}}=\sum_{w\in \mathcal{B}-(B_l\cup B_p)} d_{wu}\cdot w\in\F_q. \]
	If $u\in \mathcal{B}-B_l$ then
	\begin{eqnarray}
	C_u^F & = & \sum_{w\in B_l} d_{wu}\cdot w + \sum_{w\in B_p}d_{wu}\cdot w\nonumber\\ & = & \sum_{w\in B_l} d_{wu}\cdot w + \sum_{w\in B_l}d_{w'u}\cdot w'\in \F_q\nonumber
	\end{eqnarray}
	\[
	C_u^{F^{\perp}}=\sum_{w\in \mathcal{B}-(B_l\cup B_p)} d_{wu}\cdot w\in \F_q\]
	Now we will investigate several cases of inner-products.
	
	\textbf{Case 1: $u,v\in B_l$.} In this case, the inner product
	$Q(C_u^{F},C^F_v)$ can be written as:
	\[
	Q(C_u^{F},C^F_v)=Q(\sum_{w\in B_l} d_{wu}\cdot w + \sum_{w\in B_l}x_{w'u}\cdot w', \sum_{w\in B_l} d_{wv}\cdot w + \sum_{w\in B_l}x_{w'v}\cdot w')
	\]
	and as $B_p$ consists of hyperbolic conjugates of the elements of $B_l$, using the last equation we get
	\begin{equation}\label{explicit inner product of two leading columns}
	Q(C_u(D),C_v(D))=\underbrace{\sum_{w\in B_l}\delta_w d_{wu}\cdot x_{w'v} + \sum_{w\in B_l}\delta_{w'} x_{w'u}\cdot d_{wv}}_{Q(C_u^{F}(D),C^F_v(D))=\phi_{uv}}+\underbrace{Q(C_u^{F^{\perp}}(D),C_v^{F^{\perp}}(D))}_{\omega_{uv}\in \F_q}
	\end{equation} where the $\delta_w=Q(w,w')=\pm 1$. Clearly, $Q(C_u^{F^{\perp}},C^{F^{\perp}}_v)\in \F_q$. 
	
	\textbf{Case 2: $u,v \in \mc{B}-B_l$.} In this case, $C_u(D)$ and $C_v(D)$ defines an element of $V_m$ and hence $Q(C_u(D),C_v(D)\in \F_q$. 
	
	\textbf{Case 3: $u\in B_l, v\notin B_l$.} In this case we have
	\begin{equation}\label{explicit inner product of a leading and non-leading}
	Q(C_u(D),C_v(D))=\sum_{w\in B_l}\delta_w d_{wu} d_{w'v} + \underbrace{\sum_{w\in B_l}\delta_{w'} x_{w'u}\underbrace{d_{wv}}_{=0}}_{=0} +Q(C_u^{F^{\perp}}(D),C_v^{F^{\perp}}(D))\in \F_q.
	\end{equation}
	Notice that only the second summand contains indeterminates. However, since $w\in
	B_l$ and $v\notin B_l$ by Lemma \ref{leading row zeros} we get $d_{wv}=0$, hence
	the summand involving the indeterminates vanishes and thus, in this last case, the inner
	product is a scalar.  Recall that we write $\phi_{uv}$ to indicate the inner product
	$Q(C_u^{F},C^F_v)$ and $\omega_{uv}$ to indicate the inner product
	$Q(C_u^{F^{\perp}},C^{F^{\perp}}_v)$. We also introduce the matrices
	\begin{eqnarray}\label{Phi, Omega definitions}
	\Phi_D & = &(\phi_{uv})_{u,v\in
		B_l},\nonumber \\
	\Omega_D & = & (\omega_{uv})_{u,v\in B_l}.
	\end{eqnarray}
\end{Remark}

Let $u,v\in \mathcal{B}$ and $C_u(D)$, $C_v(D)$ be two columns of a primitive centralizer $D$
of $U$. We want to consider the equality
\[
Q(C_u(D),C_v(D))=Q(u,v).
\]

\textbf{Case 1: $u,v\in \mathcal{B}-B_l$.} By Remark \ref{inner product of free parts of leading columns} it follows that $Q(C_u(D),C_v(D))\in \F_q$.
As a result, the above equality can be checked directly.

\textbf{Case 2: $u\in B_l$, $v\in \mathcal{B}-B_l$.} Then the inner product $Q(C_u(D),C_v(D))$ is given by Eq.\eqref{explicit inner product of a leading and non-leading} above. This means, the inner product $Q(C_u(D),C_v(D))$ does not involve indeterminates and the above equation can be checked directly.

Observe that, these equalities hold for $D$ if and only
if they hold for one (hence for any) realizations of $D$. As a result we obtain the
following:

\begin{Lemma}\label{lemma for defintion of primitive symplectic centralizer}
	Let $D$ be a primitive centralizer of $U$ and $M$ be a realization of $D$. If $M\in Sp_m(q)$ then the following hold:
	\begin{enumerate}
		\item $Q(C_u(D),C_v(D))=Q(u,v)$ for all $(u,v)\in \mathcal{B}\times \mathcal{B}-B_l\times B_l$.
		\item $D_{lead}=M_{lead}$ is invertible.
	\end{enumerate}
\end{Lemma}
\begin{proof}
	The first assertion is already dealt prior to the lemma. By the Lemma \ref{leading
		row zeros}, the leading rows of $M$ and $M_{lead}$, when considered as vectors,
	define the same elements in $V_m$. Hence, a non-trivial linear relation between the
	rows of $M_{lead}$ yields a non-trivial linear relation between the rows of $M$. As
	$M$ is invertible, this can not be the case.
\end{proof}

In the light of the lemma, we say that a primitive centralizer $D$ of $U$ is a
\textbf{primitive symplectic centralizer} of $U$ if $D$ satisfies the conditions 1. and
2. of Lemma \ref{lemma for defintion of primitive symplectic centralizer}. By definition, for a fixed primitive symplectic centralizer $D$
of $U$ and its realization $M$ of $D$, it follows that $M$ is an element of $Sp_m(q)$ if
and only $Q(C_u(M),C_v(M))=Q(u,v)=0$ for $u,v\in B_l$ as elements of $B_l$ are
orthogonal to each other by Lemma \ref{F-u and B-p, B-l are hyperbolic conjugates}.
Using the matrices $\Phi$ and $\Omega$ introduced in \eqref{Phi, Omega definitions}, this observation can be rephrased as follows:


\begin{Lemma}\label{realization is symplectic iff pivotal columns are perpendicular}
	Let $M$ be a realization of a primitive symplectic centralizer $D$ of $U$. Then $M\in Sp_m(q)$ if and only if
	\[
	\Phi_M=-\Omega_M.
	\]
\end{Lemma}
\begin{proof}
	Follows from the fact that $Q(C_u(M),C_v(M))=\phi_{uv}(M)+\omega_{uv}(M)$ for $u,v\in B_l$.
\end{proof}

\begin{Proposition}\label{inner product of pivotal columns in terms of matrices}
	There exists an invertible matrix $\sigma$ such that
	\[
	Q(C_u^{F},C_v^F)=(D_{lead}^t\cdot\sigma\cdot M_{pivot})_{uv}-(D_{lead}^t\cdot\sigma\cdot M_{pivot})_{vu}
	\]
	for all $u,v\in B_l$. In particular,
	$\Phi_M=(D_{lead}^t\cdot\sigma\cdot M_{pivot})-(D_{lead}^t\cdot\sigma\cdot M_{pivot})^t.$
\end{Proposition}

We need two lemmas:

\begin{Lemma}
	Let $(V,Q)$ be a symplectic space with a hyperbolic basis $B=\{e_1,f_1\cdots,e_{k},f_k\}$ and let $v_1,\cdots,v_k$ be arbitrary elements of $V$, written as column vectors:
	
	\begin{equation}
	v_1= \begin{bmatrix}
	v_{1e_1}\\ v_{1f_1} \\ \vdots \\ v_{1,e_k} \\ v_{1,f_k}
	\end{bmatrix},\quad
	v_2=\begin{bmatrix}
	v_{2e_1}\\ v_{2f_1} \\ \vdots \\ v_{2,e_k} \\ v_{2,f_k}
	\end{bmatrix},\quad
	\cdots, \quad
	v_k=\begin{bmatrix}
	v_{ke_1}\\ v_{kf_1} \\ \vdots \\ v_{k,e_k} \\ v_{k,f_k}
	\end{bmatrix}.
	\end{equation} Let $v^e_i$ (resp. $v^f_i$) be the $k$-tuple vector obtained from $v_i$ by keeping tuples  indexed by the basis vectors $P_1=\{e_1,\cdots,e_k\}$ (resp. $P_2=\{f_1,\cdots,f_k\}$) for $i=1,\cdots,k$ and removing the
	other entries. Let $T_1$ and $T_2$ be the set of $k\times k$ matrices whose $i$-th column
	is $v^e_i$ and $v^f_i$ respectively. Then
	\begin{equation}
	Q(v_i,v_j)=(T_1^{t} T_2)_{ij}-(T_1^{t} T_2)_{ji}
	\end{equation}
\end{Lemma}
\begin{proof}
	This follows from direct calculation. The $i$-th row of $T^{t}_1$ is $(v_{ie_1},\cdots,v_{ie_k})$ and the $j$-th column of $T_2$ is $(v_{jf_1},\cdots,v_{jf_k})^{t}$ and hence the right hand side of the above equation is
	\begin{equation}\label{inner product technical lemma}
	(v_{ie_1},\cdots,v_{ie_k})\cdot \begin{bmatrix}
	v_{jf_1} \\ \vdots \\ v_{j,f_k}
	\end{bmatrix}-(v_{je_1},\cdots,v_{je_k})\cdot \begin{bmatrix}
	v_{if_1}\\ \vdots \\ v_{i,f_k}
	\end{bmatrix}
	\end{equation}
	which is clearly equal to the inner product
	\[
	Q(v_i,v_j)=Q(v_{ie_1}e_1+v_{if_1}f_1+\cdots+v_{ie_k}e_k+v_{if_k}f_k\:\:,\:\:v_{je_1}e_1+v_{jf_1}f_1+\cdots+v_{je_k}e_k+v_{jf_k}f_k).
	\]
\end{proof}

Next we assume that $P_1,P_2$ is an
arbitrary partition of $B$ so that none of the hyperbolic pairs
$e_j,f_j$ fall into the same $P_i$. Observe that the partition
above satisfies this property. We call such a partition
isotropic. Finally, a square matrix $\sigma$ is called a \textbf{signed
	permutation} matrix if each row and each column has only one
non-zero entry which is either $1$ or $-1$.

\begin{Corollary}\label{Corollary for inner product in terms of matrices}
	Let $(V,Q)$, $B$ be an arbitrary hyperbolic basis in an arbitrary order, and $v_1,\cdots,v_k$ be as above. Let $P_1,P_2$ is an isotropic partition of $B$ and $T_1$, $T_2$ be defined in the manner described in the previous lemma. Then, there is a $k\times k$ signed permutation matrix $\sigma$ such that
	\begin{equation}
	Q(v_i,v_j)=(T_1^{t}\sigma T_2)_{ij}-(T_1^{t}\sigma T_2)_{ji}
	\end{equation}
\end{Corollary}
\begin{proof}
	Multiplication with a permutation matrix on the left acts on the rows the of matrix.
	Let $g$ be the permutation of $P_2$ so that the $i$-th element of $P_1$ and $g\cdot P_2$ form hyperbolic pairs and let $\sigma_1$ be the corresponding permutation matrix. Let $\sigma_2$ be the
	diagonal matrix with entries $\pm 1$ where $(ij)-th$ entry is $-1$ if and only if
	the $i$-th element of $P_1$ is the negative part of the hyperbolic pair that is
	contained. Now take $\sigma=\sigma_2\sigma_1$.
\end{proof}

\begin{proof}(of \ref{inner product of pivotal columns in terms of matrices})
	Take $V$ to be $F_U$, which is generated by $B_l\cup B_p$. Take $P_1$ to be the set of leading basis elements and $P_2$ to be the set of pivotal basis elements and apply the corollary.
\end{proof}

Recall that if $A=(a_{uv})_{u\in B_p,v\in B_l}$ is a free-indexed $d\times d$ matrix, then $\overline{A}=(\overline{a}_{uv})_{u,v \in \mathcal{B}}$ was defined as by the rule $\overline{a}_{uv}=a_{uv}$ if $(u,v)\in B_p\times B_l$ and $\overline{a}_{uv}=0$ otherwise.

\begin{Proposition}\label{M symplectic if T satisfies}
	Let $D$ be a primitive symplectic centralizer and $M$ be a realization of $D$. Then the following are equivalent:
	\begin{enumerate}
		\item $M$ is a symplectic matrix.
		\item The free-indexed matrix $D^{t}_{lead}\cdot \sigma \cdot M_{pivot}$ satisfies the equation
		\begin{equation}\label{T satisfy the rho equation}
		D^{t}_{lead}\cdot \sigma \cdot M_{pivot}-(D^{t}_{lead}\cdot \sigma \cdot M_{pivot})^{t}=-\Omega.
		\end{equation}
		\item $D^{t}_{lead}\cdot \sigma \cdot M_{pivot}=S-\Omega/2$  where $S$ is a free-indexed symmetric matrix and $\Omega/2=(\omega_{uv}/2)_{u,v\in B_l}$.
		\item There exists a symmetric matrix $S$ such that\[
		M_{pivot}=(D^{tr}_{lead}\cdot \sigma)^{-1}\cdot S- (D^{tr}_{lead}\cdot \sigma)^{-1}\cdot \Omega/2.
		\]
	\end{enumerate}
	As a result, for each primitive symplectic centralizer $D$, there exists a realization $M$ of $D$ which is an isometry. In fact, there exists $q^{\frac{d^2+d}{2}}$ many symplectic realizations of $D$ and they are of the form
	\[
	M+\overline{(D^{tr}_{lead}\cdot \sigma)^{-1}\cdot S}
	\]
	where $S$ is an $d\times d$ symmetric matrix.
\end{Proposition}

\begin{proof} Write $T$ in place of $D^{t}_{lead}\cdot \sigma \cdot M_{pivot}$. From Lemma \ref{realization is symplectic iff pivotal columns are perpendicular} it
	follows that $M$ is an isometry if and only if $\Phi=-\Omega$. Hence the equivalence of (1) and (2) follows from Proposition
	\ref{inner product of pivotal columns in terms of matrices} which states that
	$\Phi=T-T^t$. Assuming (2) and taking $S=(T+T^t)/2$ yields (3). Conversely, assume that $T=S-\Omega/2$ with symmetric $S$. This implies $T^t=S+\Omega/2$ as $\Omega$ is an anti-symmetric matrix. As a result, $T-T^t=-\Omega$, which is the statement of (2). The equivalence of (3) and (4) follows from the fact that $D_{lead}$ and $\sigma$ are invertible matrices.
\end{proof}

\subsection{Growth of centralizers}

We keep our assumptions on $U$, $\pmb{\eta}$ and $V_m$ and consider
$V_m\subset V_n$. The hyperbolic basis for $V_{m,n}$ is denoted by
$\mathcal{B}_{m,n}=\{e_{1},f_{1},\cdots,e_{n-m},f_{n-m}\}$. Thus, the union of
the hyperbolic bases $B_i$, $i=1,\cdots,r$ is equal to $\mathcal{B}_m$ and $\mathcal{B}_{m,n}\cup \mathcal{B}_m=\mathcal{B}_n$ is a
hyperbolic basis of $V_n$. As before,
rows and columns of the matrices in $GL_{2n}(q)$ are indexed by the basis
$\mathcal{B}$. If $u\in \mathcal{B}$ and $M\in GL_{2n}(q)$ then $C_u(M)$ denotes
column of $M$ which corresponds to basis element $u$. Finally, recall that $B_l$
generates $^UV$ and $B_p$ generates $V^U$ and these bases form hyperbolic
conjugates of each other. Next consider $U^{\uparrow\uparrow n}=U\perp
I_{2(n-m)}\in Sp_{n}(q)$. An element $M\in GL_{2n}(q)$ will be considered as a
block matrix of the form $\begin{bmatrix}
M_{11} & M_{12} \\ M_{21} & M_{22}
\end{bmatrix}$, where $M_{11}$ is an $2m\times 2m$ matrix.

We recall Theorem \ref{centralizer growth in gl} in this context.
\begin{Proposition}\cite[Proposition 2.5]{WW18} \label{centralizer growth in gl 2n}
	The centralizer $C_{GL_n(q)}(U^{\uparrow\uparrow n})$ of $U^{\uparrow\uparrow n}\in GL_{2n}(\F_q)$ is given by
	\footnotesize{
		\begin{equation}
		C_{GL_{2n}(q)}(U^{\uparrow\uparrow n})=\Big\{\begin{bmatrix}
		M_{11} & M_{12} \\ M_{21} & M_{22}
		\end{bmatrix}\Big| M_{11}\in C_{GL_{2m}(q)}(U),\, M_{22}\in GL_{2(n-m)}(q),\, UM_{12}=M_{12} ,\, M_{21}U=M_{21}\Big\}
		\end{equation}
	}
\end{Proposition}

The columns of $M_{12}$ and rows of $M_{21}$ are indexed by the elements of
$\mathcal{B}_{m,n}$. Moreover, the columns of $M_{12}$ (resp. rows of $M_{21}$) are
elements of $V^U$ (resp. $^UV$). By Lemma \ref{F-u and B-p, B-l are hyperbolic conjugates}, it follows that, for $v\in \mathcal{B}_{n,m}$, the  $v$-th
column $C_v(M_{12})$ (resp. row $R_v(M_{21})$) of $M_{12}$ (resp. $M_{21}$) are of the
form
\[
C_v(M_{12})=\sum_{w\in \mathcal{B}_m} m_{wv}\cdot w= \sum_{w\in B_p} m_{wv}\cdot w
\] and
\[R_v(M_{21})=\sum_{w\in \mathcal{B}_m} m_{vw}\cdot w=\sum_{w\in B_l} m_{vw}\cdot w
\] respectively,
as $V^U$ is generated by $B_p$ and $^UV$ is generated by $B_l$. From these equations we get the following:
\begin{Lemma}\label{rows of M-21 columns of M-12}
	Columns of $M_{12}$ are orthogonal to each other. Moreover, $R_u(M_{12})=0$ if $u$ is not a pivotal basis element and $C_v(M_{21})=0$ if $v$ is not a leading basis element.
\end{Lemma}
\begin{proof}
	Let $v_1,v_2\in \mathcal{B}_{m,n}$. The inner product of $C_{v_1}(M_{12})$ and $C_{v_2}(M_{12})$ is the sum of products of the form $\delta_w m_{wv_1}m_{w'v_2}$ where $w\in \mathcal{B}_m$ and $w'$ is the hyperbolic conjugate of $w$. So, one of the factor must be zero, as $w\in B_p$ implies $w'\notin B_p$, and thus $m_{w'v_2}=0$.
\end{proof}

We will call an $n\times n$
matrix $D=\begin{bmatrix}
D_{11} & M_{12} \\
M_{21} & M_{22}
\end{bmatrix}$ a \textbf{primitive centralizer} of $U^{\uparrow\uparrow n}$ if $D_{11}$ is a primitive
centralizer of $U$, entries of $M_{12},M_{21}$ and
$M_{22}$ are in $\F_q$; and
\[
M_{22}\in GL_{2(n-m)}(\F_q),\quad UM_{12}=M_{12}, \quad M_{21}U=M_{21}.
\]

\begin{Example} Let us revisit the block diagonal matrix $U$ whose diagonal entries
	are $J_6$ and $J_{4,\epsilon}$ with $\epsilon\neq 0$ of Example \ref{shape of a generalized fixed point}. We consider generic
	fixed points of $U^{\uparrow\uparrow 7}$. By Lemma \ref{rows of M-21 columns of M-12}, $M_{21}$ and $M_{12}$ are of the form
	\begin{equation}
	M_{21}=\begin{blockarray}{ccccccccccc}
	e_{11} & e_{12} & e_{13} & f_{13} & f_{12} & f_{11} & e_{21} & e_{22} & f_{22} & f_{21}\\
	\begin{block}{[cccccccccc]c}
	* & 0 & 0 & * & 0 & 0 & * & 0 & 0 & 0 & \leftarrow e_1\\
	* & 0 & 0 & * & 0 & 0 & * & 0 & 0 & 0 & \leftarrow f_1 \\
	* & 0 & 0 & * & 0 & 0 & * & 0 & 0 & 0 & \leftarrow e_2 \\
	* & 0 & 0 & * & 0 & 0 & * & 0 & 0 & 0 & \leftarrow f_2 \\
	\end{block}
	\text{l. c.} & & & l.c & & & l.c.
	\end{blockarray}\end{equation}
	\begin{equation}
	M_{12}=\begin{blockarray}{cccccc}
	e_1 & f_1 & e_2 & f_2\\
	\begin{block}{[cccc]cc}
	0 & 0 & 0 & 0 & e_{11}  &  \text{l. r.} \\
	0 & 0 & 0 & 0 & e_{12} \\
	\pmb{z_{11}} & \pmb{z_{12}} & \pmb{z_{13}} & \pmb{z_{14}} & e_{13} &  \text{p. r.} \\
	0 & 0 & 0 & 0 &f_{13} & \text{l. r.} \\
	0 & 0 & 0 & 0 & f_{12} \\
	\pmb{z_{11}} & \pmb{z_{12}} & \pmb{z_{13}} & \pmb{z_{14}} &f_{11} &  \text{p. r.} \\
	0 & 0 & 0 & 0 &e_{21} &  \text{l. r.} \\
	0 & 0 & 0 & 0 & e_{22}  \\
	0 & 0 & 0 & 0 & f_{22} \\
	\pmb{z_{11}} & \pmb{z_{12}} & \pmb{z_{13}} & \pmb{z_{14}} &f_{21} &  \text{p. r.} \\
	\end{block}
	\end{blockarray}
	\end{equation}
	and $D_{11}$ is a primitive centralizer of $U$, i.e. $D_{11}$ is of the following form
	\[\footnotesize{
		D_{11}=\begin{blockarray}{cccccccccccc}
		e_{11} & e_{12} & e_{13} & f_{13} & f_{12} & f_{11} & e_{21} & e_{22} & f_{22} & f_{21}\\
		\begin{block}{[cccccccccc]cc} \pmb{a_{11}} & 0 & 0 & \pmb{a_{12}} & 0 & 0 & \pmb{a_{13}} & 0 & 0 & 0 & e_{11} & \leftarrow \text{leading b.} \\
		d_{21} & d_{22} & d_{23} & d_{24} & d_{25} & d_{26} & d_{27} & d_{28} & d_{29} & d_{2,10}  & e_{12} \\
		\pmb{x_{11}} & d_{32} & d_{33} & \pmb{x_{12}} & d_{43} & d_{53} & \pmb{x_{13}} & d_{38} & d_{39} & d_{3,10} & e_{13} & \leftarrow \text{pivotal b.} \\
		\pmb{a_{21}} & 0& 0& \pmb{a_{22}} & 0 & 0& \pmb{a_{23}}&  0 &0 & 0  & f_{13} & \leftarrow \text{leading b.} \\
		d_{51} & d_{52} & d_{53} & d_{54} & d_{55} & d_{56} & d_{57} & d_{58} & d_{59} & d_{5,10} & f_{12} \\
		\pmb{x_{21}} &  d_{62} & d_{63} & \pmb{x_{22}} & d_{65} & d_{66}& \pmb{x_{23}} & d_{68} & d_{69} & d_{6,10} & f_{11} & \leftarrow \text{pivotal b.} \\
		\pmb{a_{31}} & 0 & 0 & \pmb{a_{32}} & 0 & 0 & \pmb{a_{33}} & 0 & 0 & 0 & e_{21} & \leftarrow \text{leading b.} \\
		d_{81} & d_{82} & d_{83} & d_{84} & d_{85} & d_{86} & d_{87} & d_{88} & d_{89} & d_{8,10}  & e_{22}  \\
		d_{91} & d_{92} & d_{93} & d_{94} & d_{95} & d_{96} & d_{97} & d_{98} & d_{99} & d_{9,10}  & f_{22} \\
		\pmb{x_{31}} & d_{10,2} & d_{10,3} & \pmb{x_{32}} & d_{10,5} & d_{10, 6} & \pmb{x_{33}} & d_{10,8} & d_{10,9} & d_{10,10}  & f_{21} & \leftarrow \text{pivotal b.} \\
		\end{block}
		\overbrace{\text{leading column}} & & & \overbrace{\text{l. cl.}} & & & \overbrace{\text{l. cl.}}\\
		\end{blockarray}
	}\] Finally, $M_{22}$ is an arbitrary invertible $4\times 4$ matrix.
\end{Example}

In order to determine the true definition of \textbf{primitive symplectic centralizer}
of $U^{\uparrow\uparrow n}$ we will investigate the equation $Q(C_u(D),C_v(D))=Q(u,v)$
with $u,v\in \mathcal{B}$ for a fixed primitive centralizer $D$ of $U^{\uparrow\uparrow n}$ and a
realization $M$ of $D$.
Since $V_n=V_m\perp V_{m,n}$, each column vector $C_u$ of $D$ (or $M$) admits a sum $C_u(M_{12})+C_u(M_{22})$ where $C_u(M_{12})\in V_m$ and $C_u(M_{22})\in V_{m,n}$. As a consequence $Q(C_u,C_v)=Q(C_u(M_{12}),C_v(M_{12}))+Q(C_u(M_{22}),C_v(M_{22}))$. Recall that $V_m$ also admits the orthogonal decomposition $F_U\perp F_U^{\perp}$, c.f. Lemma \ref{F-u and B-p, B-l are hyperbolic conjugates}.

\textbf{Case 1: $u,v\in \mathcal{B}_{m,n}$.}

\begin{Lemma}\label{M is an isometry implies M-22 is an isometry}
	$Q(C_u(D),C_v(D))=Q(u,v)$ for all $u,v\in \mathcal{B}_{m,n}$ if and only if $M_{22}\in Sp_{n-m}(q)$. In particular, if $M\in Sp_{n}(q)$ then $M_{22}\in Sp_{n-m}(q)$.
\end{Lemma}
\begin{proof}
	As discussed above, a column $C_u(D)$ for $u\in \mathcal{B}_{m,n}$ is equal to $C_{u}(M_{12})+C_u(M_{22})$ and the summands are orthogonal to each other. So by Lemma \ref{rows of M-21 columns of M-12} it follows that $Q(C_u(D),C_v(D))=Q(C_u(M_{22}),C_v(M_{22}))$. This proves the assertion.
\end{proof}

\textbf{Case 2: $u\in \mathcal{B}_m-B_{l},\, v\in\mathcal{B}_{m,n}$.} In this case, as $u\perp v$,
the equation under discussion becomes
\[Q(C_u(M_{11}),C_v(M_{12}))+Q(C_u(M_{21}),C_v(M_{22}))=0.\]
Since $u$ is not leading, by Lemma \ref{rows of M-21 columns of M-12}, the column $C_u(M_{21})$ is the zero vector. As a consequence, the second
inner-product vanishes automatically. So, consider $C_u(M_{11})=\sum_{w\in \mathcal{B}_m}m_{wu}\cdot w$ and $C_v(M_{12})=\sum_{w\in \mathcal{B}_m}m_{wv}\cdot w$. The inner product of these elements is given by
\begin{equation}
Q(C_u(M_{11}),C_v(M_{12}))=\sum_{w\in \mathcal{B}_m} \delta_w m_{wu}m_{w'v}
\end{equation}where $w'$ is the hyperbolic conjugate of $w$ and $\delta_w$ is equal to $Q(w,w')$. But by Lemma \ref{rows of M-21 columns of M-12}, $m_{w'v}=0$ if $w'\notin B_p$, hence the above sum becomes
\[
\sum_{w\in B_l} \delta_w m_{wu}m_{w'v}.
\]
as the factor $m_{wu}=0$ if $w\in B_l$  and $u\notin B_l$. Hence, the above summation vanishes.
This proves the following:

\begin{Lemma}
	For $u\in \mathcal{B}_m-B_l$ and $v\in \mathcal{B}_{m,n}$ the equality below holds. \[Q(C_u(M),C_v(M)=0.\]
\end{Lemma}

\textbf{Case 3: $u\in B_{l},\, v\in\mathcal{B}_{m,n}$.}
The equation under discussion is again
\[Q(C_u(M_{11}),C_v(M_{12}))+Q(C_u(M_{21}),C_v(M_{22}))=0.\]
Let $M_{12,pivot}$ be the $d\times 2(n-m)$ matrix obtained by the rows of $M_{12}$ that
correspond to the pivotal basis elements in $B_p$, i.e. keeping the possible non-zero
entries. So, the rows of $M_{12,pivot}$ are indexed by $B_p$ and columns are indexed by
$\mathcal{B}_{m,n}$. Observe that the vectors induced by the columns of $M_{12}$ and
$M_{12,pivot}$ are the same, as the removed entries are all zero. As discussed in the proof
of Proposition \ref{inner product of pivotal columns in terms of matrices}, the first inner
product $Q(C_u(M_{11}),C_v(M_{12}))$ is equal to the product of the $u$-th row of
$(M_{11})_{lead}^t\cdot \sigma$ with $C_v(M_{12})=C_v(M_{12,pivot})=\sum_{u\in B_p} m_{uv}\cdot u$. Thus, fixing $v$ and letting $u$ ranges over
$B_l$ and writing $C_v(M_{12})$ as a $d\times 1$ column vector, the above equation can be written as a matrix product:

\begin{equation}\label{M-12 is uniquely determined equation}
(M_{11})^{t}_{lead}\cdot \sigma \cdot C_{v}(M_{12,pivot})=\begin{bmatrix}
Q(C_{u_1}(M_{21}),(C_{v}(M_{22})) \\ Q(C_{u_2}(M_{21}),(C_{v}(M_{22})) \\ \vdots \\ Q(C_{u_h}(M_{21}),(C_{v}(M_{22}))  
\end{bmatrix}
\end{equation}
where $u_1,\cdots,u_d\in B_l$. Since $(M_{11})_{lead}$ and $\sigma$ are invertible matrices,
it follows that $C_v(M_{12,pivot})$, and hence $C_v(M_{12})$, is uniquely determined by
$M_{lead}$, $M_{21}$ and $M_{22}$. The only non-zero entries of $M_{12}$ correspond
to pivotal basis elements and thus we denote the matrix obtained by the entries of
$M_{12}$ that are not contained in a leading row by $(M_{12})_{lead}$, which is an
$h\times k$ matrix. Likewise, we denote the matrix obtained by removing the columns of
$M_{21}$ that do not correspond to a pivotal row is denoted by $(M_{21})_{pivot}$. With
these notations we get the following.

\begin{Lemma}\label{Lemma M-12 is uniquely determined equation}
	$C_u(M)\perp C_v(M)$ for all $u\in B_l$ and for all $v\in \mathcal{B}_{m,m}$ if and only if
	\[(M_{11}^t)_{lead}\cdot \sigma\cdot M_{12,pivot}=Q(C_u(M_{21}),C_v(M_{22}))_{u\in B_p,v\\ \in\mathcal{B}_{m,n}}.\]
\end{Lemma}
\begin{proof}
	Notice also that the right hand side of \eqref{M-12 is uniquely determined equation} is uniquely determined by $C_v'$, as $M^{t}_{lead}\cdot \sigma$ is invertible.
\end{proof}

\textbf{Case 4: $u,v\in B_{l}$.} As before, the equations under discussion becomes
\[Q(C_u(M_{11}),C_v(M_{12}))+Q(C_u(M_{21}),C_v(M_{22}))=0,\quad u,v\in B_l\]
since the leading basis elements are orthogonal to each other.

\begin{Lemma}\label{M in sp implies D-11 primitive symplectic}
	If $M\in Sp_n(q)$ then $D_{11}$ is a primitive symplectic centralizer of $U$.
\end{Lemma}
\begin{proof}
	Let $u,v\in \mathcal{B}_{m}$ and assume
	that $u$ is not leading. Writing $C_u(D)=C_u(D_{11})+C_u(M_{21})$,
	$C_v(D)=C_v(D_{11})+C_v(M_{21})$ and using the fact that the summands are
	orthogonal to each other along with the fact that $C_u(M_{21})=0$, it follows that
	\begin{eqnarray}
	Q(u,v) & = & Q(C_u(M),C_v(M))=Q(C_u(D),C_v(D))\nonumber\\ & = & Q(C_u(D_{11}),C_v(D_{11})\nonumber
	\end{eqnarray}
	$ \forall u\in B_l$, $v\in \mathcal{B}_m$, as $M$ is in an isometry. As a result, $D_{11}$ is a primitive symplectic centralizer.
\end{proof}

With these observations, the following definition makes sense.

\begin{Definition}
	A primitive centralizer $D=\begin{bmatrix}
	D_{11} & M_{12} \\ M_{21} & M_{22}
	\end{bmatrix}$ of $U^{\uparrow\uparrow n}$ is called a \textbf{primitive symplectic centralizer} of $U^{\uparrow\uparrow n}$ if $D_{11}$ is a primitive symplectic centralizer of $U$, $M_{22}\in Sp_{n-m}(q)$ and $M_{12}$ satisfy the equation in Lemma \ref{Lemma M-12 is uniquely determined equation}.
\end{Definition}
Let $D$ be a primitive symplectic centralizer of $U^{\uparrow\uparrow n}$ and $M$ be a realization of $D$. Notice that $M$ is automatically contained in the centralizer of $U$.

\begin{Lemma}
	$M\in Sp_n(q)$ if and only if
	\[
	Q(C_u(M),C_v(M))=0
	\] for all $u,v\in B_l$.
\end{Lemma}
\begin{proof}
	According to the discussion prior to the definition of primitive symplectic centralizer of $U^{\uparrow\uparrow n}$, we have $ Q(C_u(M),C_v(M))=Q(u,v)$ for all $u,v\in \mathcal{B}\times\mathcal{B}-B_l\times B_l$.
\end{proof}

As we have done in the previous section, we will write $C_u(D)$ as a sum of orthogonal vectors. $V_n$ is equal to the orthogonal sum $V_m\oplus V_{m,n}$ and $V_m$ is equal to the orthogonal sum of $F_{U}$ and $(F_U)^{\perp}$. So, each leading column vector $C_u(D)$ can be written as an orthogonal sum
\[
C_u(D)=C_u(D_{11})^F+C_u(D_{11})^{F^{\perp}}+C_u(M_{21})
\]
where $C_U(\cdot)^F$ and $C_u(\cdot)^{F^{\perp}}$ were defined in Lemma \ref{F-u and B-p, B-l are hyperbolic conjugates}. By the last lemma, $M\in Sp_n(q)$ if and only if
\[0=Q(C_u(D_{11})^F,C_v(D_{11})^F)  +Q(C_u(D_{11})^{F^{\perp}},C_v(D_{11})^{F^{\perp}})+Q(C_u(M_{21}),C_v(M_{21}))\]
or equivalently
\[
\Phi=-\Omega-Q(C_u(M_{21}),C_v(M_{21}))_{u,v\in B_l}.\]
The following lemma can be proved in the same way Proposition \ref{M symplectic if T satisfies} is proved.
\begin{Lemma}\label{M is symplectic final lemma}
	$M$ is a symplectic matrix if and only if there exists a symmetric matrix $S$ such that
	\begin{equation}\label{M is symplectic final lemma eq}
	(M_{11})_{pivot}=M_{pivot}=(D^{tr}_{lead}\cdot \sigma)^{-1}\cdot (S-\Omega/2- Q(C_u(M_{21}),C_v(M_{21}))_{u,v\in B_l})/2).
	\end{equation}
\end{Lemma}

Combining all, we get the following variant of Proposition \ref{centralizer growth in gl} which is proved in \cite{WW18}:

\begin{Proposition}\label{centralizer growth in sp}
	The centralizer of $U^{\uparrow\uparrow n}$ in $Sp_{s}(q)$ admits the following description:
	\footnotesize{
		\begin{eqnarray}
		C_{Sp_n(q)}(U^{\uparrow\uparrow n}) & = & \Big\{\begin{bmatrix}
		M_{11} & M_{12} \\ M_{21} & M_{22}
		\end{bmatrix}\in GL_{2n}(q)  \Big  | M_{22}\in Sp_{n-m}(q),UM_{12}=M_{12}, M_{21}U=M_{21},\nonumber\\
		& &    M_{12,lead}=((M_{11}^t)_{lead}\cdot \sigma)^{-1} Q(C_u(M_{21}),C_v(M_{22}))_{u\in B_l,v\in \mathcal{B}_{m,n}} \\
		& & M'_{11}:= M_{11}+\overline{((M_{11})^{t}_{lead}\cdot \sigma)^{-1}\cdot Q(C_u(M_{21}),C_v(M_{21}))_{u,v\in B_l}}/2\in C_{GL_{2m}(q)}(U)\cap Sp_m(q)\Big\}.\nonumber
		\end{eqnarray}
	}
	Equivalently
	\footnotesize{
		\begin{eqnarray}
		C_{Sp_n(q)}(U^{\uparrow\uparrow n}) & = & \Big\{\begin{bmatrix}
		M_{11} & M_{12} \\ M_{21} & M_{22}
		\end{bmatrix}\in C_{GL_{2n}}(U^{\uparrow 2n})  \Big |  M_{22}\in Sp_{n-m}(q), \nonumber\\
		& & M_{12,lead}=((M_{11}^t)_{lead}\cdot \sigma)^{-1} Q(M_{21},M_{22}) \\
		&  & M'_{11}:= M_{11}+\overline{((M_{11})^{t}_{lead}\cdot \sigma)^{-1}\cdot Q(C_u(M_{21}),C_v(M_{21}))_{u,v\in B_l}}/2\in C_{GL_{2m}(q)}(U)\cap Sp_m(q) \Big\}.\nonumber
		\end{eqnarray}
	}
	In particular, if $U\in Sp_m(q)$ is an arbitrary isometry whose modified symplectic type is $\pmb{\lambda}$ and $||\overline{\pmb{\lambda}}||=2m$, then
	\[
	|C_{Sp_n}(U^{\uparrow\uparrow n})|=|C_{Sp_m}(U)|\cdot|Sp_{n-m}(q)|\cdot q^{2d(n-m)}.
	\]
\end{Proposition}
\begin{proof} ($\subseteq$)
	Let $M\in C_{Sp_n(q)}(U^{\uparrow\uparrow n})$. By Lemma \ref{M is an isometry
		implies M-22 is an isometry} $M_{22}\in Sp_{n-m}(q)$. The equalities
	$UM_{12}=M_{12}$ and $M_{21}U=M_{21}$ follow from Proposition \ref{centralizer
		growth in gl}. The equality
	\[
	M_{12,lead}=((M_{11}^t)_{lead}\cdot \sigma)^{-1} Q(C_u(M_{21}),C_v(M_{22}))_{u\in B_l,v\in \mathcal{B}_{m,n}}
	\] follows from the Lemma
	\ref{Lemma M-12 is uniquely determined equation}.
	The only difference between $M_{11}$ and
	$M_{11}'$ occur in the free indices, hence $M_{11}'$ is also in the centralizer of $U$ by Proposition \ref{shape of a generalized fixed point}. By Lemma \ref{M is symplectic final lemma}, \[
	(M_{11})_{pivot}=M_{pivot}=(D^{tr}_{lead}\cdot \sigma)^{-1}\cdot (S-\Omega/2- Q(C_u(M_{21}),C_v(M_{21})))_{u,v\in B_l})/2.
	\] As a result $M'_{11}$ satisfies the 4th of Proposition
	\ref{M symplectic if T satisfies} hence $M'_{11}$ is an isometry and commutes with $U$.
	
	($\supseteq$) Let $M$
	be an element of the right handside. The last condition ensures that $M_{11}'$ is in the
	centralizer of $U$, and hence as above, $M_{11}$ is in the centralizer of $U$. The
	first three conditions now ensure that $M$ is in the centralizer of $U^{\uparrow\uparrow
		n}$. The fact that $M$ is an isometry is a consequence of the previous investigations.
	
	The second set equality follows from the first one, as the defining conditions of the second set implies that $M$ is an isometry, as dealt in the preceding discussion. Now consider equality concerning the cardinalities. First assume that $U$ is a unipotent element. Then the equality follows from the previous set equality as the $M_{12}$ is uniquely determined by $M_{11},M_{21}$ and $M_{22}$, and the number of possible $M_{21}$ matrices is $q^{2h(n-m)}$ as $h$ is the dimension of the $1$-eigenspace of $U$. For general $U$, the result follows from Remark \ref{schurs lemma and primary cyclic modules}.
\end{proof}

Now assume that $U_1,U_2\in Sp_{m}(q)$ where $\pmb{\lambda}$ and $\pmb{\mu}$ are their modified
symplectic types. Moreover, assume that $U_1U_2=U=J_{\overline{\pmb{\eta}}}$ and
$||\pmb{\eta}||=||\pmb{\lambda}||+||\pmb{\mu}||$.
\begin{Lemma}\label{centralier of the intersection general linear group} The following equality holds:
	\begin{eqnarray}
	C_{GL_n(q)}(U_1^{\uparrow\uparrow n})\cap C_{GL_n(q)}(U_2^{\uparrow\uparrow n})=\Big\{\begin{bmatrix}
	M_{11} & M_{12} \\ M_{21} & M_{22}
	\end{bmatrix} & \Big| & M_{11}\in C_{GL_m(q)}(U_1)\cap C_{GL_m(q)}(U_2),\nonumber\\ & & M_{22}\in GL_{2(n-m)}(\F_q),\nonumber\\ & & UM_{12}=M_{12}, M_{21}U=M_{21}\Big\}.
	\end{eqnarray}
\end{Lemma}
\begin{proof}
	For $i=1,2$, Proposition \ref{centralizer growth in gl} implies that $M\in C_{GL_{2n}(q)}(U_i^{\nonumber\uparrow n})$ if and only if the following hold:
	\begin{enumerate}
		\item $M_{11}\in C_{GL_{2m}(q)}(U_i)$
		\item Columns of $M_{12}$ consist of eigen-vectors of $U_i$,
		\item Columns of $M_{21}^t$ consist of eigen-vectors of the $U_i^t$.
	\end{enumerate}
	Let $V^{U_1},V^{U_2},V^{U}$ denote the fixed spaces of $U_1,U_2$ and $U$, respectively. By Lemma \ref{reflection length and residual dimension}/\ref{intersection equality for the top coefficients} we know that
	\[
	V^{U_1}\cap V^{U_2}=V^{U}, \text{ and}\quad V^{U^t_1}\cap V^{U^t_2}=V^{U^t}
	\]
	as reflection length of $U_i$ and $U_i^t$ are same. Now assume that $M$ is contained in the intersection. Then by 1., $M_{11}\in C_{GL_{2m}(q)}(U_1)\cap C_{GL_{2m}(q)}(U_2)$. Conversely, assume that $M$ is contained in the intersection. Then $M_{11}\in C_{GL_{2m}(q)}(U_1)\cap C_{GL_{2m}(q)}(U_2)$. As columns of $M_{12}$ (respectively rows of $M_{21}$) consists of elements of $V^{U}=V^{U_1}\cap V^{U_2}$ (respectively $V^{U^t}=V^{U^t_1}\cap V^{U^t_2}$) it follows that $M\in C_{GL_{2n}(q)}(U_1^{\uparrow\uparrow n})\cap C_{GL_{2n}(q)}(U_2^{\uparrow\uparrow n})$ by Lemma \ref{centralizer growth in gl}.
\end{proof}

\begin{Lemma}\label{dilation of intersection through free variables}
	Let $A\in C_{GL_{2m}(q)}(U_1)\cap C_{GL_{2m}(q)}(U_2)$ and $B\in GL_{2m}(q)$. Let $C=(c_{uv})_{u,v\in \mathcal{B}_m}=A-B$. Assume that $c_{uv}=0$ if $(u,v)\notin B_p\times B_l$. Then $B\in C_{GL_{2m}(q)}(U_1)\cap C_{GL_{2m}(q)}(U_2)$ as well.
\end{Lemma}
\begin{proof}
	All the entries of $C$ except $C_{pivot}$ is zero. We know from Remark \ref{A-overline eigen vector} that each column (resp. row) of $C$ is then a $1$-eigenvector of $U$ (resp. $U^t$). Invoking \ref{reflection length and residual dimension}/\ref{intersection equality for the top coefficients} we see that each column (resp. row) of $C$ is then a $1$-eigenvector of $U_1$ and $U_2$ (resp. $U_1^t$ and $U^t_2$). This means, $C$ is contained in $C_{GL_{2m}(q)}(U_1)\cap C_{GL_{2m}(q)}(U_2)$. Now the result follows from the fact that $A\in C_{GL_{2m}(q)}(U_1)\cap C_{GL_{2m}(q)}(U_2)$ and $B=A-C$.
\end{proof}

\begin{Proposition}\label{centralizer intersection growth in sp}
	Let $C_{\pmb{\mu}, \pmb{\lambda}}(n)$ denote the intersection $C_{Sp_n(q)}(U_1^{\uparrow\uparrow n})\cap C_{Sp_n(q)}(U_2^{\uparrow\uparrow n})$ for $n\geq m$. Then the set equality
	\footnotesize{\begin{eqnarray}
		C_{\pmb{\mu}, \pmb{\lambda}}(n) & = & \Big\{\begin{bmatrix}
		M_{11} & M_{12} \\ M_{21} & M_{22}
		\end{bmatrix}\in GL_{2n}(q)  \Big  | M_{22}\in Sp_{n-m}(q),   UM_{12}=M_{12}, M_{21}U=M_{21},\nonumber\\ &  & M_{12,lead}=((M_{11}^t)_{lead}\cdot \sigma)^{-1} Q(M_{21},M_{22}) \\
		& &  M'_{11}=M_{11}+\overline{((M_{11})^{t}_{lead}\cdot \sigma)^{-1}\cdot Q(C_u(M_{21}),C_v(M_{21}))_{u,v\in B_l}}/2 \in C_{\pmb{\mu}, \pmb{\lambda}}(m)\Big\}\nonumber
		\end{eqnarray}}
	holds for $n\geq m$. In particular, if $U,U_1,U_2\in Sp_m(q)$ are isometries and the modified symplectic type of $U$ is $\pmb{\lambda}$ with $||\overline{\pmb{\lambda}}||=2m$ and $U_1U_2=U$, then
	\[
	|C_{Sp_n(q)}(U_1^{\uparrow\uparrow n})\cap C_{Sp_n(q)}(U_2^{\uparrow\uparrow n})|=|C_{Sp_m(q)}(U_1)\cap C_{Sp_m(q)}(U_2)|\cdot|Sp_{n-m}(q)|\cdot q^{2h(n-m)}.
	\]
\end{Proposition}
\begin{proof}
	Let $M\in C_{Sp_n(q)(U_1)}\cap C_{Sp_n(q)(U_2)}$. Then $M\in C_{Sp_n(q)}(U^{\uparrow\uparrow n})$ as $U_1U_2=U$. So by Proposition \ref{centralizer growth in sp}, the assertions $ M_{22}\in Sp_{n-m}(q), UM_{12}=M_{12}, M_{21}U=M_{21}$, and $(M_{12})_{lead}=((M_{11}^t)_{lead}\cdot \sigma)^{-1} Q(M_{21},M_{22})$ follows immediately. By Lemma \ref{centralier of the intersection general linear group}, $M_{11}$ is an element of $C_{GL_{2m}(q)}(U_1)\cap C_{GL_{2m}(q)}(U_2)$ and by Lemma \ref{dilation of intersection through free variables}, $M_{11}'\in C_{GL_{2m}(q)}(U_1)\cap C_{GL_{2m}(q)}(U_2)$. As argued in Proposition \ref{centralizer growth in sp}, $M_{11}'$ is an isometry. The converse containment follows from direct calculation using the discussion concerning the sufficiency conditions for $M$ being an isometry.
\end{proof}

\bibliographystyle{plain}
\bibliography{Stability_of_the_Hecke_algebra_of_wreath_products}

\end{document}